\documentclass[11pt,a4paper,oneside]{amsart}

\usepackage{bbm}
\usepackage{bm}
\usepackage{amssymb,amsmath,epsfig,mathrsfs, enumerate}
\usepackage{graphicx}
\usepackage[normalem]{ulem}
\usepackage{fancyhdr}
\fancyheadoffset{0cm}
\usepackage[margin=3cm]{geometry}

\usepackage{color}
\usepackage{hyperref}
\hypersetup{
 colorlinks   = true,
 urlcolor     = blue,
 linkcolor    = blue,
 citecolor   = red ,
 bookmarksopen=true
}

\numberwithin{equation}{section}
\theoremstyle{plain}
\newtheorem{Prop}{Proposition}[section]
\newtheorem{Thm}[Prop]{Theorem}
\newtheorem*{Thm*}{Theorem}
\newtheorem{Lem}[Prop]{Lemma}
\newtheorem{Cor}[Prop]{Corollary}

\theoremstyle{definition}
\newtheorem{Def}[Prop]{Definition}

\theoremstyle{remark}
\newtheorem{Rem}[Prop]{Remark}

\def\vint_#1{\mathchoice%
          {\mathop{\kern 0.2em\vrule width 0.6em height 0.69678ex
depth -0.58065ex
                  \kern -0.8em \intop}\nolimits_{\kern -0.4em#1}}%
          {\mathop{\kern 0.1em\vrule width 0.5em height 0.69678ex
depth -0.60387ex
                  \kern -0.6em \intop}\nolimits_{#1}}%
          {\mathop{\kern 0.1em\vrule width 0.5em height 0.69678ex
              depth -0.60387ex
                  \kern -0.6em \intop}\nolimits_{#1}}%
          {\mathop{\kern 0.1em\vrule width 0.5em height 0.69678ex
depth -0.60387ex
                  \kern -0.6em \intop}\nolimits_{#1}}}
\def\vintslides_#1{\mathchoice%
          {\mathop{\kern 0.1em\vrule width 0.5em height 0.697ex depth -0.581ex
                  \kern -0.6em \intop}\nolimits_{\kern -0.4em#1}}%
          {\mathop{\kern 0.1em\vrule width 0.3em height 0.697ex depth -0.604ex
                  \kern -0.4em \intop}\nolimits_{#1}}%
          {\mathop{\kern 0.1em\vrule width 0.3em height 0.697ex depth -0.604ex
                  \kern -0.4em \intop}\nolimits_{#1}}%
          {\mathop{\kern 0.1em\vrule width 0.3em height 0.697ex depth -0.604ex
                  \kern -0.4em \intop}\nolimits_{#1}}}

\newcommand{\intav}{\vint}
\newcommand{\aveint}[2]{\mathchoice
          {\mathop{\kern 0.2em\vrule width 0.6em height 0.69678ex
depth -0.58065ex
                  \kern -0.8em \intop}\nolimits_{\kern -0.45em#1}^{#2}}%
          {\mathop{\kern 0.1em\vrule width 0.5em height 0.69678ex
depth -0.60387ex
                  \kern -0.6em \intop}\nolimits_{#1}^{#2}}%
          {\mathop{\kern 0.1em\vrule width 0.5em height 0.69678ex
depth -0.60387ex
                  \kern -0.6em \intop}\nolimits_{#1}^{#2}}%
          {\mathop{\kern 0.1em\vrule width 0.5em height 0.69678ex
depth -0.60387ex
                  \kern -0.6em \intop}\nolimits_{#1}^{#2}}}

\DeclareMathOperator{\spn}{span}
\DeclareMathOperator{\supp}{supp}

\DeclareMathOperator{\dv}{div}

\DeclareMathOperator{\osc}{osc}

\DeclareMathAlphabet{\mathsfit}{T1}{\sfdefault}{\mddefault}{\sldefault}
\SetMathAlphabet{\mathsfit}{bold}{T1}{\sfdefault}{\bfdefault}{\sldefault}

\newcommand{\set}[2]{\left\{#1 : #2\right\}}
\newcommand{\emp}{\emptyset}
\newcommand{\sub}{\subseteq}
\newcommand{\N}{\mathbb{N}}
\newcommand{\R}{\mathbb{R}}

\newcommand{\del}{\partial}

\newcommand{\eps}{\varepsilon}
\newcommand{\inv}[1]{{#1}^{-1}}
\newcommand{\dx}{\, dx}
\newcommand{\loc}{\text{\rm loc}}
\newcommand{\Om}{\Omega}
\newcommand{\inp}[2]{\big\langle #1,#2\big\rangle}
\newcommand{\vertiii}[1]{{\left\vert\kern-0.25ex\left\vert\kern-0.25ex\left\vert #1 
    \right\vert\kern-0.25ex\right\vert\kern-0.25ex\right\vert}}
\newcommand{\gr}{\nabla}
\newcommand{\wto}{\rightharpoonup}

\newcommand{\X}{\mathfrak{X}}
\newcommand{\Xu}{\X u}
\newcommand{\XX}{\X\X}
\newcommand{\dvh}{\dv_{\!_H}}
\newcommand{\F}{\textsc{F}}
\newcommand{\weight}{\F(|\X u|)}

\newcommand{\LL}{\mathcal{L}}

\newcommand{\tvr}{\textsl{v}}
\newcommand{\subsob}{\mathcal{S}}

\title[H\"ormander vector fields of step two]
{Regularity of quasi-linear equations with H\"ormander vector fields of step two}

\author[Giovanna Citti]{Giovanna Citti}
\address[G.\ Citti]{Department of Mathematics, 
University of Bologna, Piazza di Porta San Donato 5,
  40126 - Bologna (BO), Italy.}
\email{giovanna.citti@unibo.it}

\author[Shirsho Mukherjee]{Shirsho Mukherjee}
\address[S.\ Mukherjee]{Department of Mathematics, 
University of Bologna, Piazza di Porta San Donato 5,
  40126 - Bologna (BO), Italy.}
\email{shirsho.mukherjee2@unibo.it}
\date{\today}
\begin{document}

\begin{abstract}
If the smooth vector fields $X_1,\ldots,X_m$ and their commutators span the tangent space at every point in $\Om\sub \R^N$
for any fixed $m\leq N$, then we establish the full interior regularity theory of quasi-linear equations $\sum_{i=1}^m X_i^*A_i(X_1u, \ldots,X_mu)=  0$ with $p$-Laplacian type growth condition. In other words, we show that a weak solution of the equation
is locally $C^{1,\alpha}$. 
\end{abstract}

\maketitle
\setcounter{tocdepth}{1}
\phantomsection
\tableofcontents

\section{Introduction}\label{sec:Introduction}

The development of partial differential equations arising from families of noncommuting vector fields has been conspicuous since the fundamental work of  H\"ormander \cite{Hor}. Such equations, nowadays referred to as {\it sub-elliptic equations}, are not elliptic in general. Nevertheless, they have a remarkable structural similarity with elliptic equations. This has been evident in the theory that sub-elliptic linear equations entail, where the given non-commuting vector fields $X_1,\ldots,X_m$ on a domain $\Om\sub\R^N$ satisfy the, so-called, H\"ormander's rank condition 
\begin{equation}\label{eq:horcond}
\dim\big(\texttt{Lie}(X_1,\ldots,X_m)\big)=N, 
\end{equation}
at every point in $\Om$. It  was first used in \cite{Hor} to show that the linear operator $\LL u=\sum_{i=1}^m X_i^*X_i u$, called sub-elliptic Laplacian or {\it sub-Laplacian}, is hypoelliptic. Later, it was found that it also satisfies estimates of the form 
$$ \|u\|_{\subsob^{k+2,p}} \leq c(N,p) \big( \|\LL u\|_{\subsob^{k,p}} + \| u\|_{\subsob^{k,p}}\big), $$
where $\subsob^{k,p}$ are Sobolev spaces defined by the vector fields (see Section \ref{sec:Notations and Preliminaries} for details), which clearly shows the resemblance to the corresponding elliptic estimates. This has been unveiled at increasing levels of generality by Folland-Stein \cite{Folland-Stein}, Folland \cite{Folland} and Rothschild-Stein \cite{Roth-Stein}. Their methods also reflect the importance of the special case when the vector fields are left-invariant with respect to a Nilpotent Lie group. The H\"ormander's condition \eqref{eq:horcond} endows a stratification on the Lie algebra of such groups and they are called Carnot Groups (see \cite{Bonfig-Lanco-Ugu}). Later, it has been shown by Nagel-Stein-Wainger \cite{Nag-Stein-Wain} (see also S\'anchez-Calle \cite{Calle}) that in the general case, the Carnot-Carath\`eodory metric is equivalent to the natural metrics intrinsic to the given vector fields, that are generalizations of the homogeneous metrics of Carnot groups. 

These subsequently lead to the development of an extensive literature and nowadays, the theory of linear sub-elliptic equations is largely established. 
It naturally leads to the inquiry for a regularity theory of quasi-linear equations in the sub-elliptic setting, similar to that of quasi-linear elliptic equations. This area has remained relatively open over the years and the present paper is aimed towards this pursuit.

Let $\Om \subset \R^N$ be an open, simply connected bounded domain and $X_1,\ldots,X_m$ be smooth vector fields with $m\leq N$ such that they, along with their commutators, span the tangent space at every point in $\Om$. This is a special case of \eqref{eq:horcond}, referred to as the ``step two hypothesis" by Nagel-Stein \cite{Nag-Stein}, where commutators up to second-order are considered in place of the full Lie algebra generated by the vector fields (see \eqref{eq:step2}). 
We consider the equation 
\begin{equation}\label{eq:peq}
\dvh ( A(\X u)) =\sum_{i=1}^m X_i^*A_i(\X u)=  0 \qquad \text{in}\ \Om,
\end{equation}
where $\X u = (X_1u, \ldots, X_m u)$ is the sub-elliptic gradient, $X_i^*$ is the formal adjoint of $X_i$ and 
$A: \R^m\to \R^m$ is a given $C^1$-function that satisfies the following structure condition,
\begin{equation}\label{eq:pstr}
 \begin{aligned}
(\delta+|z|^2)^\frac{p-2}{2}|\xi|^2 \leq \,&\inp{DA(z)\,\xi}{\xi}\leq L(\delta+|z|^2)^\frac{p-2}{2}|\xi|^2;\\
 &|A(z)|\leq L\,|z|(\delta+|z|^2)^\frac{p-2}{2},
\end{aligned} 
\end{equation}
for every $z,\xi\in \R^m$, with $1<p<\infty, L\geq 1$ and $\delta\geq 0$. 
The equation \eqref{eq:peq} with \eqref{eq:pstr} is modelled on the degenerated sub-elliptic equation
\begin{equation}\label{eq:Hplap}
\dvh \big((\delta+|\X u|^2)^\frac{p-2}{2}\X u\big) = 0,
\end{equation}
which, for $\delta=0$, is called the sub-elliptic $p$-Laplacian equation $\dvh(|\X u|^{p-2}\X u)=0$.   

In the Euclidean setting, it is  known that weak solutions of the classical $p$-Laplacian equation $\dv(|\gr u|^{p-2}\gr u)=0$ are locally $C^{1,\alpha}$, as shown independently by DiBenedetto \cite{Dib}, 
Lewis \cite{Lewis} and Tolksdorf \cite{Tolk} using techniques descending from the seminal work of De Giorgi \cite{DeG}. Hence, for the sub-elliptic case, $C^{1,\alpha}$-regularity of weak solutions with respect to the metrics introduced in \cite{Nag-Stein-Wain} is likewise expected. The sub-elliptic Poincar\'e inequality was first proved by Jerison \cite{Jerison} which leads to the Sobolev embedding theorem which has been later shown in different 
levels of generality by Garofalo-Nhieu \cite{Gar-Nhi}, Franchi-Lu-Wheeden \cite{Fran-Lu-Wheed}, Hajlasz-Koskela \cite{Haj-Kos}, etc.
However, obtaining higher regularity of the gradient $\X u$ using the classical techniques is quite difficult even in the case of step two, due to the non-commutativity of the vector fields. 
For previous results and similar theories in this direction, we refer to \cite{Cap--reg,C-D-G,Cap-Garo,Foglein,Dom-Man--reg,Dom-Man--cordes,Marchi,Man-Min,Dom, Min-Z-Zhong, Ric, Cap-Citti-LD-Otta} and references therein. Recently, the problem has been resolved for the case of the Heisenberg Group (see \cite{Zhong, Muk-Zhong}), which is the most trivial example of a step two Carnot group where the vector space of commutators is one dimensional. The case of non-Nilpotent groups like $\mathbb{SU}(3)$ and compact semi-simple Lie groups has been considered in \cite{Dom-Man--su3}. 
In this paper, we prove the $C^{1,\alpha}$-regularity for the most general step two case. 

First, we show that the sub-elliptic gradient is locally bounded as our first main result. 
\begin{Thm}\label{thm:mainthm1}
 If $ u \in \subsob^{1,p}(\Om)$ is a weak solution of equation \eqref{eq:peq}
 equipped with the structure condition \eqref{eq:pstr} for any $1<p<\infty$ and $\delta\geq 0$, then 
 $\X u \in L^\infty_\loc(\Om, \R^{m})$. Moreover, for any $x_0\in \Om$ and metric ball $B_r=B_r(x_0) \subset \Om $, we have 
 \begin{equation}\label{eq:locbound}
  \sup_{B_{\sigma r}}\ |\X u|  \leq \frac{c}{(1-\sigma)^{Q/p}}\bigg(\intav_{B_r}\big(\delta+|\X u|^2\big)^\frac{p}{2}\dx \bigg)^\frac{1}{p}
 \end{equation}
 for any $0<\sigma<1$, where $c = c(N,p,L) > 0 $ and $Q=Q(x_0)= \liminf_{s\to 0^+} \frac{\log(|B_s(x_0)|)}{\log(s)}$. 
\end{Thm}
The following Corollary is an easy consequence of Theorem \ref{thm:mainthm1}. 
\begin{Cor}\label{cor:1lip}
 If $ u \in \subsob^{1,p}(\Om)$ is a weak solution of equation \eqref{eq:peq}
 equipped with the structure condition \eqref{eq:pstr} for any $1<p<\infty$ and $\delta\geq 0$, then 
 $ u \in C^{0,1}_\loc(\Om)$. Moreover, for any metric ball $B_r \subset \Om $, we have 
 \begin{equation}\label{eq:1lip}
  |u(x)-u(y)| \leq c\bigg(\intav_{B_r}\big(\delta+|\X u|^2\big)^\frac{p}{2}\dx \bigg)^\frac{1}{p} d_c(x,y)
 \end{equation}
 for any $x,y\in B_r$, where $c = c(N,p,L) > 0 $ and $d_c$ is the Carnot-Carath\`eodory metric. 
\end{Cor}
Second, we show that the sub-elliptic gradient is locally H\"older continuous with respect to the Carnot-Carath\`eodory metric. 
The following theorem is the second main result of this paper. 
\begin{Thm}\label{thm:mainthm2}
 If $ u \in \subsob^{1,p}(\Om)$ is a weak solution of equation \eqref{eq:peq}
 equipped with the structure condition \eqref{eq:pstr} for any $1<p<\infty$ and $\delta\geq 0$, then there exists a positive $\alpha=\alpha(N,p,L)\leq 1$ such that 
 $\X u \in C^{0,\alpha}_\loc(\Om, \R^{m})$. Moreover, for any concentric metric balls $B_r\subset B_{r_0} \subset \Om $, we have 
 \begin{equation}\label{eq:locboundholder}
  \max_{1\le i \le m} \osc_{B_r} X_i u
\le
c\Big(\frac{r}{r_0}\Big)^\alpha\bigg(\intav_{B_{r_0}}\big(\delta+|\X u|^2\big)^\frac{p}{2}\dx \bigg)^\frac{1}{p}
 \end{equation}
 for any $0<r<r_0$, where $c = c(N,p,L) > 0 $ is a constant. 
\end{Thm}
The following Corollary is an easy consequence of Theorem \ref{thm:mainthm2} and Schauder estimates \cite{Xu}. 
\begin{Cor}\label{cor:smooth}
 If $ u \in \subsob^{1,p}(\Om)$ is a weak solution of equation \eqref{eq:peq}
 equipped with the structure condition \eqref{eq:pstr} for any $1<p<\infty$ and $\delta>0$, moreover if $A$ is smooth, then 
 $ u \in C^\infty(\Om)$.
\end{Cor}

The significance of this paper lies mainly in the fact that it shows the left invariance of the vector fields with respect to a Lie group is superfluous for the purpose of regularity. As reflected by the results in this paper, all necessary techniques and estimates are independent of any Lie group structure on $\R^N$. Therefore, the commutation relations between the vector fields are quite arbitrary within the restriction of the step two hypothesis. This is much more general than that for the case of the Heisenberg Group or any other Carnot groups of step two. 

In order to get the full regularity, the general structure of the vector fields makes it necessary to obtain Caccioppoli type estimates of every possible type by testing the equations satisfied by $X_iu$ and $[X_i,X_j]u$ for all $i,j$ with every combination of the vector fields and their commutators. The interplay between the estimates leading to the proofs of the theorems is similar to the earlier works. The guiding principle remains that the behavior of homogeneous sub-elliptic equations \eqref{eq:peq} is similar to that of inhomogeneous elliptic equations with an integrable right-hand side (this role is played by the extra terms containing the commutators for the sub-elliptic case). Hence, the high integrability of the commutators (Corollary \ref{lem:est Tu} and similar others) allows the use of De Giorgi type techniques. The double truncation $\tvr=\min\big(\mu(r)/8 , \max(\mu(r)/4-X_lu,0)\big)$ with $\mu(r)= \max_{1\leq i\leq m}\sup_{B_r} |X_iu|$, has been used within the estimates particularly to deal with the singular case $1<p<2$, similarly as in \cite{Muk-Zhong}. This argument originates in the earlier works \cite{Tolk, Lieb--bound}, etc. 
The Riemannian approximation is used for the local Lipschitz continuity. 

The paper is organized as follows. All necessary notations and preliminaries are provided in Section \ref{sec:Notations and Preliminaries} and the apriori Caccioppoli type estimates are proved in Section \ref{sec:apest}. Then, the proof of Theorem \ref{thm:mainthm1} is provided in Section \ref{sec:lip}. In Section \ref{sec:hold}, further requisite estimates are obtained that lead to the proof of Theorem \ref{thm:mainthm2} in the end. 

Finally, we remark that the present regularity results seem to be the best possible ones within the framework of the existing techniques and methods. The estimates of this paper use properties that are specific to step two in subtle ways which, apparently do not have any immediate generalization for higher-order commutators. We believe the regularity theory involving vector fields satisfying the general H\"ormander's condition \eqref{eq:horcond} would require new fundamental ideas. 

\section{Notations and Preliminaries}\label{sec:Notations and Preliminaries} 
 In this section, we develop relevant notations and provide basic properties and previously known results 
that shall be used throughout this paper. 

Here onwards we assume the set of smooth vector fields $\{X_1,\ldots, X_m\}$
defined on $\Om\sub \R^N$ for some integer $m\leq N$, satisfy the step two hypothesis of Nagel-Stein \cite{Nag-Stein}; precisely, 
\begin{equation}\label{eq:step2}
X_1,\ldots, X_m\quad \text{and}\quad [X_j,X_k]\ \forall\ j,k\in\{1.\ldots,m\}\quad \text{span the tangent space at each}\ x_0\in\Om.
\end{equation}
The nontrivial case occurs if $m<N$ 
because the case of equality is identical to the familiar Euclidean setting. We choose an enumeration of the commutators as $\{T_1,\ldots,T_n\}$ such that $\{X_1,\ldots, X_m,T_1,\ldots,T_n\}$ span the whole tangent space, for some integer $0<n\leq m(m-1)/2$. 

For a function $u:\Om \to \R$, let us denote the gradients as $$\X u = (X_1u, \ldots, X_m u)\quad\text{and}\quad Tu= (T_1u,\ldots, T_n u).$$  The Hessian is denoted as $ \XX u = (X_jX_iu)_{i,j}$, similarly $\X T u = (X_jT_ku)_{k,j}$ and $ T \X  u = (T_k X_ju)_{j,k}$ which are respectively $m\times n$ and $n\times m$ matrices. The commutator matrices are denoted as $ [\X, \X] u = ([X_j, X_i]u)_{i,j}$ and similarly $[\X, T]$ and $[T,\X]$.
The {\it sub-elliptic divergence} is  defined as 
$$\dvh (F) = \sum_{j=1}^m X_j^* f_j,\qquad \forall\ F= (f_1,\ldots, f_m).$$ 
The Euclidean vector fields on $\R^m$ are denoted as $D_i$ and 
$D=(D_1,\ldots, D_m)$ is the Euclidean gradient; sometimes $D_i$ is also denoted as $\del_{x_i}$ or something analogous in terms of coordinates.

We note that the step two hypothesis \eqref{eq:step2} implies
\begin{equation}\label{eq:comm}
[X_i,X_j] = \sum_{k=1}^n \alpha^{(k)}_{i,j} T_k + \sum_{l=1}^m \beta^{(l)}_{i,j} X_l\quad\text{and}\quad 
[X_i, T_j] =  \sum_{l=1}^m \gamma^{(l)}_{i,j} X_l + \sum_{k=1}^n \delta^{(k)}_{i,j} T_k, 
\end{equation}
and with appropriate choice of scaling in the norm of vector fields, we can assume without loss of generality that
\begin{equation}\label{eq:normbounds}
\max_{i,j} \sum_{k,l}\big(\|\alpha^{(k)}_{i,j}\|_{L^\infty}+\|\beta^{(l)}_{i,j}\|_{L^\infty} +\|\gamma^{(l)}_{i,j}\|_{L^\infty} +\|\delta^{(k)}_{i,j}\|_{L^\infty}\big) \leq c(N), \quad \text{for some}\quad c(N)>0,
\end{equation}
and from smoothness, we can also conclude boundedness 
$\|\X\alpha^{(k)}_{i,j}\|_{L^\infty}+\|T\alpha^{(k)}_{i,j}\|_{L^\infty} \leq c(N)$ and similar bounds for the others. 
Therefore, some basic inequalities are in order. Using \eqref{eq:comm} twice successively along with \eqref{eq:normbounds} and the bounds, we have the following inequalities :
\begin{align}
\label{eq:XTX} (i)&\qquad |[X_i,T_k]X_ju| \leq c \big(|\XX u|+|T\X u| \big)
 \leq c \big(|\XX u|+|\X Tu| + |\X u| + |Tu|\big);\\
\label{eq:XXX} (ii)&\qquad \big|[[X_i,X_l],X_j]u\big| \leq c\big(|[T, \X] u|+|[\X ,\X] u| +|\X u| + |Tu| \big)\leq c \big( |\X u| + |Tu|\big),
\end{align}
 for every $i,j,l\in\{1,\ldots,m\}$ and $k\in\{1,\ldots,n\}$, where $c=c(N)>0$ is a constant.
Furthermore, 
we have the inequality 
\begin{equation}\label{eq:ineqT}
|Tu|\leq c|\XX u|,
\end{equation}
for some constant $c=c(N)>0$, since $T_k = [X_l,X_j]$ for some $l,j\in \{1,\ldots,m\}$. 
%

\subsection{Sub-elliptic metrics}\label{subsec:submetric}
Let us recall the \textit{Carnot-Carath\'eodory metric} (CC-metric) which is defined in terms of the length $\ell(\gamma)$ of curves $\gamma:[0,1]\to \Om$, as  
\begin{equation}\label{eq:cc metric}
d_c(x,y)= \ \inf\set{\ell(\gamma)}{\gamma \in \Gamma(x,y)},  
\end{equation}
where the class $\Gamma$ of absolutely continuous curves is given by
$$\Gamma(x,y) = \set{\gamma\in AC\big([0,1];\Om \big)}{\gamma(0) = x,\gamma(1) = y,
\ \gamma'(t)\in \spn \{X_1|_{\gamma(t)}, \ldots, X_m|_{\gamma(t)}\}}.$$  
Chow's connectivity theorem (see \cite{Chow}) gurantees $\Gamma(x,y) \neq \emp $ so that $d_c$ is finite for any $x,y\in \R^N$. However, to determine the volume of the metric balls, another metric is required that is intrinsic to the vector fields, as shown in \cite{Nag-Stein-Wain}. 
Let $\{Y_1,\ldots, Y_{m+n}\}=\{X_1,\ldots, X_m,T_1,\ldots,T_n\}$. We observe that we can associate an integer, called formal degree $d_j=\deg(Y_j)\in \{1,2\}$; indeed, 
we define that $\deg(Y_j)=1$ if $Y_j\in \spn\{X_1,\ldots, X_m\}$ and $\deg(Y_j)=2$ otherwise. Then, it is not difficult to see that \eqref{eq:comm} and \eqref{eq:step2} respectively imply: 
\begin{itemize}
\item[] (a)   For every $j,k$ we have $[Y_j,Y_k] \in \spn\{Y_l : d_l\leq d_j+d_k\}$; 
\item[] (b)   $Y_1,\ldots,Y_{m+n}$ span the whole tangent space.
\end{itemize}
Then, we can define following metric as in Nagel-Stein-Wainger \cite{Nag-Stein-Wain}. For any $r>0$, let $\Gamma(r)$ be the class of absolutely continuous curves $\gamma:[0,1]\to \Om$ satisfying the differential equation
$$ \gamma'(t) =\sum_j a_j(t) Y_j(\gamma(t)) \quad\text{with}\quad |a_j(t)|<r^{d_j},$$ 
then the metric is defined as 
\begin{equation}\label{eq:rhometric}
\rho(x,y)=  \inf\set{r>0}{\exists \gamma \in \Gamma(r), \gamma(0)=x,\gamma(1)=y}.   
\end{equation}
Proposition 1.1 in \cite{Nag-Stein-Wain}, imply that within compact subsets of $\Om$, we have 
\begin{equation}\label{eq:eeq}
c_1|x-y| \leq \rho(x,y) \leq c_2 |x-y|^{1/2} \qquad\forall\ x,y\in \Om\sub \R^N
\end{equation}
for some positive constants $c_1,c_2$. Let us denote the metric balls $B_r(x)=\set{y\in \R^N}{\rho(x,y)<r}$ and it has been shown in \cite{Nag-Stein-Wain} that they follow all standard properties. 
In general $m+n\geq N$ and it is possible to choose a local basis $Y_{i_1},\ldots, Y_{i_N}$ for the tangent space. For every such $I=(i_1,\ldots,i_N)$, letting $d(I)=d_{i_1}+\ldots + d_{i_N}$ and 
$$ \Lambda (x, r) = \sum_I |\lambda_I(x)| r^{d(I)},$$
where $\lambda_I(x)= \det(a_{j,k}(x))$ whenever $Y_{i_j}=\sum_{k}a_{j,k}\del_{x_k}$, one can estimate the volume of the balls. 
Thus, $\Lambda \neq 0$ whenever $Y_{i_1},\ldots, Y_{i_N}$ is a basis for the tangent space. Let
$|E|$ be the Lebesgue measure of $E$ for any $E\subset \R^N$. Then, 
from Theorem 1 in \cite{Nag-Stein-Wain}, there exists $c_2>c_1>0$ such that we have 
$$ c_1\Lambda (x, r) \leq |B_r(x)| \leq c_2 \Lambda (x, r)$$ 
in compact subsets of $\Om$. Therefore, let us denote 
\begin{equation}\label{eq:dim}
Q(x) = \liminf_{r\to 0^+}\, \frac{\log(|B_r(x)|)}{\log(r)}, 
\end{equation}
so that we have $c'_1 r^Q \leq |B_r(x)| \leq c'_2 r^Q$ in compact subsets of $\Om$ for some $c'_2>c'_1>0$ and $Q=Q(x)$. Since, $\Lambda$ is a polynomial in $r$ of fixed degree, hence $Q$ is lower semi-continuous. It is a deep result of Nagel-Stein-Wainger \cite[Theorem 4]{Nag-Stein-Wain}, that the metric $\rho$ is equivalent to the CC-metric $d_c$. Henceforth, in the rest of the paper, we shall denote CC-metric balls by $B_r$ without loss of generality and $Q$ as the local dimension.
\begin{Rem}[Step 2 Carnot Groups]\label{rem:carnot2}
For a special case, when the vector fields are left invariant with respect to a Nilpotent Lie Group on $\R^N$ along with the hypothesis \eqref{eq:step2}, then it is called a Carnot Group of step 2. In that case, if $X_1,\ldots, X_m$ are linearly independent, then we have $m+n=N$, the metric $\rho$ is equivalent to the metric induced by the homogeneous norm of the group, and $Q=m+2n$. Also, the commutation relations are more special than \eqref{eq:comm}, we have $[X_i, T_k]=0$ and $[X_i,X_j]=  \sum_{k=1}^n \alpha^{(k)}_{i,j} T_k$ where $\alpha^{(k)}_{i,j}$'s are precisely prescribed by the composition law of the group. Furthermore, if it is a free group then we have $n=m(m-1)/2$. 
\end{Rem}
The H\"older classes $C^{k, \alpha}$ (also denoted as $\Gamma^{k, \alpha}$) has been introduced by Folland and Stein, see \cite{Folland, Folland-Stein--book}. The functions in these classes are H\"older continuous with respect to the CC-metric.  
\begin{Def}[Folland-Stein classes]\label{eq:fsclass}
Given an open set $\Om \subset \R^N$, a function $u:\Om\to \R$ and $0<\alpha\leq 1$, we say 
$u\in C^{0, \alpha}(\Om)$ if there exists a constant $M>0$, such that 
$$|u(x) - u(y)| \leq M\, d_c(x,y)^{\alpha}, \qquad\forall\ x,y\in\Om.$$ 
It is a Banach space with the norm $\|u\|_{C^{0,\alpha}(\Om)}= \|u\|_{L^\infty(\Om)} + [u]_{C^{0,\alpha}(\Om)}$, where
the H\"older seminorm is defined as 
\begin{equation}\label{eq:holdsemi}
[u]_{C^{0,\alpha}(\Om)}= \underset{\underset{x \neq y}{x,y\in \Om}}{\sup} \frac{|u(x)-u(y)|}{d_c(x,y)^{\alpha}}.
\end{equation}
For any $k\in \mathbb N$, the spaces $C^{k, \alpha}(\Om)$ are defined inductively as follows: we say that $u \in C^{k, \alpha}(\Om)$ if $X_i u \in C^{k-1,\alpha}(\Om)$ for every $i\in\{1,\ldots, m\}$. 
\end{Def}
\begin{Def}[Sub-elliptic Sobolev spaces]\label{def:subesob}
Given an open set $\Om \subset \R^N$, a function $u:\Om\to \R$ and $1\leq p<\infty$, the sub-elliptic Sobolev spaces are defined as 
$$ \subsob^{1,p}(\Om)=\set{u\in L^p(\Om)}{X_ju\in L^p(\Om),\ \forall\, j\in\{1,\ldots, m\}},$$ 
which is a Banach space with the norm $\|u\|_{\subsob^{1,p}(\Om)} = \|u\|_{L^p(\Om)} + \|\X u\|_{L^p(\Om, \R^m)}$. The subspaces 
$\subsob^{1,p}_\loc (\Om)$ and $\subsob^{1,p}_0(\Om)$ are defined as usual. Moreoever, for any $k\in\N$, the space $\subsob^{k,p}(\Om)$ 
is also defined by standard extension. 
\end{Def}

For any metric ball $B_r=B_r(x_0)$, the following Poincar\'e inequality for $u\in \subsob^{1,p}(B_r)$ 
\begin{equation}\label{eq:poincare}
\Big(\intav_{B_r}|u-(u)_{B_r}|^p\dx\Big)^{1/p} 
\leq c\,r \Big(\intav_{B_r}|\X u|^p\dx\Big)^{1/p} 
\end{equation}
holds, see \cite{Jerison, Haj-Kos}. This leads to the following Sobolev inequality, 
\begin{equation}\label{eq:sob emb}
\left(\intav_{B_r}| v|^{\frac{Q q}{Q-q}}\, dx\right)^{\frac{Q-q}{Q q}}
\leq\, c\,r \left(\intav_{B_r}| \X  v|^q\, dx\right)^{\frac 1 q}, 
\end{equation} 
for all $v \in \subsob^{1,q}_0(B_r)$ and $1<q<Q$ with $Q=Q(x_0)$.
\subsection{Sub--elliptic equations}

If $u:\Om \to \R$ is a weak solution of 
\eqref{eq:peq}, then for every test function $ \phi \in C^\infty_0(\Om)$, we have 
 \begin{equation}\label{eq:weak soln}
  \int_\Om \inp{A(\X u)}{\X \phi}\dx= \sum_{i=1}^m \int_\Om A_i(\X u) X_i\phi \dx =  0. 
 \end{equation}
Furthermore, for all non-negative $ \phi \in C^\infty_0(\Om)$, if the integral above is positive (resp. negative) then $u$ is called a weak supersolution (resp. subsolution) of the equation \eqref{eq:peq}.

It is easy to see that \eqref{eq:pstr} implies the following ellipticity condition
\begin{equation}\label{eq:elliptic} 
\inp{A(z)}{z}\geq c(p)\,|z|^2(\delta+|z|^2)^\frac{p-2}{2} 
\end{equation} 
and also the following strong monotonicity inequality 
\begin{equation}\label{eq:monotone}	
\inp{A(z)-A(w)}{z-w} \geq \frac{1}{L'}(\delta+|z|^2 +|w|^2)^\frac{p-2}{2}|z-w|^2
\end{equation}
for all $z,w\in\R^m$, where $c(p)>0$ and $L'= L'(p,L)>0$. 

The inequalities \eqref{eq:elliptic} and \eqref{eq:monotone} can be used to prove the existence of a weak solution 
of equation \eqref{eq:peq}, by standard variational methods. In fact, given any $ u_0 \in \subsob^{1,p}(\Om)$, there exists a 
unique weak solution $ u \in \subsob^{1,p}(\Om)$ of the following Dirichlet problem 
\begin{equation}\label{eq:dirichlet prob}
 \begin{cases}
  \dvh (A(\X u))= \ 0\ \ \text{in}\ \Om;\\
 \ u - u_0\in \subsob^{1,p}_0(\Om).
 \end{cases}
\end{equation}
The uniqueness follows from the comparison principle, i.e. if $u$ and $v$ respectively are weak super and subsolution 
of the equation \eqref{eq:peq} and 
$u \geq v$ on $\del\Om$ in the trace sense, then we have 
$ u \geq v$ a.e. in $\Om$. This can be shown easily by appropriate test fuctions  on the equation \eqref{eq:dirichlet prob}.

\begin{Rem}[Divergence-free vector fields]
Given a vector field $X$, its formal adjoint with respect to the Euclidean volume form (Lebesgue measure) is defined as $X^*$, so that for any $v\in C^\infty_0(\Om)$, we have 
\begin{equation}\label{eq:int by parts}
\int_\Om v X^* u \dx = \int_\Om u Xv \dx,
\end{equation}
where $X^*u= -Xu -\dv(X) u$. For simplicity, we assume that the given vector fields $X_1,\ldots, X_m$ are divergence-free, i.e. $\dv(X_i)=0$. Hence, we have that $T_1,\ldots, T_n$ are also divergence-free, from the formula 
$\dv([X_i,X_j])= X_i(\dv(X_j))-X_j(\dv(X_i))$. The general case is a perturbation by a lower order term and can be obtained with minor modification. 
Moreover, in particular, if the $X_i$'s are left-invariant with respect to a Lie group then they are always divergence-free with respect to the Haar measure of the group which is the Lebesgue measure up to multiples. In more general cases if the vector fields are defined on a manifold and the adjoint defined with respect to a volume form, then with appropriate geometric conditions (e.g. metrics with certain curvature conditions or manifolds with Killing forms, contact forms, etc.) divergence-free frames can exist and it is possible to write the equation \eqref{eq:peq} with respect to such frames having structure conditions similar to \eqref{eq:pstr}. We refer to \cite{Cap-Citti-LD-Otta} for details. 
\end{Rem}

Henceforth, we assume $X_1,\ldots, X_m$ are divergence-free, so that $X_i^*=-X_i$ and $T_k^*=-T_k$. 

\begin{Rem}[The case of $\delta=0$]\label{rem:del0}
If apriori estimates are available for the case of $\delta>0$ and if the constants are independent of $\delta$, then the case $\delta=0$ follows as well by a standard regularization technique. A brief outline is as follows.

Let $A$ satisfy the structure condition \eqref{eq:pstr} for the case $\delta =0$, i.e. for all $z,\xi\in \R^m$, 
\begin{equation}\label{eq:0pstr}
 \begin{aligned}
|z|^{p-2}|\xi|^2 \leq \,&\inp{DA(z)\,\xi}{\xi}\leq L|z|^{p-2}|\xi|^2;\\
 &|A(z)|\leq L\,|z|^{p-1}. 
\end{aligned} 
\end{equation}
There are many ways to define a regularization. In particular, for any $0<\delta<1$, one can define 
\begin{equation}\label{Adelta}
A_\delta (z) = \big(1-\eta_\delta(|z|)\big) A(z) + \eta_\delta (|z|) (\delta+|z|^2)^\frac{p-2}{2} z
\end{equation}
where $\eta_\delta\in C^{0,1}([0,\infty))$ can be chosen (see \cite{Muk0}) such that $A_\delta \to A$ uniformly on compact subsets of $\R^m$ as $\delta \to 0^+$ and the structure condition 
\begin{equation}\label{eq:str cond reg}
\begin{aligned}
\frac{1}{L'}(\delta+|z|^2)^\frac{p-2}{2}|\xi|^2 \leq \,&\inp{DA_\delta (z)\,\xi}{\xi}\leq L'(\delta+|z|^2)^\frac{p-2}{2}|\xi|^2;\\
 &|A_\delta (z)|\leq L'\,|z|(\delta+|z|^2)^\frac{p-2}{2},
\end{aligned}
\end{equation}
holds for some $ L'= L'(p,L)>1$. 
Given a weak solution $u\in \subsob^{1,p}(\Om)$ of \eqref{eq:peq}, 
we consider $u_\delta$ as the weak solution of the following regularized equation 
\begin{equation}\label{eq:dirichlet prob reg}
 \begin{cases}
  \dvh (A_\delta(\X u_\delta))=  0\ \ \text{in}\ \Om';\\
  \ u_\delta - u\in \subsob^{1,p}_0(\Om'), 
 \end{cases}
\end{equation}
for any $\Om'\subset\subset \Om$. Now, for the case of $\delta>0$ whenever we obtain 
uniform estimates for $\X u_\delta$ (with constants independent of $\delta$),
taking limit 
$\delta \to 0^+$, one can show that 
\eqref{eq:locbound} and \eqref{eq:locboundholder} also hold for the case of $\delta=0$. 
\end{Rem}
Therefore, we shall assume $\delta>0$ without loss of generality, hereafter.

\section{Apriori Estimates}\label{sec:apest}
In this section, we consider the equation 
\begin{equation}\label{eq:maineq}
\dvh ( A(\X u)) =  0 \qquad \text{in}\ \Om,
\end{equation}
where $A: \R^m\to \R^m$ is a given $C^1$-function that satisfies the structure condition
\begin{equation}\label{eq:str}
 \begin{aligned}
 \F(|z|)|\xi|^2 \leq \,&\inp{DA(z)\,\xi}{\xi}\leq L\,\F(|z|)|\xi|^2;\\
 &|A(z)|\leq L\,|z|\F(|z|),
\end{aligned} 
\end{equation}
for every $z,\xi\in \R^m$, given $\F: [0,\infty) \to [0,\infty)$ and $L>1$. We assume $\F$ to be continuous in $(0,\infty)$ and there 
exists $\delta_0>0$ such that 
$\F(t)\geq \delta_0$ for all  $t>0$.

Here onwards, we denote $u\in \subsob^{1,2}_\loc(\Om)$ as a weak solution of the equation \eqref{eq:maineq}. 
Furthermore, in this section we assume apriori that 
%
\begin{equation}\label{eq:ap reg}
\X u \in \subsob^{1,2}_\loc(\Om,\R^{m}),\qquad 
  Tu  \in \subsob^{1,2}_\loc(\Om,\R^n)\cap L^{\infty}_\loc(\Om,\R^n).
\end{equation}
Thus, \eqref{eq:ap reg} implies $YX_ju, [X_i,Y]u \in L^2_\loc(\Om)$ for any $i,j$, where $Y=X_j$ or $Y=T_k$. 
We only require the regularity \eqref{eq:ap reg} for the apriori estimates 
and the assumption shall be removed later by an approximation argument. 
\begin{Lem}\label{lem:yeqw1}
Let $Y=X_j$ or $Y= T_k$. For every $ \varphi \in C^\infty_0(\Om)$, we have  
\begin{equation}\label{eq:yeqw1}
\sum_{i,j}\int_\Om D_jA_i(\X u)YX_ju X_i\varphi\dx = \sum_{i=1}^m  \int_\Om A_i(\X u) [X_i, Y]\varphi\dx. 
\end{equation}
\end{Lem}
\begin{proof}
Given any $ \varphi \in C^\infty_0(\Om)$, we choose $\phi= Y\varphi$ on \eqref{eq:weak soln} and use integral by parts, to obtain 
\begin{equation}\label{eq:w}
\begin{aligned}
 0=& \sum_{i=1}^m \int_\Om A_i(\X u) X_i Y\varphi \dx\\
=& \sum_{i=1}^{m}\int_\Om  A_i(\X u)YX_i\varphi\dx +\sum_{i=1}^m  \int_\Om A_i(\X u) [X_i, Y]\varphi\dx\\
=& -\sum_{i=1}^{m}\int_\Om  Y (A_i(\X u))X_i\varphi \dx+ \sum_{i=1}^m  \int_\Om A_i(\X u) [X_i, Y]\varphi\dx.\\
\end{aligned}
\end{equation}
The proof is finished by 
using chain rule $Y (A_i(\X u))= \sum_{j=1}^m D_jA_i(\X u)YX_ju$ on \eqref{eq:w}.  
\end{proof}
Using $YX_j= X_jY -[X_j,Y]$ on \eqref{eq:yeqw1}, we have the following variant; 
\begin{equation}\label{eq:yeqw2}
\begin{aligned}
\sum_{i,j}\int_\Om D_jA_i(\X u)X_jYu X_i\varphi\dx &= \sum_{i,j}\int_\Om D_jA_i(\X u)[X_j,Y]u X_i\varphi\dx \\
&\quad + \sum_{i=1}^m  \int_\Om A_i(\X u) [X_i, Y]\varphi\dx
\end{aligned}
\end{equation}
for every $ \varphi \in C^\infty_0(\Om)$. 
Thus for $Y=X_l$, we have  
\begin{equation}\label{eq:yeqwXl}
\begin{aligned}
\sum_{i,j}\int_\Om D_jA_i(\X u)X_jX_lu X_i\varphi\dx &= \sum_{i,j}\int_\Om D_jA_i(\X u)[X_j,X_l]u X_i\varphi\dx \\
&\quad + \sum_{i}  \int_\Om A_i(\X u) [X_i, X_l]\varphi\dx,
\end{aligned}
\end{equation} 
and for $Y= T_k$, using integral by parts on the last term, we have 
\begin{equation}\label{eq:Teq}
\begin{aligned}
\sum_{i,j}\int_\Om D_jA_i(\X u)X_jT_ku X_i\varphi\dx &= \sum_{i,j}\int_\Om D_jA_i(\X u)[X_j,T_k]u X_i\varphi\dx \\
&\quad - \sum_{i,j}\int_\Om  \varphi D_jA_i(\X u) [X_i, T_k]X_j u\dx.
\end{aligned}
\end{equation}



\subsection{Caccioppoli type inequalities}\label{subsec:Caccioppoli type inequalities} 
Here we provide Caccioppoli type estimates satisfied by the derivatives of a weak solution $u$ of equation \eqref{eq:maineq}. 

\begin{Lem}\label{lem:cacci T}
For any $\beta\geq 0$ and $\eta\in C^\infty_0(\Omega)$, there exists 
$c = c(N,L)>0$ such that 
\begin{equation}\label{eq:cacci T}
\begin{aligned}
\int_\Omega \eta^2\F(|\X u|) |Tu|^\beta |\X Tu|^2\,
dx\,&\leq\, \frac{c}{\tau}(\beta+1)^2\int_\Omega
(\eta^2 +|\X\eta|^2)  \F(|\X u|)|Tu|^\beta
\big(|\X u|^2+|Tu|^2 \big)\dx\\
&\qquad + \tau \int_\Om \eta^2 \F(|\X u|)|Tu|^\beta |\XX u|^2\dx,
\end{aligned}
\end{equation}
for any arbitrary $0<\tau<1$. 
\end{Lem}

\begin{proof}
In \eqref{eq:Teq}, we choose $\varphi=\varphi_k=\eta^2|Tu|^\beta T_ku $ on \eqref{eq:Teq} and sum over $k\in\{1,\ldots,n\}$, 
 to obtain the following, 
\begin{equation}\label{eq:t1}
  \begin{aligned}
 \sum_{i,j,k}\int_\Omega \eta^2 &|Tu|^\beta D_jA_i(\X u)X_jT_k uX_iT_ku\, dx \\
   +\ \beta  &\sum_{i,j,k}\int_\Omega  \eta^2 | Tu|^{\beta-1}T_k u D_jA_i(\X u)X_jT_k uX_i(|Tu|)\dx \\
   &=  -2\sum_{i,j,k}\int_\Omega \eta\, |Tu|^\beta T_k u D_jA_i(\X u)X_jT_k uX_i\eta\, dx\\
&\qquad + \sum_{i,j,k}  \int_\Om D_jA_i(\X u)[X_j,T_k]u X_i\varphi_k\dx \\
&\qquad + \sum_{i,j,k}    \int_\Om  \varphi_k D_jA_i(\X u) [X_i, T_k]X_j u\dx\\
&= R_1 +R_2 +R_3.
  \end{aligned}
 \end{equation}
Then, we use the structure condition \eqref{eq:str} to estimate both 
sides of \eqref{eq:t1}. The left hand side of the above is estimated as 
\begin{equation}\label{eq:r0}
  \begin{aligned}
 \text{LHS of \eqref{eq:t1}} &\geq \int_\Omega \eta^2\F(|\X u|) |Tu|^\beta |\X Tu|^2\dx + \beta \int_\Omega \eta^2\F(|\X u|) |Tu|^\beta |\X(|Tu|)|^2\dx\\
&\geq \int_\Omega \eta^2\F(|\X u|) |Tu|^\beta |\X Tu|^2\dx.
  \end{aligned}
 \end{equation}
We estimate the right hand side of \eqref{eq:t1}, one by one. 

Using \eqref{eq:str} and Young's inequality, we have
\begin{equation}\label{eq:r1}
  \begin{aligned}
 |R_1| &\leq c\int_\Omega |\eta| \F(|\X u|) |Tu|^{\beta+1} |\X Tu||\X \eta|\dx \\
& \leq \eps \int_\Omega \eta^2\F(|\X u|) |Tu|^\beta |\X Tu|^2\dx + \frac{c}{\eps}\int_\Omega
|\X\eta|^2  \F(|\X u|)|Tu|^{\beta+2}\dx.
  \end{aligned}
 \end{equation}

Notice that 
\begin{equation}\label{eq:r2}
  \begin{aligned}
 R_2 &= \sum_{i,j,k}  \int_\Om \eta^2 D_jA_i(\X u)[X_j,T_k]u X_i(|Tu|^\beta T_ku)\dx \\
&\ + 2\sum_{i,j,k}  \int_\Om \eta D_jA_i(\X u)[X_j,T_k]u |Tu|^\beta T_ku X_i\eta\dx = R_{2,1} + R_{2,2}. 
  \end{aligned}
 \end{equation}
Using \eqref{eq:str} and \eqref{eq:comm} followed by Young's inequality, we have 
\begin{equation}\label{eq:r21}
  \begin{aligned}
|R_{2,1}| &\leq c(\beta+1)\int_\Om \eta^2 \F(|\X u|) \big(|\X u|+|Tu| \big)|Tu|^\beta |\X Tu|\dx\\
&\leq \eps \int_\Omega \eta^2\F(|\X u|) |Tu|^\beta |\X Tu|^2\dx\\
&\quad + \frac{c}{\eps}(\beta+1)^2 \int_\Omega \eta^2\F(|\X u|) |Tu|^\beta
\big(|\X u|^2+|Tu|^2 \big)\dx
  \end{aligned}
 \end{equation}
Similarly, we have
\begin{equation}\label{eq:r22}
  \begin{aligned}
|R_{2,2}| &\leq c\int_\Omega |\eta| \F(|\X u|) |Tu|^{\beta+1} \big(|\X u|+|Tu| \big) |\X \eta|\dx
  \end{aligned}
 \end{equation}
which is clearly majorised by the right hand side of \eqref{eq:cacci T} from Young's inequality. 

Then, $R_3$ is estimated using 
\eqref{eq:str},\eqref{eq:XTX} and Young's inequality on each term as 
\begin{equation}\label{eq:r3}
  \begin{aligned}
|R_3| &\leq c\int_\Omega \eta^2 \F(|\X u|) |Tu|^{\beta+1} \big(|\XX u|+|\X Tu| + |\X u| + |Tu|\big)\dx\\
&\leq \eps \int_\Omega \eta^2\F(|\X u|) |Tu|^\beta |\X Tu|^2\dx + c\Big(\frac{1}{\eps}+\frac{1}{\tau}\Big)\int_\Omega
\eta^2  \F(|\X u|)|Tu|^\beta \big(|\X u|^2+|Tu|^2 \big)\dx\\
&\qquad + \tau \int_\Om \eta^2 \F(|\X u|)|Tu|^\beta |\XX u|^2\dx,
  \end{aligned}
 \end{equation}
for any $0<\tau<1$. Combining \eqref{eq:r1},\eqref{eq:r21},\eqref{eq:r22},\eqref{eq:r3} with $\eps = 1/6$, it is easy to obtain 
\eqref{eq:cacci T} and the proof is finished. 
\end{proof}
\begin{Rem}\label{rem:groupcase1}
For the special case when $X_1,\ldots,X_m$ are left invariant with respect to a step $2$ Carnot Group, we have $ [X_i,T_k]=0$ and hence, it corresponds to the case of $\tau=0$ in \eqref{eq:cacci T}, i.e.  
$$ \int_\Omega \eta^2\F(|\X u|) |Tu|^\beta |\X Tu|^2\,
dx\leq c(\beta+1)^2\int_\Omega |\X\eta|^2  \F(|\X u|)|Tu|^{\beta +2}\dx.$$
Nevertheless, it does not make any difference in the subsequent results. 
\end{Rem}
The following is the Caccioppoli type estimate of the vector fields. 
\begin{Lem}\label{cacci:Xu}
For any $\beta\geq 0$ and $\eta\in C^\infty_0(\Omega)$, there exists 
$c= c(N,L)>0$ such that 
\begin{equation}\label{eq:cacci Xu}
\begin{aligned}
\int_{\Omega}\eta^2\, \F(|\X u|) |\X u|^\beta|\X\X u|^2\, dx\le&
\ c\int_\Omega \big(| \X \eta|^2+|\eta T\eta|\big) \F(|\X u|) |\X u|^{\beta+2}\, dx\\
&+c(\beta+1)^4\int_\Omega\eta^2\,  \F(|\X u|) |\X u|^\beta\big(|\X u|^2+|Tu|^2 \big)\, dx.
\end{aligned}
\end{equation}
\end{Lem}
\begin{proof}
Let $\varphi_l= \eta^2 |\X u|^\beta X_lu$ with $\eta\in C^\infty_0(\Omega)$ and $\beta\geq 0$. Notice that 
$$X_i\varphi_l =  \eta^2 |\X u|^\beta X_iX_lu + \beta  \eta^2 |\X u|^{\beta-1} X_lu X_i(|\X u|) + 2\eta |\X u|^\beta X_lu X_i\eta.$$ 
We take $Y=X_l$ and $\varphi=\varphi_l$ in \eqref{eq:yeqw2} and take summation over $l\in \{1,\ldots,m\}$ to obtain 
\begin{equation}\label{eq:xu2}
\begin{aligned}
\sum_{i,j,l}&\int_{\Omega}\eta^{2}|\X u|^\beta 
D_jA_i(\X u)X_jX_luX_iX_lu\dx\\
& +\ \beta\sum_{i,j,l}\int_\Omega \eta^2 |\X u|^{\beta-1} X_luD_jA_i(\X u)X_jX_lu X_i(|\X u|)\dx\\
& =\ \  \sum_{i,j,l}\int_{\Omega}\eta^{2}|\X u|^\beta 
D_jA_i(\X u)[X_j,X_l]u X_iX_lu\dx \\
&\quad + \beta\sum_{i,j,l}\int_\Omega \eta^2 |\X u|^{\beta-1} X_luD_jA_i(\X u)[X_j,X_l]u X_i(|\X u|)\dx\\
&\qquad-2\sum_{i,j,l}\int_{\Omega}\eta |\X u|^\beta X_luD_jA_i(\X u)X_lX_juX_i\eta\dx\\
&\qquad\ \ + \sum_{i,l}  \int_\Om A_i(\X u) [X_i, X_l]\varphi_l\dx  = J_1+J_2+J_3+J_4,
\end{aligned}
\end{equation}
where $J_i$ are the consecutive items of the right hand side of the above.

First, we estimate the left hand side of \eqref{eq:xu2} using \eqref{eq:str} to have 
\begin{equation}\label{eq:xuleft}
  \begin{aligned}
 \text{LHS of \eqref{eq:xu2}}  &\geq  \int_\Om \eta^2 \F(|\X u|) |\X u|^\beta |\X\X u|^2\dx + 
\beta \int_\Omega \eta^2\F(|\X u|) |\X u|^\beta |\X(|\X u|)|^2\dx\\
&\geq \int_\Omega \eta^2\F(|\X u|) |\X u|^\beta |\X\X u|^2\dx.
  \end{aligned}
 \end{equation}

For the right hand side of \eqref{eq:xu2}, we claim that the following holds,
\begin{equation}\label{eq:xuclaim}
\begin{aligned}
|J_i| &\leq \eps \int_\Om \eta^2 \F(|\X u|) |\X u|^\beta |\X\X u|^2\dx
+\frac{c}{\eps}\int_\Omega \big(| \X \eta|^2+|\eta T\eta|\big) \F(|\X u|) |\X u|^{\beta+2}\, dx\\
&\qquad + \frac{c}{\eps} (\beta+1)^4\int_\Omega\eta^2\,  \F(|\X u|) |\X u|^\beta \big(|\X u|^2+|Tu|^2 \big)\, dx,
\end{aligned}
\end{equation}
for any $0<\eps<1$ and $i\in\{1,2,3,4\}$, where $c= c(N,L)>0$. Then \eqref{eq:xuclaim} with 
a choice of small enough $\eps=\eps(N,L)>0$ together with \eqref{eq:xuleft} would imply \eqref{eq:cacci Xu}, thereby finishing  
the proof. We shall prove \eqref{eq:xuclaim} by estimating each $J_i$ of \eqref{eq:xu2}, one by one. 

Using \eqref{eq:str}, \eqref{eq:comm} and Young's inequality, note that 
\begin{equation}\label{j12}
\begin{aligned}
|J_1|+|J_2| &\leq c(\beta+1)\int_\Om \eta^2 \F(|\X u|) |\X u|^\beta |\X \X u| \big(|\X u|+|Tu| \big)\dx \\
&\leq \eps \int_\Om \eta^2 \F(|\X u|) |\X u|^\beta |\X\X u|^2\dx \\
&\qquad+ \frac{c}{\eps} (\beta+1)^2\int_\Omega\eta^2 \F(|\X u|) |\X u|^\beta\big(|\X u|^2+|Tu|^2 \big)\, dx,
\end{aligned}
\end{equation}
and the claim \eqref{eq:xuclaim} follows for $i=1,2$. Similarly, for $J_3$ we have 
\begin{equation}\label{j3}
\begin{aligned}
|J_3| &\leq c\int_\Om |\eta| \F(|\X u|) |\X u|^{\beta+1} |\X \X u| |\X \eta|\dx \\
&\leq \eps \int_\Om \eta^2 \F(|\X u|) |\X u|^\beta |\X\X u|^2\dx 
+\frac{c}{\eps}\int_\Omega | \X \eta|^2 \F(|\X u|) |\X u|^{\beta+2}\, dx.
\end{aligned}
\end{equation}
To estimate $J_4$, first note that 
$$ [X_i, X_l]\varphi_l = \eta^2 |\X u|^\beta [X_i, X_l]X_lu + \beta \eta^2 |\X u|^{\beta-1}X_lu [X_i, X_l](|\X u|)
+ 2\eta |\X u|^\beta X_lu [X_i, X_l]\eta,$$
and we rewrite $J_4$ as 
\begin{equation}\label{j4}
\begin{aligned}
J_4 &=  \sum_{i,l}   \int_\Om \eta^2 |\X u|^\beta  A_i(\X u) [X_i, X_l]X_lu \dx\\
&\quad+\beta\sum_{i,l}  \int_\Om \eta^2 |\X u|^{\beta-1}X_lu A_i(\X u) [X_i, X_l](|\X u|)\dx\\
&\quad+2\sum_{i,l}  \int_\Om \eta |\X u|^\beta X_lu A_i(\X u) [X_i, X_l]\eta\dx = J_{4,1}+ J_{4,2}+ J_{4,3}.
\end{aligned}
\end{equation}
The estimate of the last term is straightforward. Using \eqref{eq:str} and \eqref{eq:comm}, we have
\begin{equation}\label{j43}
|J_{4,3}| \leq c\int_\Om (|\eta T\eta|+|\eta \X\eta|) \F(|\X u|) |\X u|^{\beta+2} \dx,
\end{equation}
which is majorised by the right hand side of \eqref{eq:xuclaim}.
To estimate the rest of \eqref{j4}, we shall use $[X_i, X_l]X_lu= X_l[X_i, X_l]u + [[X_i,X_l],X_l]u$ and 
$$ [X_i, X_l](|\X u|)= |\X u|^{-1}\sum_{k=1}^m X_ku [X_i, X_l]X_ku =|\X u|^{-1}\sum_{k=1}^m X_ku \big(X_k[X_i, X_l]u+[[X_i,X_l],X_k]u\big),$$
to split the first two terms as $J_{4,1}=J_{4,1,1}+J_{4,1,2} $ and $J_{4,2}=J_{4,2,1}+J_{4,2,2}$ so that the second terms i.e. $J_{4,1,2}$ and $J_{4,2,2}$ contain the terms with the commutator of commutators.

We rewrite the first terms by integral by parts as 
\begin{equation}\label{j41121}
\begin{aligned}
J_{4,1,1} + J_{4,2,1} &=  \sum_{i,l}  \int_\Om \eta^2 |\X u|^\beta A_i(\X u) X_l[X_i, X_l]u \dx\\
&\qquad  +\beta\sum_{i,l,k}  \int_\Om \eta^2 |\X u|^{\beta-2}X_lu A_i(\X u) X_ku X_k[X_i, X_l]u\dx\\
&= - \sum_{i,l} \int_\Om X_l\Big( \eta^2 |\X u|^\beta A_i(\X u)\Big) [X_i, X_l]u \dx\\
&\qquad  -\beta\sum_{i,l,k}  \int_\Om X_k\Big(\eta^2 |\X u|^{\beta-2}X_lu X_ku A_i(\X u) \Big)[X_i, X_l]u\dx. 
\end{aligned}
\end{equation}
Then, further computation of the terms in \eqref{j41121} followed by the use of \eqref{eq:str} and \eqref{eq:comm} yields
\begin{equation}\label{j41121est}
\begin{aligned}
|J_{4,1,1}| + |J_{4,2,1}| &\leq c(\beta+1)^2 \int_\Om \eta^2  \F(|\X u|) |\X u|^{\beta} |\X \X u| \big(|\X u|+|Tu| \big)\dx\\
&\qquad + c(\beta+1) \int_\Om |\eta|\F(|\X u|) |\X u|^{\beta+1} |\X\eta| \big(|\X u|+|Tu| \big)\dx\\
&\leq \eps \int_\Om \eta^2 \F(|\X u|) |\X u|^\beta |\X\X u|^2\dx\\
&\qquad + \frac{c}{\eps} (\beta+1)^4\int_\Omega\eta^2\,  \F(|\X u|) |\X u|^\beta\big(|\X u|^2+|Tu|^2 \big)\, dx\\
&\qquad \ + c\int_\Omega |\X \eta|^2 \F(|\X u|) |\X u|^{\beta+2}\, dx,
\end{aligned}
\end{equation}
where the latter inequality of the above follows by using Young's inequality on both items. Now we are only left with
\begin{equation}\label{j41222}
\begin{aligned}
J_{4,1,2} + J_{4,2,2} &=  \sum_{i,l}  \int_\Om \eta^2 |\X u|^\beta A_i(\X u)  [[X_i,X_l],X_l]u \dx\\
&\qquad  +\beta\sum_{i,l,k}  \int_\Om \eta^2 |\X u|^{\beta-2}X_lu A_i(\X u) X_ku  [[X_i,X_l],X_k]u\dx\\
\end{aligned}
\end{equation}
To estimate \eqref{j41222}, we use \eqref{eq:XXX} to have 
\begin{equation}\label{j41222est}
\begin{aligned}
|J_{4,1,2}| + |J_{4,2,2}| \leq c(\beta+1) \int_\Om \eta^2 \F(|\X u|) |\X u|^{\beta+1}\big( |\X u| + |Tu|\big)\dx,
\end{aligned}
\end{equation}
which is estimated easily just as in \eqref{j41121est}. 
Thus, \eqref{j41121est}, \eqref{j41222} and \eqref{j43} together 
imply that the claim \eqref{eq:xuclaim} holds for $i=4$ as well. This completes the proof of the the claim \eqref{eq:xuclaim} and hence the proof of the lemma. 
\end{proof}

The following is a reverse type inequality similar to earlier papers, see \cite{Min-Z-Zhong, Zhong, Muk0}, etc. 
\begin{Lem}\label{lem:rev}
For any $\beta\ge 2$ and all non-negative $\eta\in C^\infty_0(\Omega)$, we have
\begin{equation}\label{eq:rev}
\begin{aligned}
\int_\Omega \eta^{\beta+2} \F(|\X u|)| Tu|^{\beta}|\X\X u|^2\, dx 
\le
c(\beta+1)^4 K_\eta
\int_{\Omega}\eta^{\beta}\weight |\X u|^2 | Tu|^{\beta-2}| \X\X u|^2 \dx
\end{aligned}
\end{equation}
for some constant $c=c(N,L)>0$, where $K_\eta =\| \eta\| _{L^\infty}^2 + \| \X\eta\| _{L^\infty}^2$. 
\end{Lem}

\begin{proof}
Recalling \eqref{eq:yeqw1} with $Y=X_l$ and $\varphi=\varphi_l$, we have
\begin{equation}\label{yeqX}
\sum_{i,j}\int_\Om D_jA_i(\X u)X_lX_ju X_i\varphi_l\dx = \sum_{i=1}^m  \int_\Om A_i(\X u) [X_i, X_l]\varphi_l\dx, 
\end{equation}
where we use $\varphi_l=\eta^{\beta+2}| Tu|^\beta X_lu$, where $\eta\in C^\infty_0(\Omega)$ is non-negative and $\beta\geq 2$. 
Note that 
$$
X_i\varphi_l=\eta^{\beta+2} | Tu|^\beta X_iX_l u+\beta \eta^{\beta+2}| Tu|^{\beta-1}
X_l uX_i(|Tu|)+(\beta+2)\eta^{\beta+1} X_i\eta| Tu|^\beta X_l u.$$
Using this on the above, we obtain 
\begin{equation}\label{1rv}
\begin{aligned}
\sum_{i,j}\int_\Om \eta^{\beta+2} &| Tu|^\beta D_jA_i(\X u)X_lX_ju  X_iX_l u\dx \\
&= -\beta \sum_{i,j}\int_\Om \eta^{\beta+2}| Tu|^{\beta-1}X_lu D_jA_i(\X u)X_lX_ju X_i(|Tu|)\dx\\
&\quad - (\beta+2)\sum_{i,j}\int_\Om \eta^{\beta+1} | Tu|^\beta X_l u D_jA_i(\X u)X_lX_ju X_i\eta\dx \\
&\qquad - \sum_{i,j}  \int_\Om \varphi_l D_jA_i(\X u) [X_i, X_l]X_j u\dx,
\end{aligned}
\end{equation}
where we have used integral by parts on the term in the right hand side of \eqref{yeqX}.
Now we use $X_iX_l = X_lX_i +[X_i,X_l]$ on the left hand side term and then sum over $l\in\{1,\ldots,m\}$ both sides of 
\eqref{1rv} to obtain 
 \begin{equation}\label{2rv}
\begin{aligned}
\sum_{i,j,l} \int_\Om \eta^{\beta+2} &| Tu|^\beta D_jA_i(\X u)X_lX_ju X_lX_i u\dx \\
&= -\sum_{i,j,l} \int_\Om \eta^{\beta+2}| Tu|^\beta D_jA_i(\X u)X_lX_ju [X_i,X_l]u\dx \\
&\quad -\beta \sum_{i,j,l} \int_\Om \eta^{\beta+2}| Tu|^{\beta-1}X_lu D_jA_i(\X u)X_lX_ju X_i(|Tu|)\dx\\
&\quad - (\beta+2)\sum_{i,j,l} \int_\Om \eta^{\beta+1} | Tu|^\beta X_l u D_jA_i(\X u)X_lX_ju X_i\eta\dx \\
&\qquad -  \sum_{i,j,l}  \int_\Om \eta^{\beta+2}| Tu|^\beta X_lu D_jA_i(\X u) [X_i, X_l]X_j u\dx\\
&= I_1+I_2+I_3+I_4.
\end{aligned}
\end{equation}
We will estimate both sides of \eqref{2rv} as follows. 

For the
left hand side, the structure condition \eqref{eq:str} implies that
$$
\text{LHS of \eqref{2rv}}
\ge \int_{\Omega} \eta^{\beta+2}\weight| Tu|^\beta |\XX u|^2\, dx.
$$
For the right hand side of \eqref{2rv}, we will show that for each item,
the following estimate holds,
\begin{equation}\label{claim1}
\begin{aligned}
| I_k|\ \le &\ c\eps \int_{\Omega}\eta^{\beta+2}\weight| Tu|^{\beta}|\X\X u|^2\, dx\\
&\ + \frac{c}{\eps}(\beta+1)^4 K_\eta
\int_{\Omega}\eta^{\beta}\weight |\X u|^2 | Tu|^{\beta-2}| \X\X u|^2 \dx,
\end{aligned}
\end{equation}
for $k=1,2,3, 4$. Then, a choice of a small enough $\eps=\eps(N,L)>0$, completes the proof. 

To prove the claim \eqref{claim1}, first we note that using \eqref{eq:ineqT} on \eqref{eq:cacci T} of Lemma \ref{lem:cacci T} along with minor modifications, we can obtain 
\begin{equation}\label{weak3}
\begin{aligned}
\int_{\Omega}\eta^{\beta+4}\weight 
| Tu|^\beta|\X Tu|^2\dx
\le c(\beta+1)^2 K_\eta \int_{\Omega} \eta^{\beta+2}\weight|
Tu|^{\beta}\big( |\X u|^2 + |\XX u|^2\big)\dx
\end{aligned}
\end{equation}
where $K_\eta =\| \eta\| _{L^\infty}^2 + \| \X\eta\| _{L^\infty}^2$.
This shall be used multiple times 
below.

First, to prove that \eqref{claim1} holds for $I_1$, we note that 
\begin{align*}
I_1 &=  -\sum_{i,j,l} \int_\Om \eta^{\beta+2}| Tu|^\beta D_jA_i(\X u)X_lX_ju [X_i,X_l]u\dx\\
&= - \sum_{i,l}\int_\Om \eta^{\beta+2}| Tu|^\beta X_l(A_i(\X u))[X_i,X_l]u\dx,
\end{align*}
from chain rule. Then 
integration by parts yields
\begin{equation*}
\begin{aligned}
I_1=& \sum_{i,l}\int_{\Omega} A_i(\X u)X_l\big(\eta^{\beta+2}| Tu|^\beta [X_i,X_l]u \big)\, dx\\
=&\sum_{i,l}\int_{\Omega}\eta^{\beta+2} A_i(\X u)X_l\big(| Tu|^\beta [X_i,X_l]u \big)\, dx\\
&-(\beta+2)\sum_{i,l}\int_{\Omega}\eta^{\beta+1}
A_i(\X u) X_l\eta \,| Tu|^\beta [X_i,X_l]u \, dx
=I_{1,1}+I_{1,2}.
\end{aligned}
\end{equation*}
Hence, we have 
\begin{equation}\label{i11}
\begin{aligned}
I_{1,1}&=\sum_{i,l}\int_{\Omega}\eta^{\beta+2} A_i(\X u)| Tu|^\beta X_l[X_i,X_l]u \, dx\\
&\quad +\beta \sum_{i,l}\int_{\Omega}\eta^{\beta+2} A_i(\X u)| Tu|^{\beta-1}[X_i,X_l]u \,X_l(|Tu|)\, dx\\
&= I_{1,1,1}+I_{1,1,2}.
\end{aligned}
\end{equation}
For $I_{1,1,1}$, using \eqref{eq:comm} and \eqref{eq:str}, we have 
\begin{equation*}
\begin{aligned}
| I_{1,1,1}| &\le c\int_{\Omega}\eta^{\beta+2}\weight |\X u| | Tu|^\beta| \X Tu|\dx 
+ c \int_{\Omega}\eta^{\beta+2}\weight |\X u| | Tu|^\beta | \XX u|\dx\\
&\leq \frac{\eps}{(\beta+1)^2 K_\eta}\int_\Omega \eta^{\beta+4}\weight
| Tu|^\beta|\X Tu|^2\,
\dx + \frac{c}{\eps} (\beta+1)^2 K_\eta 
\int_{\Omega}\eta^{\beta}\weight |\X u|^2 | Tu|^{\beta}\dx\\
&\quad + \eps \int_\Om \eta^{\beta+2}\weight | Tu|^\beta| \XX u|^2\dx 
+ \frac{c}{\eps}\int_\Om  \eta^{\beta+2}\weight | Tu|^\beta| \X u|^2\dx,
\end{aligned}
\end{equation*}
where the latter inequality is obtained by using Young's inequality on both terms. 
Then, using \eqref{weak3} to estimate the first term of the right hand side of the above, it is not hard to see that 
\begin{equation}\label{i111}
| I_{1,1,1}| \leq c\eps \int_\Om \eta^{\beta+2}\weight | Tu|^\beta| \XX u|^2\dx 
+  \frac{c}{\eps}(\beta+1)^2 K_\eta \int_\Om  \eta^{\beta}\weight | Tu|^\beta| \X u|^2\dx
\end{equation}
holds for $0<\eps<1,\ K_\eta =\| \eta\| _{L^\infty}^2 + \| \X\eta\| _{L^\infty}^2$ and some large enough $c=c(N,L)>0$. Thus, we have that 
$I_{1,1,1}$ satisfies the estimate of the claim \eqref{claim1}.

For $I_{1,1,2}$, we just use $|[X_i,X_l]u|\leq 2|\XX u|$ along with \eqref{eq:str} and Young's inequality to get
\begin{equation}\label{i112}
\begin{aligned}
| I_{1,1,2}|\ &\le \ c\beta \int_{\Omega}\eta^{\beta+2}\weight |\X u| | Tu|^{\beta-1}| \XX u| | \X Tu| \dx\\
&\leq\ \frac{\eps}{(\beta+1)^2 K_\eta}\int_\Omega \eta^{\beta+4}\weight
| Tu|^\beta|\X Tu|^2\,
dx\\
&\qquad + \frac{c}{\eps}\beta^2(\beta+1)^2 K_\eta 
\int_{\Omega}\eta^{\beta}\weight |\X u|^2 | Tu|^{\beta-2}| \X\X u|^2\, dx. 
\end{aligned}
\end{equation} 
Then, using \eqref{weak3} to estimate the first term similarly as before, it is easy to see that $I_{1,1,2}$ also satisfies the estimate of the claim \eqref{claim1}. 

For $I_{1,2}$, we similarly use \eqref{eq:str}, $|[X_i,X_l]u|\leq 2|\XX u|$ and Young's inequality to obtain 
\begin{equation}\label{i12}
\begin{aligned}
 | I_{1,2}| &\le  c(\beta+1)\int_{\Omega}\eta^{\beta+1}|\X\eta|\weight |\X u|| Tu|^{\beta}|\XX u|\dx\\
&\leq c\eps \int_\Om \eta^{\beta+2}\weight | Tu|^\beta| \XX u|^2\dx 
+  \frac{c}{\eps}(\beta+1)^2  \int_\Om  \eta^{\beta}|\X\eta|^2\weight | Tu|^\beta| \X u|^2\dx.
\end{aligned}
\end{equation} 
Combining \eqref{i111},\eqref{i112} and \eqref{i12}, we conclude that the claim \eqref{claim1} holds for $I_1$. 

 Similarly, we estimate $I_2$ by structure condition \eqref{eq:str} and Young's inequality to get 
\begin{equation*}
\begin{aligned}
| I_2|\le\, c\beta \int_\Omega
\eta^{\beta+2}\weight |\X u|| Tu|^{\beta-1}| \X \X u\| \X Tu|\, dx,
\end{aligned}
\end{equation*}
which is estimated exactly as $I_{1,1,2}$ and thus the claim \eqref{claim1} holds for $I_2$. 

For $I_3$, we have by structure condition \eqref{eq:str}, that 
\[
|I_3 |\le c(\beta+1)\int_{\Omega}\eta^{\beta+1}|\X\eta|\weight |\X u|| Tu|^\beta| \X\X u|\, dx,
\]which is estimated exactly as $I_{1,2}$ and thus the claim \eqref{claim1} holds for $I_3$ as well.

Finally, we note that 
\begin{equation*}
\begin{aligned}
I_4 &= -  \sum_{i,j,l}  \int_\Om \eta^{\beta+2}| Tu|^\beta X_lu D_jA_i(\X u) X_j[X_i, X_l] u\dx\\
&\qquad +  \sum_{i,j,l}  \int_\Om \eta^{\beta+2}| Tu|^\beta X_lu D_jA_i(\X u) [[X_i, X_l],X_j]u\dx\\
&= I_{4,1}+I_{4,2}. 
\end{aligned}
\end{equation*}
 From \eqref{eq:comm} and \eqref{eq:str}, we find that 
\begin{equation}\label{i41}
\begin{aligned}
| I_{4,1}|\ \le  c\int_{\Omega}\eta^{\beta+2}\weight |\X u| | Tu|^\beta| \X Tu|\dx 
+ c \int_{\Omega}\eta^{\beta+2}\weight |\X u| | Tu|^\beta | \XX u|\dx
\end{aligned}
\end{equation}
which is estimated exactly as $I_{1,1,1}$. For $I_{4,2}$, we use \eqref{eq:str} followed by \eqref{eq:XXX} and \eqref{eq:ineqT} used successively to obtain
\begin{equation}\label{i42}
\begin{aligned}
| I_{4,2}|\ \le  c\int_{\Omega}\eta^{\beta+2}\weight |\X u| | Tu|^\beta| \big(|\X u| +|\XX u|)\dx,
\end{aligned}
\end{equation}
the first term of the above is estimated from \eqref{eq:ineqT} and the second term is exactly equal to the second term of \eqref{i41}. Thus, from \eqref{i41} and \eqref{i42} we conclude that the claim \eqref{claim1} holds for $I_4$ and the proof is finished. 
\end{proof}
\begin{Rem}
It is not difficult to see that from the above proof, the case $\beta=0$ leads to 
\begin{equation}\label{eq:beta0XX}
\int_\Om \eta^2 \weight |\XX u|^2\dx \leq c\big(\| \eta\| _{L^\infty}^2 + \| \X\eta\| _{L^\infty}^2\big)
 \int_{\supp(\eta)} \weight |\X u|^2\dx,
\end{equation}
since all the terms containing $\beta |Tu|^{\beta-1}$ in the proof vanish in the first steps of the estimates. 
\end{Rem}
The following corollary is proved by using \eqref{eq:ineqT} and H\"older's inequality on Lemma \ref{lem:rev}.
\begin{Cor}\label{cor:rep}
Given $\beta\ge 2$ and non-negative $\eta\in C^\infty_0(\Omega)$, we have $c=c(N,L)>0$ such that
\begin{equation}\label{eq:rep}
\begin{aligned}
\int_\Omega \eta^{\beta+2}\weight| Tu|^\beta| \X\X u|^2\, dx \le c^{\frac{\beta}{2}}(\beta+1)^{2\beta} K_\eta^{\beta/2}
\int_\Omega \eta^2\, \weight |\X u|^{\beta}  |\X\X u|^2\, dx 
\end{aligned}
\end{equation}
where $K_\eta =\| \eta\| _{L^\infty}^2 + \| \X\eta\| _{L^\infty}^2$ and $c=c(N,L)>0$. 
\end{Cor}

\begin{proof}
The right hand side of \eqref{eq:rev} is 
estimated by H\"older's inequality as 
\begin{equation}\label{rep2}
\begin{aligned}
\int_{\Omega}\eta^{\beta}& \weight |\X u|^2| Tu|^{\beta-2}| \X\X u|^2 \dx \\
&\leq \Big( \int_\Omega \eta^{\beta+2} \F(|\X u|)| Tu|^{\beta}|\X\X u|^2\, dx \Big)^\frac{\beta-2}{\beta}
 \Big(\int_\Omega \eta^2\, \weight |\X u|^{\beta}  |\X\X u|^2\, dx\Big)^\frac{2}{\beta}.
\end{aligned}
\end{equation}
Then, using Young's inequality on \eqref{rep2} and combining with \eqref{eq:rev}, 
the proof is finished.
\end{proof}

Now, we can remove the error term in \eqref{eq:cacci Xu} and obtain the following apriori estimate. This shall be 
used to prove local Lipschitz regularity of weak solutions. 
\begin{Lem}\label{lem:est Xu}
For any $\beta\geq 2$ and all non-negative $\eta\in C^\infty_0(\Omega)$, we have that
\begin{equation}\label{eq:est Xu}
\int_{\Omega}\eta^2\, \F(|\X u|) |\X u|^\beta|\X\X u|^2\, dx \le c (\beta+1)^{12} K_\eta
\int_{\supp(\eta)}\weight |\X u|^{\beta+2}\, dx,
\end{equation}
where $K_\eta =\| \eta\|_{L^\infty}^2+\| \X\eta\|_{L^\infty}^2+\|\eta
T\eta\|_{L^\infty}$ and $c=c(N,L)>0$. 
\end{Lem}
\begin{proof}
Let us recall \eqref{eq:cacci Xu} and rewrite it as 
\begin{equation}\label{cx}
\begin{aligned}
\int_{\Omega}\eta^2\, \F(|\X u|) |\X u|^\beta|\X\X u|^2\, dx\le&
\ c(\beta+1)^4\int_\Omega \big(\eta^2+| \X \eta|^2+|\eta T\eta|\big) \F(|\X u|) |\X u|^{\beta+2}\, dx\\
&+c(\beta+1)^4\int_\Omega\eta^2\,  \F(|\X u|) |\X u|^\beta| Tu|^2\, dx.
\end{aligned}
\end{equation}
It is clear that we need to estimate only the last term of the right hand side of \eqref{cx}. We estimate it by H\"older's inequality and Young's inequality as  
\begin{equation}\label{ef1}
\begin{aligned}
c&(\beta+1)^4\int_\Omega\eta^2\weight |\X u|^\beta| Tu|^2\, dx\\
&\quad \le c(\beta+1)^4\Big(\int_\Omega \eta^{\beta+2}\weight 
|Tu|^{\beta+2}\, dx\Big)^{\frac{2}{\beta+2}}\Big(\int_{\supp(\eta)}\weight |\X u|^{\beta+2}\, dx\Big)^{\frac{\beta}{\beta+2}}\\
&\quad \le c\, h^\frac{\beta+2}{2} \int_\Omega \eta^{\beta+2}\weight 
|Tu|^{\beta+2}\, dx + \frac{c}{h^\frac{\beta+2}{\beta}}(\beta+1)^{\frac{4(\beta+2)}{\beta}} \int_{\supp(\eta)}\weight |\X u|^{\beta+2}\, dx
\end{aligned}
\end{equation}
for some $h>0$ to be chosen later. The first term of the right hand side of \eqref{ef1} is estimated using 
\eqref{eq:ineqT} and Corrollary \ref{cor:rep} as 
\begin{equation}\label{ef2}
\begin{aligned}
 c\, h^\frac{\beta+2}{2}\int_\Omega \eta^{\beta+2}\weight 
|Tu|^{\beta+2}\, dx &\leq  c\, h^\frac{\beta+2}{2}\int_\Omega \eta^{\beta+2}\weight 
|Tu|^{\beta}|\XX u|^2\, dx\\
&\leq (ch)^\frac{\beta+2}{2}(\beta+1)^{2\beta} K_\eta^{\beta/2} \int_\Omega \eta^2\, \weight |\X u|^{\beta}  |\X\X u|^2\, dx.
\end{aligned}
\end{equation}
Now, for a large enough $c=c(N,L)>0$ and small enough $\eps>0$, we take
\begin{equation}\label{eq:hch}
 h = \frac{\eps}{c(\beta+1)^{4\beta/(\beta+2)} K_\eta^{\beta/(\beta+2)}},
\end{equation}
so that the above implies 
\begin{equation}\label{ef3}
\begin{aligned}
 c\, h^\frac{\beta+2}{2}\int_\Omega &\eta^{\beta+2}\weight |Tu|^{\beta+2} \dx \leq \eps^\frac{\beta+2}{2} \int_\Omega \eta^2\, \weight |\X u|^{\beta}  |\X\X u|^2\, dx
\end{aligned}
\end{equation}
Then, it is not difficult to observe that \eqref{ef3} used in \eqref{ef1} followed by \eqref{ef1} used in \eqref{cx} with 
the choice of $h$ in \eqref{eq:hch} for a small enough $\eps=\eps(N,L)>0$, leads to 
\[
\int_\Omega \eta^2\, \weight |\Xu|^\beta |\XX u|^2\, dx
\le c(\beta+1)^{\frac{4(\beta+2)}{\beta}+4} K_\eta \int_{\supp(\eta)} \weight |\X u|^{\beta+2}\, dx,
\]
for some $c=c(N,L)>0$. 
This completes the proof. 
\end{proof}
The following corollaries would be essential for proving the $C^{1,\alpha}$-regularity of weak solutions. 
\begin{Cor}\label{lem:est Tu}
For any $\beta\geq 2$ and all non-negative $\eta\in C^\infty_0(\Omega)$, we have that
\begin{equation}\label{eq:est Tu}
\int_\Omega\eta^{\beta+2}\,\weight| Tu|^{\beta+2}\, dx \le c(\beta)K_\eta^{\frac{\beta+2}{2}}
\int_{\supp(\eta)}\weight |\X u|^{\beta+2}\, dx,
\end{equation}
where $K_\eta =\| \eta\|_{L^\infty}^2+\| \X\eta\|_{L^\infty}^2+\|\eta
T\eta\|_{L^\infty}$ and $c(\beta)=c(N,L,\beta)>0$. 
\end{Cor}
\begin{proof}
It is easy to see that from \eqref{eq:ineqT}, we have 
$$ \int_\Omega\eta^{\beta+2}\,\weight| Tu|^{\beta+2}\, dx \leq c\int_\Omega\eta^{\beta+2}\,\weight| Tu|^{\beta}|\XX u|^2\, dx $$ 
and then \eqref{eq:est Tu} follows easily from Corrollary \ref{cor:rep} and \eqref{eq:est Xu}. This completes the proof.
\end{proof}

\begin{Cor}\label{cor:est XTu}
For any $\beta\geq 2$ and all non-negative $\eta\in C^\infty_0(\Omega)$, we have that
\begin{equation}\label{eq:est XTu}
\int_{\Omega}\eta^{\beta+4}\weight 
| Tu|^\beta|\X Tu|^2\dx \leq c(\beta)K_\eta^{\frac{\beta+4}{2}}
\int_{\supp(\eta)}\weight |\X u|^{\beta+2}\, dx,
\end{equation}
where $K_\eta =\| \eta\|_{L^\infty}^2+\| \X\eta\|_{L^\infty}^2+\|\eta
T\eta\|_{L^\infty}$ and $c(\beta)=c(N,L,\beta)>0$. 
\end{Cor}
\begin{proof}
Recalling the variant \eqref{weak3} of \eqref{eq:cacci T} of Lemma \ref{lem:cacci T}, we have 
\begin{equation*}
\begin{aligned}
\int_{\Omega}\eta^{\beta+4}\weight 
| Tu|^\beta|\X Tu|^2\dx
\le c(\beta+1)^2 K_\eta \int_{\Omega} \eta^{\beta+2}\weight|
Tu|^{\beta}\big( |\X u|^2 + |\XX u|^2\big)\dx. 
\end{aligned}
\end{equation*}
The second term of the right hand side of the above is estimated from Corrollary \ref{cor:rep} and \eqref{eq:est Xu} and the first term of the right hand side of the above is estimated by H\"older's inequality as 
$$ \int_{\Omega} \eta^{\beta+2}\weight|
Tu|^{\beta} |\X u|^2\dx \leq \Big( \int_{\Omega} \eta^{\beta+2}\weight
|Tu|^{\beta+2}\dx\Big)^\frac{\beta}{\beta+2} 
\Big( \int_{\Omega} \eta^{\beta+2}\weight
|\X u|^{\beta+2}\dx\Big)^\frac{2}{\beta+2} $$ 
and it is further estimated from \eqref{eq:est Tu}. Combining all the estimates, the proof is finished.
\end{proof}
Finally, we need the following Corrollary, which also relies on \eqref{eq:rep}, \eqref{eq:est Xu} and \eqref{eq:est Tu}. 
\begin{Cor}\label{cor:est0 XTu}
For any $\beta\geq 2$ and all non-negative $\eta\in C^\infty_0(\Omega)$, we have that
\begin{equation}\label{eq:est0 XTu}
\int_{\Omega}\eta^{\beta+4}\weight 
| \X u|^\beta|\X Tu|^2\dx \leq c(\beta) K_\eta^\frac{\beta+4}{2}
\int_{\supp(\eta)}\weight |\X u|^{\beta+2}\, dx,
\end{equation}
where $K_\eta =\| \eta\|_{L^\infty}^2+\| \X\eta\|_{L^\infty}^2+\|\eta
T\eta\|_{L^\infty}$ and $c(\beta)=c(N,L,\beta)>0$. 
\end{Cor}
\begin{proof}
By testing the equation \eqref{eq:Teq} with $\varphi_k=\eta^{\beta+4}|\X u|^\beta 
T_ku $ and summing over $k\in\{1,\ldots,n\}$, we get 
\begin{equation}\label{eq:xt1}
  \begin{aligned}
 \sum_{i,j,k}\int_\Omega &\eta^{\beta+4} |\X u|^\beta D_jA_i(\X u)X_jT_k uX_iT_ku\, dx \\
   &\ = -(\beta+4)\sum_{i,j,k}\int_\Omega \eta^{\beta+3} \, |\X u|^\beta 
T_k u D_jA_i(\X u)X_jT_k uX_i\eta\, dx\\
 &\qquad -\beta \sum_{i,j,k}\int_\Omega  \eta^{\beta+4} |\X u|^{\beta-1} 
T_k u D_jA_i(\X u)X_jT_k uX_i(|\X u|)\dx \\
&\qquad + \sum_{i,j,k}  \int_\Om D_jA_i(\X u)[X_j,T_k]u X_i\varphi_k\dx \\
&\qquad + \sum_{i,j,k}    \int_\Om  \varphi_k D_jA_i(\X u) [X_i, T_k]X_j u\dx\\
&= L_1+L_2+L_3+L_4.
  \end{aligned}
 \end{equation}
From the structure condition \eqref{eq:str}, 
\begin{equation}\label{eq:r0}
  \begin{aligned}
 \text{LHS of \eqref{eq:xt1}} &\geq \int_\Omega \eta^{\beta+4}\F(|\X u|) |\X u|^\beta |\X Tu|^2\dx 
+ \beta \int_\Omega \eta^{\beta+4}\F(|\X u|) |\X u|^\beta |\X(|Tu|)|^2\dx\\
&\geq \int_\Omega \eta^{\beta+4}\F(|\X u|) |\X u|^\beta |\X Tu|^2\dx.
  \end{aligned}
 \end{equation}
We estimate the right hand side of \eqref{eq:xt1}, one by one. We claim that for $i\in \{1,2,3,4\}$, 
\begin{equation}\label{eq:lclaim}
 |L_i|  \leq \eps \int_\Omega \eta^{\beta+4}\F(|\X u|) |\X u|^\beta |\X Tu|^2\dx + \frac{c(\beta)}{\eps} K_\eta^\frac{\beta+4}{2}
\int_{\supp(\eta)}\weight |\X u|^{\beta+2}\, dx,
\end{equation}
which, with a small enough $\eps=\eps(N,L)>0$, is enough to conclude the proof.

First, using \eqref{eq:str} and Young's inequality, we have
\begin{equation}\label{eq:l1}
  \begin{aligned}
 |L_1| &\leq c(\beta+1)\int_\Omega |\eta|^{\beta+3} \F(|\X u|) |\X u|^{\beta} |Tu| |\X Tu||\X \eta|\dx \\
& \leq \eps \int_\Omega \eta^{\beta+4}\F(|\X u|) |\X u|^\beta |\X Tu|^2\dx + \frac{c}{\eps}(\beta+1)^2\int_\Omega
\eta^{\beta+2}|\X\eta|^2  \F(|\X u|)|\X u|^\beta |Tu|^2\dx,
  \end{aligned}
 \end{equation}
and the claim \eqref{eq:lclaim} follows from H\"older's inequality and \eqref{eq:est Tu}. Similarly, we have 
\begin{equation}\label{eq:l2}
  \begin{aligned}
 |L_2| &\leq c\beta \int_\Omega \eta^{\beta+4} \F(|\X u|) |\X u|^{\beta-1} |Tu| |\X Tu||\XX u|\dx \\
& \leq \eps \int_\Omega \eta^{\beta+4}\F(|\X u|) |\X u|^\beta |\X Tu|^2\dx + \frac{c}{\eps}\beta^2\int_\Omega
\eta^{\beta+4} \F(|\X u|)|\X u|^{\beta-2} |Tu|^2 |\XX u|^2\dx,
  \end{aligned}
 \end{equation}
where the last term in the right hand side of \eqref{eq:l2} is estimated by H\"older's inequality as 
\begin{equation*}
\begin{aligned}
\int_{\Omega}\eta^{4}& \weight |T u|^2| \X u|^{\beta-2}| \X\X u|^2 \dx \\
&\leq \Big( \int_\Omega \eta^{2} \F(|\X u|)| \X u|^{\beta}|\X\X u|^2\, dx \Big)^\frac{\beta-2}{\beta}
 \Big(\int_\Omega \eta^{\beta+2}\, \weight |T u|^{\beta}  |\X\X u|^2\, dx\Big)^\frac{2}{\beta}.
\end{aligned}
\end{equation*}
and the claim \eqref{eq:lclaim} follows from \eqref{eq:rep} and \eqref{eq:est Xu}.

To continue the estimates, note that 
\begin{equation}\label{l3}
  \begin{aligned}
L_3&= \sum_{i,j,k}\int_\Omega \eta^{\beta+4} |\X u|^\beta D_jA_i(\X u)[X_j,T_k]u X_iT_ku\, dx \\
   &\ +(\beta+4)\sum_{i,j,k}\int_\Omega \eta^{\beta+3}\, |\X u|^\beta 
T_k u D_jA_i(\X u)[X_j,T_k]u X_i\eta\, dx\\
 &\ -\beta \sum_{i,j,k}\int_\Omega  \eta^{\beta+4} |\X u|^{\beta-1} 
T_k u D_jA_i(\X u)[X_j,T_k]u X_i(|\X u|)\dx.
  \end{aligned}
 \end{equation}
Using \eqref{eq:str} and \eqref{eq:comm}, we obtain
\begin{equation*}
  \begin{aligned}
|L_3|&\leq c\int_\Omega \eta^{\beta+4}\weight |\X u|^\beta \big( |\X u|+|T u|\big) |\X Tu| \, dx \\
   &\ +c (\beta+1)\int_\Omega |\eta|^{\beta+3} \weight |\X u|^\beta 
|T u| \big( |\X u|+|T u|\big)|\X\eta|\, dx\\
 &\ +c\beta \int_\Omega  \eta^{\beta+4}\weight |\X u|^{\beta-1} 
|T u| \big( |\X u|+|T u|\big) |\XX u| \dx,
  \end{aligned}
 \end{equation*}
and then, by Young's inequality on all the above terms, we have 
\begin{equation}\label{eq:estl3}
  \begin{aligned}
|L_3| &\leq \eps \int_\Omega \eta^{\beta+4}\F(|\X u|) |\X u|^\beta |\X Tu|^2\dx \\
&\ + \frac{c}{\eps}(\beta+1)^2\int_\Omega \eta^{\beta+2}(\eta^2+|\X\eta|^2)\F(|\X u|) |\X u|^\beta \big( |\X u|^2+|T u|^2\big)\dx\\
&\ + c\beta^2\int_\Omega
\eta^{\beta+4} \F(|\X u|)|\X u|^{\beta-2} |Tu|^2 |\XX u|^2\dx,
 \end{aligned}
 \end{equation}
which is estimated similarly as $L_1$ and $L_2$ in \eqref{eq:l1} and \eqref{eq:l2} above. 

Finally, we note that 
$$ L_4 = \sum_{i,j,k} \int_\Om \eta^{\beta+4} |\X u|^\beta T_ku D_jA_i(\X u) [X_i, T_k]X_j u\dx $$ 
which, by \eqref{eq:str} and  
\eqref{eq:XTX}, leads to 
\begin{equation}\label{eq:l4}
  |L_4| \leq c\int_\Omega\eta^{\beta+4} \weight |\X u|^\beta  |Tu| \big(|\XX u|+|\X Tu| + |\X u| + |Tu|\big)\dx,
\end{equation}
which is a linear combination of the right hand sides of the estimates $L_1, L_2$ and $L_3$ of the above, and hence is estimated accordingly. Thus \eqref{eq:lclaim} holds for all $L_i$'s and the proof is finished. 
\end{proof}

\begin{Rem}
Similarly as in the case of \eqref{eq:beta0XX}, it is not hard to see that from the above proof, the case $\beta=0$ leads to 
\begin{equation}\label{eq:beta0XT}
\int_\Om \eta^2 \weight |\X T u|^2\dx \leq cK_\eta 
 \int_{\supp(\eta)} \weight \big( |\X u|^2+|T u|^2\big)\dx,
\end{equation}
with $K_\eta = \| \eta\|_{L^\infty}^2 + \| \X\eta\|_{L^\infty}^2+ \|\eta T\eta\|_{L^\infty}$, 
since all the terms containing $\beta |\X u|^{\beta-1}$ vanish. 
\end{Rem}

\section{Local Lipschitz Continuity}\label{sec:lip}
In this section, we shall prove Theorem \ref{thm:mainthm1} using the Caccioppoli type estimate \eqref{eq:est Xu} obtained in Section \ref{sec:apest} and an approximation argument. We shall fix $x_0\in \Om$ and denote all CC-balls as $B_r=B_r(x_0)$ with dimension $Q=Q(x_0)$. 
Notice that whenever $\delta>0$, the function 
\begin{equation}\label{eq:fp}
\F(t) = (\delta+t^2)^\frac{p-2}{2}
\end{equation}
satisfies all conditions in the begining of Section \ref{sec:apest} for any $1<p<\infty$ and \eqref{eq:str} is identical to the structure condition \eqref{eq:pstr}. The substitution \eqref{eq:fp} on the estimate \eqref{eq:est Xu} followed by the use of Sobolev inequality and Moser's iteration, gives us the estimate of Theorem \ref{thm:mainthm1} subject to the apriori assumption \eqref{eq:ap reg}.
The proof is standard, we provide a brief outline below.
\begin{Thm}\label{thm:aplip}
 If $ u \in \subsob^{1,2}_\loc(\Om)$ is a weak solution of equation \eqref{eq:peq} and \eqref{eq:ap reg} holds, then 
 \begin{equation}\label{eq:aplocbound}
  \sup_{B_{\sigma r}}\ |\X u|  \leq \frac{c}{(1-\sigma)^{Q/p}}\bigg(\intav_{B_r}\big(\delta+|\X u|^2\big)^\frac{p}{2}\dx \bigg)^\frac{1}{p}
 \end{equation}
holds for any $0<\sigma<1$, where $1<p<\infty,\ \delta>0 $ and $c = c(N,p,L) > 0 $.
\end{Thm}

\begin{proof}
We rewrite the Caccioppoli type estimate \eqref{eq:est Xu} with the substitution \eqref{eq:fp} to have 
\begin{equation}\label{eq:pestXu}
\int_{\Omega}\eta^2 (\delta+|\X u|^2)^\frac{p-2}{2}|\X u|^\beta|\X\X u|^2\, dx \leq 
c (\beta+1)^{12} K_\eta
\int_{\supp(\eta)} (\delta+|\X u|^2)^\frac{p-2}{2} |\X u|^{\beta+2}\, dx
\end{equation}
for any $\beta\geq 2$, where $K_\eta =\| \eta\|_{L^\infty}^2+\| \X\eta\|_{L^\infty}^2+\|\eta
T\eta\|_{L^\infty}$. Letting $w = (\delta+|\X u|^2)^{p/2}$ and using a standard choice of test function $ \eta \in C^\infty_0(B_r) $ such that 
$ 0\leq \eta \leq 1$ and $\eta \equiv 1$ in $B_{r'}$ for $ 0<r'< r$, $|\X \eta|\leq 4/(r-r') $ and $|\XX \eta| \leq 16N/(r-r')^2,$
\eqref{eq:pestXu} implies 
\begin{equation}\label{eq:moser}
\int_{B_{r'}} w^\gamma  |\X w |^2\,dx \leq  
 \frac{c(\gamma+1)^{12}}{(r-r')^2}\int_{B_r} w^{\gamma+2}\,dx
\end{equation} 
for large enough $\gamma$. Using Sobolev's inequality \eqref{eq:sob emb} for $ q =2$
on \eqref{eq:moser}, we get that
$$ \left(\int_{B_{r'}}w^{(\gamma+2)\kappa}\,dx \right)^\frac{1}{\kappa}\leq  
\ \frac{c(\gamma+1)^{14}}{(r-r')^2}\int_{B_r} w^{\gamma+2}\,dx, $$ 
where $\kappa = Q/(Q-2)$. Standard Moser's iteration is carried out by taking an appropriate sequence $\gamma_i $ and 
$ r_i = \sigma r + (1-\sigma)r/2^i $ on the above, then the standard interpolation argument \cite[p. 299--300]{Dib-Tru} leads to 
\eqref{eq:aplocbound} and thereby completes the proof. 
\end{proof}

\subsection{Riemannian Approximation}
The sub-Riemannian metric generated by the vector fields $\{X_1,\ldots,X_m\}$ are approximated in the Gromov-Hausdorff sense via Riemannian approximation generated by $\{X_1^\eps,\ldots,X_{m+n}^\eps\}$, which is a relabelling 
of $\{X_1,\ldots,X_m, \eps T_1,\ldots, \eps T_n\}$, as $\eps\to 0^+$ and generates the norm $|\cdot|_\eps$ on $\R^{m+n}$. The second order vector fields remain $\{T_1^\eps,\ldots, T_n^\eps\}= \{T_1,\ldots, T_n\}$. In other words, the approximate vector fields in consideration are 
\begin{equation}\label{eq:appxt}
\{X_1^\eps,\ldots,X_{m+n}^\eps, T_1^\eps,\ldots, T_n^\eps\}=\{X_1,\ldots,X_m, \eps T_1,\ldots, \eps T_n, T_1,\ldots, T_n\}
\end{equation}
which is exactly the step $2$ case of the approximations provided in \cite{Cap-Citti-Rea, Cap-Citti}. The idea is to be able to carry out everything above with $X_j^\eps$ and $T_k^\eps$ in place of $X_j$ and $T_k$. The reason for the relabelling \eqref{eq:appxt} is that the equation with respect to $X_j^\eps$'s become automatically elliptic since 
$\{X_1,\ldots,X_m,  T_1,\ldots,  T_n\}$ spans the tangent space and hence, we can assume as much apriori regularity as needed. 
Then the conclusion shall follow with the limit $\eps\to 0^+$ provided that all constants do not blow up. We briefly illustrate in the following that this indeed is the case when the set up of the begining of Section \ref{sec:Notations and Preliminaries} is re-written in terms of the approximation \eqref{eq:appxt}. 

We denote the approximation gradients as  
$$\X^\eps u= (X_1^\eps u,\ldots,X_{m+n}^\eps u) = (X_1u,\ldots,X_mu, \eps T_1u,\ldots, \eps T_nu),
\quad\text{and}\quad T^\eps u= Tu$$
and $\dvh^\eps$ is similarly defined. 
Hence, $\X^\eps u \to (\X u, 0)$ and 
$|\X^\eps u|_\eps^2 = |\X u|^2 +\eps^2 |Tu|^2 \to |\X u|^2$ 
as $\eps\to 0^+$. Given $A: \R^m\to \R^m$ that satisfies \eqref{eq:pstr}, one can define 
$A_\eps:\R^{m+n}\to \R^{m+n}$ as 
\begin{equation}\label{Aeps}
A_\eps (\bar z) = \big(1-\eta_\eps(|\bar z|_\eps)\big) \tilde A(\bar z) + \eta_\eps (|\bar z|_\eps) 
(\delta+|\bar z|_\eps^2)^\frac{p-2}{2} \bar z
\end{equation}
similarly as \eqref{Adelta} for any $0<\eps<1$, where $\tilde A(\bar z)= (A(z),0)$ for any $\bar z=(z,z')\in \R^{m+n}$ 
and $\eta_\eps\in C^{0,1}([0,\infty))$ can be chosen (see \cite{Muk0}) such that $A_\eps \to A$ uniformly on compact subsets as $\eps \to 0^+$ and the structure condition 
\begin{equation}\label{eq:pstreps}
 \begin{aligned}
\frac{1}{L'}(\delta+|\bar z|_\eps^2)^\frac{p-2}{2}|\xi|^2 \leq \,&\inp{DA_\eps (\bar z)\,\xi}{\xi}\leq L'(\delta+|\bar z|_\eps^2)^\frac{p-2}{2}|\xi|^2;\\
 &|A_\eps(\bar z)|\leq L'\,|\bar z|_\eps (\delta+|\bar z|_\eps^2)^\frac{p-2}{2},
\end{aligned} 
\end{equation}
holds for any $\bar z,\xi\in \R^{m+n}$, for some $L'=L'(p,L)>1$. Since \eqref{eq:pstreps} ensures that the equation 
$\dvh^\eps ( A_\eps(\X^\eps u))=0$ is uniformly elliptic on $\Om$, hence it has all necessary regularities including 
\eqref{eq:ap reg}. 
As an example we see that 
\begin{equation}\label{eq:Hplapeps}
\dvh^\eps \big((\delta+|\X^\eps u|_\eps^2)^\frac{p-2}{2}\X^\eps u\big) = 0 \quad\text{in}\ \Om, 
\end{equation}
is a natural regularization of the sub-elliptic $p$-Laplacian \eqref{eq:Hplap}. 

From the step $2$ hypothesis \eqref{eq:step2}, note that $\{X_1,\ldots,X_m,  T_1,\ldots,  T_n\}$ spans the tangent space automatically implies $\{X_1,\ldots,X_m,  (1+\eps)T_1,\ldots,  (1+\eps)T_n\}$ spans the tangent space
for any $0\leq \eps \leq 1$ and therefore, any vector field can be written as linear combinations  of $\{X^\eps_1,\ldots,X^\eps_m,  T^\eps_1,\ldots,  T^\eps_n\}$ with bounded coefficients; in particular the commutators 
$[X^\eps_i,X^\eps_j]$ and $[X^\eps_i,T^\eps_j]$ satisfy commutation relations similar to \eqref{eq:comm} 
with uniformly bounded coefficients as in \eqref{eq:normbounds} independent of $\eps$. Hence, all constants in the following 
inequalities are independent of $\eps$ as well. 
But establishing sub-elliptic metrics similarly as in Section \ref{sec:Notations and Preliminaries} is non-trivial. Indeed, 
we have to similarly define 
$$\{Y^\eps_1,\ldots, Y^\eps _{m+2n}\}=\{X^\eps_1,\ldots, X^\eps_{m+n},T^\eps_1,\ldots,T^\eps_n\}$$ and define an approximate degree $d^\eps_j=\deg^\eps (Y_j^\eps)\in \{1,2\}$ so that $\deg^\eps (X^\eps_j)=1$ and $\deg^\eps(T^\eps_j)=2$ (as $\eps \to 0^+$ it is easy to see that $\deg^\eps \to \deg$ with $\deg$ as defined earlier in Section \ref{sec:Notations and Preliminaries}). As before, $Y^\eps_1,\ldots,Y^\eps_{m+2n}$ span the whole tangent space and 
$[Y^\eps_j,Y^\eps_k] \in \spn\{Y^\eps_l : d^\eps_l\leq d^\eps_j+d^\eps_k\}$ and we are in the set up of Nagel-Stein-Wainger \cite{Nag-Stein-Wain}. However, when we have 
$$ [Y^\eps_j,Y^\eps_k]\ = \sum_{d^\eps_l\leq d^\eps_j+d^\eps_k} c^l_{j,k}(\eps) Y^\eps_l,$$ 
the coefficients $c^l_{j,k}(\eps)$ will be unbounded as $\eps\to 0^+$. In principle
this could be a problem as the doubling constant for the metric balls in \cite{Nag-Stein-Wain} depends on these coefficients. In other words, letting $d^\eps(I)$ defined with respect to $d^\eps_j$'s and $ \Lambda^\eps (x, r) = 
\sum_I |\lambda^\eps_I(x)| r^{d^\eps(I)}$ 
where $\lambda^\eps_I(x)= \det(a^\eps_{j,k}(x))$ whenever $Y^\eps_{i_j}=\sum_{k}a^\eps_{j,k}\del_{x_k}$ as earlier, there exists $c_2(\eps)>c_1(\eps)>0$ such that we have 
$$ c_1(\eps)\Lambda^\eps (x, r) \leq |B^\eps_r(x)| \leq c_2(\eps) \Lambda^\eps (x, r)$$ 
due to \cite{Nag-Stein-Wain}, where $B^\eps_r\subset \Om$ are the metric balls defined similarly as in Section \ref{sec:Notations and Preliminaries}. 
Nevertheless, it has been shown by Capogna-Citti-Rea \cite[Proposition 4.4]{Cap-Citti-Rea} (see also \cite{Cap-Citti}) that there exists $0<\bar \eps<1$ such that for every $0<\eps<\bar\eps$, the constants $c_1(\eps)$ and $c_2(\eps)$
may be chosen to be independent of $\eps$ in compact subsets of $\Om$. The argument involves a very delicate use of exponentiation of sub-families of the vector fields depending on $\eps<r$ or $r<\eps<\bar\eps$. 

Combining all the above statements, we can conclude that the approximate vector fields \eqref{eq:appxt} has all the required properties to carry out the estimates of Section \ref{sec:apest}
for the equation $\dvh^\eps ( A_\eps(\X^\eps u))=0$ and the constants are independent of $\eps$. The constants of 
Poincar\'e and Sobolev inequalities are also stable under $\eps\to 0$, we refer to \cite{Cap-Citti}. 
Hence, we conclude from Theorem \ref{thm:aplip} that 
\begin{equation}\label{eq:aplocboundeps}
  \sup_{B^\eps_{\sigma r}}\ |\X^\eps u|_\eps  \leq \frac{c(N,p,L)}{(1-\sigma)^{Q/p}}\bigg(\intav_{B^\eps_r}\big(\delta+|\X^\eps u|_\eps^2\big)^\frac{p}{2}\dx \bigg)^\frac{1}{p}
 \end{equation}
holds for any $0<\sigma<1$. It is known that $B^\eps_r\to B_r$ in terms of Hausdorff distance as $\eps\to 0^+$, where $B_r\sub \Om$ is the concentric CC-ball (see \cite{Zhong--sub}). 
This shall be used in the following. 

\subsection{Proof of Theorem \ref{thm:mainthm1}} 
We are ready to prove Theorem \ref{thm:mainthm1} by a standard approximation argument and similar as in the Euclidean setting, that incorporates the approximation in order to remove the apriori regularity \eqref{eq:ap reg}. 
\begin{proof}[Proof of Theorem \ref{thm:mainthm1}]
Let $B_r = B_r(x_0) \subset \Om$ be fixed and $B^\eps_r$ be a sequence of metric balls with limit $B_r$ as the Hausdorfff distance limit. We shall denote similarly also for concentric balls of smaller radii. 
Given $ u \in \subsob^{1,p}(\Omega) $ as a weak solution of \eqref{eq:peq}, there exists 
a smooth approximation $\phi_k \in C^\infty(B_r)$ such that $\phi_k \to u$ in 
$\subsob^{1,p}(B_r)$ as $k\to\infty$. 
Now, let $u_k^\eps$ be the weak solution of the following Dirichlet problem,
\begin{equation}\label{eq:ADprob}
\begin{cases}
&\dvh^\eps (A_\eps(\X^\eps u_k^\eps))= \ 0 \quad  \text{ in } B^\eps_{\theta r}\\
& u_k^\eps-\phi_k\in \subsob^{1,p}_0(B^\eps_{\theta r}).
\end{cases}
\end{equation}
for some $\theta=\theta(N)<1$. 
The choice of test function $u_k^\eps-\phi_k$ on \eqref{eq:ADprob}, yields
\begin{equation}\label{eq:el1}
 \int_{B^\eps_{\theta r}}\inp{A_\eps(\X^\eps u_k^\eps)}{\X^\eps u_k^\eps}\dx 
=  \int_{B^\eps_{\theta r}}\inp{A_\eps(\X^\eps u_k^\eps)}{\X^\eps \phi_k}\dx.
\end{equation}
In this case, we have an ellipticity condition similar to \eqref{eq:elliptic}, which together with the structure condition \eqref{eq:pstreps} leads to 
\begin{equation}\label{eq:el2}
\begin{aligned}
 \int_{B^\eps_{\theta r}} \big(\delta+|\X^\eps u_k^\eps|_\eps^2\big)^\frac{p}{2}\dx
&\leq c\int_{B^\eps_{\theta r}} \big(\delta+|\X^\eps \phi_k|_\eps^2\big)^\frac{p}{2}\dx
\leq c\int_{B_{\theta r}} \big(\delta+|\X \phi_k|^2\big)^\frac{p}{2}\dx + o(\eps)\\
&\leq c\int_{B_{\theta r}} \big(\delta+|\X u|^2\big)^\frac{p}{2}\dx+o(\eps) + o(1/k)
\end{aligned}
\end{equation}
for $c = c(N,p,L) >0$ and $o(\eps), o(1/k) \to 0$ as $\eps\to 0^+$ and $k\to\infty$. 
Now, from \eqref{eq:aplocboundeps}, we have 
\begin{equation}\label{eq:epsG}
 \sup_{B^\eps_{\sigma \theta r}}\ |\X^\eps u_k^\eps|_\eps  \leq \frac{c}{(1-\sigma)^{Q/p}}\bigg(\intav_{B^\eps_{\theta r}}\big(\delta+|\X^\eps u_k^\eps|_\eps^2\big)^\frac{p}{2}\dx \bigg)^\frac{1}{p}
\end{equation}
for some $c = c(N,p,L) >0,\ \sigma\in(0,1)$. 
Then a standard argument follows, since \eqref{eq:el2} ensures 
that there exists $ u' \in \subsob^{1,p}(B_{\theta r})$ such that 
up to a subsequence $u_k^\eps \wto  u'$. 
As $u_k^\eps-\phi_k \in  \subsob^{1,p}_0(B^\eps_{\theta r})$, hence we have 
$u'-u\in \subsob^{1,p}_0(B_{\theta r})$. 
We have a monotonicity inequality similar to \eqref{eq:monotone}, using that one can show that $u'$ is a weak solution of \eqref{eq:peq}. From uniqueness, 
$u' = u$. Taking $\eps\to 0^+$ and $k\to\infty$ along the convergent subsequence in 
\eqref{eq:epsG} and \eqref{eq:el2}, we conclude 
$$ \sup_{B_{\sigma \theta r}}\ |\X u|  \leq \frac{c}{(1-\sigma)^{Q/p}}\bigg(\intav_{B_r}\big(\delta+|\X u|^2\big)^\frac{p}{2}\dx \bigg)^\frac{1}{p} $$
for some $c = c(N,p,L) >0$ 
and \eqref{eq:locbound} follows from the above by a simple covering argument. This, together with Remark \ref{rem:del0}, concludes the proof. 
\end{proof}

\section{The local $C^{1,\alpha}$ regularity}\label{sec:hold}
Throughout this section, we denote $u\in \subsob^{1,p}(\Omega)$ as a weak solution of equation \eqref{eq:peq} satisfying the structure condition \eqref{eq:pstr} with $\delta>0$. Equipped with Theorem \ref{thm:mainthm1}, we have $\X u \in L^\infty_\loc(\Om, \R^{m})$, which together with $\delta>0$, implies that
\begin{equation}\label{eq:str cond new}
 \begin{aligned}
 \inv{\nu}|\xi|^2 \leq \,&\inp{DA(\X u)\,\xi}{\xi}\leq \nu\,|\xi|^2;\\
 &|A(\X u)|\leq \nu\,|\X u|,
 \end{aligned}
\end{equation}
holds locally in $\Om$, 
for some $\nu = \nu(N,p, L, \|\X u\|_{L^\infty},\delta)>0$.
The equation $\dvh (A(\X u))=0$ with condition \eqref{eq:str cond new}, has been considered by Domokos-Manfredi \cite{Dom-Man--hor} (and previously by Capogna \cite{Cap--reg} for the case of the vector fields $X_1,\ldots, X_m$ being left invariant with respect to the Heisenberg Group). From \cite[Theorem 5.2]{Dom-Man--hor}, we have that
\begin{equation}\label{eq:doman reg}
\X u \in \subsob^{1,2}_\loc(\Om,\R^{m})\cap C^{0,\alpha}_\loc(\Om,\R^{m}) ,\qquad 
  Tu  \in \subsob^{1,2}_\loc(\Om,\R^n)\cap C^{0,\alpha}_\loc(\Om,\R^{n}),
\end{equation}
which is, in fact, stronger than \eqref{eq:ap reg}. Also, we resume the notation \eqref{eq:fp}, i.e. 
$\F(t) = (\delta+t^2)^\frac{p-2}{2}$ throughout this section and therefore all lemmas and apriori estimates of Section \ref{sec:apest} are also valid in the settings of this section. The local boundedness of $\X u$ allows us to define
\begin{equation}\label{def:mu}
\mu_i(r)=\sup_{{B_r}} \vert X_i u\vert, \quad \mu(r)=\max_{1\le i\le m}\mu_i(r), 
\end{equation}
for any ball $B_r\subset \Omega$ and $i\in\{1,\ldots, m\}$. 

Note that the function 
$ h(t)=(\delta+t^2)^{\frac{p-2}{2}} t^q = \F(t) t^q$
is non-decreasing on $[0,\infty)$ whenever $q\geq 0$ and 
$p+q-2\ge 0$ and hence, we have the inequality
\begin{equation}\label{trivial}
\F(|\X u|) |\X u|^q=
\big(\delta+\vert\X u\vert^2\big)^{\frac{p-2}{2}}\vert\X u\vert^q\le c(m,p,q)\big(\delta+\mu(r)^2\big)^{\frac{p-2}{2}} \mu(r)^q \quad \text{in }
B_r,
\end{equation}
where $c(m,p,q)=m^{(q+p-2)/2}$ if $p\ge 2$ and $c(m,p,q)=m^{q/2}$ if $1<p< 2$. Inequalities of the type \eqref{trivial}  shall be used multiple times in the local estimates. 

Similarly as \cite{Muk-Zhong}, we fix any $l\in \{ 1,\ldots , m\}$ and consider the following double 
truncation 
\begin{equation}\label{def:v}
\tvr:=\min\big(\mu(r)/8\,,\,\max\,(\mu(r)/4-X_lu,0)\big).
\end{equation}
In the setting of Euclidean spaces, similar 
truncations have been used previously in \cite{Tolk,Lieb--bound}, etc. 
The main property of $\tvr$ that is used in the integral estimates is the fact that 
it avoids any possible singularities within the set
\begin{equation}\label{eq:setE}
E=\{ x\in \Omega: \mu(r)/8< X_lu<\mu(r)/4\},
\end{equation}
as easily seen from the following inequality 
\begin{equation}\label{eq:comble1}
\mu(r)/8\le |\Xu|\le m^\frac{1}{2}\mu(r) \quad\text{in } E\cap B_r,
\end{equation}
which shall be used several times throughout this section. 
Furthermore, note that 
\begin{equation}\label{vis0}
\X \tvr=\begin{cases}
-\X X_lu \ & \text{a.e. in } E;\\
0 &\text{a.e. in } \Omega\setminus E,
\end{cases}
\quad\text{and}\quad 
T \tvr=\begin{cases}
-T X_l u \ & \text{a.e. in } E;\\
0 &\text{a.e. in } \Omega\setminus E,
\end{cases}
\end{equation}
which shall be exploited in the following integral estimates. The technique of involving the truncation in the Caccioppoli type estimates is useful particularly for the singular case $1<p<2$, whereas the required estimates can be obtained in easier ways for the case $p\geq 2$. 

Finally, we remark that properties similar to the above also holds corresponding to 
\begin{equation}\label{def:v'}
 \tvr'=\min\big(\mu(r)/8, \max(\mu(r)/4+X_lu,0)\big),
\end{equation}
and all following estimates invovling $\tvr$ shall also hold for $\tvr'$. Henceforth, we shall only be using the truncation $\tvr$ without loss of generality. 

\subsection{Caccioppoli type inequalities with truncation} 
We prove certain integral estimates of the gradient by testing \eqref{eq:yeqw1} with the truncation $\tvr$ being incorporated in the test function. 
Due to \eqref{eq:comble1}, these will eventually lead to an estimate similar to that of a uniformly elliptic equation without the singular weight. 

The following lemma is a truncated analogue of Lemma \ref{cacci:Xu}. The proof follows along the direction of \cite[Lemma 3.1]{Muk-Zhong} but it is lengthier due to the commutation relations \eqref{eq:comm}.
\begin{Lem}\label{lem:start}
For any $\beta, \kappa\ge 0$ and all non-negative $\eta\in C^\infty_0(\Omega)$, we have that
\begin{equation}\label{eq:cvineq1}
\begin{aligned}
\int_\Omega \eta^{\beta+2}\tvr^{\beta+2}&\weight |\X u|^\kappa|\X\X u|^2\, dx \\
&\leq c (\beta+1)^2\int_\Omega \eta^\beta \big(|\X \eta|^2+\eta| T\eta|\big)\tvr^{\beta+2}
\weight |\X u|^{\kappa+2} \, dx\\
&\ +  c(\beta+1)^2(\kappa+1)^2\int_\Omega \eta^{\beta+2} \tvr^ \beta\weight |\X u|^{\kappa+2}| \X \tvr|^2\, dx\\
&\ +c(\kappa+1)^4\int_\Omega\eta^{\beta+2} \tvr^{\beta+2}\weight |\X u|^\kappa \big(| Tu|^2+|\X u|^2\big)\, dx,
\end{aligned}
\end{equation}
where $\tvr$ is as in \eqref{def:v} and $c=c(N,p,L)>0$. 
\end{Lem}

\begin{proof}
For any $k\in \{1,\ldots, m\}$, we take $Y= X_k$ in \eqref{eq:yeqw2} to have 
\begin{equation}\label{eq:xleqw}
\begin{aligned}
\sum_{i,j}\int_\Om D_jA_i(\X u)X_jX_ku X_i\varphi\dx &= \sum_{i,j}\int_\Om D_jA_i(\X u)[X_j,X_k]u X_i\varphi\dx \\
&\quad + \sum_{i}  \int_\Om A_i(\X u) [X_i, X_k]\varphi\dx,
\end{aligned}
\end{equation}
and then, we choose $\varphi= \eta^{\beta+2}\tvr^{\beta+2} |\X u|^\kappa X_k u$ with a non-negative test function $\eta\in C^\infty_0(\Om)$ and $ \beta,\kappa\geq 0$
on \eqref{eq:xleqw} to obtain 
\begin{equation}\label{eq1}
\begin{aligned}
\int_\Omega \sum_{i,j} &\eta^{\beta+2}\tvr^{\beta+2}D_jA_i(\X u)X_jX_kuX_i\big( |\X u|^\kappa X_k u\big)\, dx\\
= & -(\beta+2)\int_\Omega\sum_{i,j} \eta^{\beta+1} \tvr^{\beta+2}|\X u|^\kappa X_ku D_jA_i(\X u)X_jX_kuX_i\eta\, dx\\
& -(\beta+2)\int_\Omega \sum_{i,j} \eta^{\beta+2} \tvr^{\beta+1}
|\X u|^\kappa X_k uD_jA_i(\X u)X_jX_kuX_i\tvr\, dx\\
& +\sum_{i,j}\int_\Om D_jA_i(\X u)[X_j,X_k]u X_i\big(\eta^{\beta+2}
\tvr^{\beta+2}|\X u|^\kappa X_ku\big)\, dx\\
& -\sum_i \int_\Omega  [X_i, X_k]\big(A_i(\X u)\big)\eta^{\beta+2}\tvr^{\beta+2}|\X u|^\kappa X_ku\, dx\\
= &\, I_1^k+I_2^k+I_3^k+I_4^k, 
\end{aligned}
\end{equation}
where integral by parts is applied to the last item in \eqref{eq:xleqw}. Then, by taking summation over $k\in \{1,\ldots, m\}$, 
we end up with
\begin{equation}\label{equ3}
\int_\Omega \sum_{i,j,k} \eta^{\beta+2}\tvr^{\beta+2}D_jA_i(\X u)X_jX_kuX_i\big( |\X u|^\kappa X_ku\big)\, dx=  \sum_k \sum_{j=1}^4 I_j^k,
\end{equation}
where $I_j^k$'s are the respective terms of the right hand side of \eqref{eq1}. 

In the following, we estimate both sides of \eqref{equ3}. For the left hand of \eqref{equ3}, first note that 
$ X_i\big(|\X u|^\kappa X_ku\big)=|\X u|^\kappa X_iX_k u+X_i(|\X u|^\kappa)X_ku$, hence by the structure condition \eqref{eq:pstr}, we have 
\[ \sum_{i,j,k}D_jA_i(\X u)X_jX_luX_i\big(|\X u|^\kappa X_ku\big)\ge 
\weight |\X u|^\kappa|\X\X u|^2,\]
which gives us the following estimate for the left hand side of \eqref{equ3} as 
\begin{equation}\label{est:left}
\text{left of } \eqref{equ3} \ge \int_\Omega \eta^{\beta+2}\tvr^{\beta+2}
\weight |\X u|^\kappa |\X\X u|^2\, dx.
\end{equation}
To estimate the right hand side of \eqref{equ3}, we shall show that
$I_j^k$ satisfies 
\begin{equation}\label{est1claim}
\begin{aligned}
| I_j^k| &\le  c\eps \int_\Omega \eta^{\beta+2}\tvr^{\beta+2}\weight |\X u|^\kappa |\X\X u|^2\,dx\\
 &\quad +\frac{c}{\eps}(\beta+1)^2\int_\Omega \eta^\beta \big(|\X \eta|^2+\eta| T\eta|\big) \tvr^{\beta+2}
\weight |\X u|^{\kappa+2} \, dx\\
&\quad +  \frac{c}{\eps}(\beta+1)^2(\kappa+1)^2\int_\Omega \eta^{\beta+2} \tvr^ \beta\weight |\X u|^{\kappa+2}| \X \tvr|^2\, dx\\
&\quad +\frac{c}{\eps}(\kappa+1)^4\int_\Omega\eta^{\beta+2} \tvr^{\beta+2}\weight |\X u|^\kappa \big(| Tu|^2+|\X u|^2\big)\, dx,
\end{aligned}
\end{equation}
for each $k\in \{1,\ldots, m\}$ and each $j=1,2,3,4$, 
where $c=c(N,p,L)>0$ and $0<\eps<1$ is small enough. Then \eqref{eq:cvineq1} follows from the above estimates 
\eqref{est:left} and \eqref{est1claim}
for both sides of \eqref{equ3}. Thus the proof of the lemma is finished, modulo
the proof of \eqref{est1claim}. In the rest, we prove \eqref{est1claim} in the order of $j=1, 2,3,4$.

First, for $j=1$, by the structure condition \eqref{eq:pstr}
and Young's inequality that
\begin{equation}\label{est:I1}
\begin{aligned}
| I_1^k| &\le c(\beta+1)\int_\Omega \eta^{\beta+1}| \X \eta| \tvr^{\beta+2}\weight |\X u|^{\kappa+1}|\X\X u|\, dx,\\
&\le c\eps \int_\Omega \eta^{\beta+2}\tvr^{\beta+2}
\weight |\X u|^\kappa|\X\X u|^2\,dx\\
 &\quad + \frac{c}{\eps}(\beta+1)^2\int_\Omega \eta^\beta |\X \eta|^2 \tvr^{\beta+2}
\weight |\X u|^{\kappa+2}\, dx.
\end{aligned}
\end{equation}
Thus \eqref{est1claim} holds for $I_1^k$. 

Second, when $j=2$, by the structure condition \eqref{eq:pstr} 
and Young's inequality, 
\begin{equation}\label{est:I2}
\begin{aligned}
| I_2^k|&\leq c(\beta+1)\int_\Omega\eta^{\beta+2}\tvr^{\beta+1}
\weight |\X u|^{\kappa+1}|\X\X u\|\X \tvr|\, dx,\\
&\le c\eps \int_\Omega \eta^{\beta+2}\tvr^{\beta+2}
\weight |\X u|^\kappa|\X\X u|^2\,dx\\
 &\quad + \frac{c}{\eps}(\beta+1)^2\int_\Omega \eta^{\beta+2} \tvr^{\beta}
\weight |\X u|^{\kappa+2}|\X \tvr|^2\, dx,
\end{aligned}
\end{equation}
which proves \eqref{est1claim} for $I_2^k$.

Third, when $j=3$, we note that
\begin{equation*}
\begin{aligned}
\big| X_i  \big(  \eta^{\beta+2}\tvr^{\beta+2}|\X u|^\kappa X_ku\big)\big| &\le  (\kappa+1) \eta^{\beta+2} \tvr^{\beta+2} |\X u|^\kappa |\X\X u| + (\beta+2)\eta^{\beta+1}\tvr^{\beta+2}
|\X u|^{\kappa+1}|\X \eta| \\
&\quad +(\beta+2)\eta^{\beta+2}\tvr^{\beta+1}
|\X u|^{\kappa+1}| \X \tvr|.
\end{aligned}
\end{equation*}
Thus by the structure condition \eqref{eq:pstr} and \eqref{eq:comm}, we have 
\begin{equation*}
\begin{aligned}
| I_3^k| \le &\,  c(\kappa+1) \int_\Omega \eta^{\beta+2}\tvr^{\beta+2}\weight |\X u|^\kappa|\X \X u| \big(| Tu|+|\X u|\big)\, dx\\
&\ + c(\beta+1)\int_\Omega\eta^{\beta+1}|\X \eta| \tvr^{\beta+2}\weight |\X u|^{\kappa+1}\big(| Tu|+|\X u|\big)\, dx\\
&\ + c(\beta+1)\int_\Omega \eta^{\beta+2} \tvr^{\beta+1}\weight |\X u|^{\kappa+1}|\X \tvr| \big(| Tu|+|\X u|\big)\, dx,
\end{aligned}
\end{equation*}
from which, by Young's inequality on each term, we get
\begin{equation}\label{est:I3}
\begin{aligned}
| I_3^k| \le &\, c\eps\int_\Omega \eta^{\beta+2}\tvr^{\beta+2}
\weight |\X u|^\kappa |\X\X u|^2\,dx\\
&\ +\frac{c}{\eps}(\kappa+1)^2 \int_\Omega \eta^{\beta+2}\tvr^{\beta+2}\weight |\X u|^\kappa 
\big(| Tu|^2+|\X u|^2\big)\, dx\\
&\ + c(\beta+1)^2 \int_\Omega\eta^{\beta}|\X \eta|^2 \tvr^{\beta+2}\weight |\X u|^{\kappa+2}\, dx\\
&\ + c(\beta+1)^2\int_\Omega \eta^{\beta+2} \tvr^{\beta}\weight |\X u|^{\kappa+2}|\X \tvr|^2\, dx.
\end{aligned}
\end{equation}
This proves \eqref{est1claim} for $I_3^k$.

Finally, we are left with the case $j=4$ and the estimate is more involved. 
Recalling our choice of test function $\varphi= \eta^{\beta+2}\tvr^{\beta+2} |\X u|^\kappa X_k u$, we note that 
\begin{equation}\label{est:I41}
\begin{aligned}
I_4^k=& -\sum_i \int_\Omega  [X_i, X_k]\big(A_i(\X u)\big)\varphi \dx\\
=& -\sum_i \int_\Om X_iX_k\big(A_i(\X u)\big)\varphi \dx + \sum_i \int_\Om X_kX_i\big(A_i(\X u)\big)\varphi \dx\\
=&\sum_i \int_\Omega \Big(X_k\big(A_i(\X u)\big) X_{i}\varphi-X_{i}\big(A_i(\X u)\big)X_k\varphi \Big)\, dx,
\end{aligned}
\end{equation}
where integration by parts is used to get the last term of the above. 
Let us denote
\begin{equation}\label{def:w} 
w=\eta^{\beta+2}|\X u|^\kappa X_ku.
\end{equation}
so that we can rewrite $\varphi = \tvr^{\beta+2}w$ and hence, we have 
$\X \varphi= (\beta+2)\tvr^{\beta+1}w\X \tvr+ \tvr^{\beta+2}\X w$. 
Using this, note that \eqref{est:I41} becomes
\begin{equation}\label{est:I42}
\begin{aligned}
I_4^k &= (\beta+2)\int_\Omega \tvr^{\beta+1} w
\sum_i\Big(X_k\big(A_i(\X u)\big) X_{i}\tvr-X_{i}\big(A_i(\X u)\big)X_k\tvr \Big)\, dx\\
&\qquad +\int_\Omega \tvr^{\beta+2}\sum_i\Big(X_k\big(A_i(\X u)\big) X_{i}w-X_{i}\big(A_i(\X u)\big)X_kw \Big)\, dx\\
&= R^k + P^k,
\end{aligned}
\end{equation}
where we denote the first and the second integral in the right hand side of 
the above by $R^k$ and $P^k$, respectively. We shall estimate both of them one by one. 
First, by the structure condition \eqref{eq:pstr}, the
definition of $w$ as in (\ref{def:w}),
and Young's inequality, we have 
\begin{equation}\label{est:Jl}
\begin{aligned}
| R^k| &\le c(\beta+1)\int_\Omega \eta^{\beta+2} \tvr^{\beta+1}
\weight |\X u|^{\kappa+1}| \X\X u\|\X \tvr|\, dx\\
&\le  c\eps\int_\Omega \eta^{\beta+2}\tvr^{\beta+2}\weight |\X u|^\kappa|\X\X u|^2\,dx\\
&\quad +\frac{c}{\eps}(\beta+1)^2\int_\Omega \eta^{\beta+2} \tvr^{\beta}
\weight |\X u|^{\kappa+2}|\X \tvr|^2\, dx.
\end{aligned}
\end{equation}
To estimate $P^k$, we use
integration by parts again to get 
\begin{equation}\label{est:Kl}
\begin{aligned}
P^k=&\, (\beta+2) \int_\Omega \tvr^{\beta+1} \sum_i A_i(\X u)\Big( X_{i} \tvr X_k w-X_k\tvr X_{i}w\Big)\, dx\\
&\quad + \int_\Omega \tvr^{\beta+2}\sum_i A_i(\X u) [X_i, X_k]w\, dx\\
= &\, P^k_1+P^k_2.
\end{aligned}
\end{equation}
For $P^k_1$, we have by the structure condition \eqref{eq:pstr} that
\begin{equation*}
\begin{aligned}
| P^k_1| \le\ &  c(\beta+1)(\kappa+1)\int_\Omega \eta^{\beta+2}\tvr^{\beta+1}\weight |\X u|^{\kappa+1}|\X\X u| |\X \tvr|\, dx\\
& +c(\beta+1)^2\int_\Omega\eta^{\beta+1}\tvr^{\beta+1}\weight |\X u|^{\kappa+2}|\X \tvr | |\X \eta|\, dx
\end{aligned}
\end{equation*}
from which it follows by Young's inequality that
\begin{equation}\label{est:Kl1}
\begin{aligned}
| P^k_1| \le &\, c\eps \int_\Omega \eta^{\beta+2}\tvr^{\beta+2}\weight |\X u|^\kappa|\X\X u|^2\,dx\\
&\quad +\frac{c}{\eps}(\beta+1)^2(\kappa+1)^2\int_\Omega \eta^{\beta+2} \tvr^{\beta}
\weight |\X u|^{\kappa+2}|\X \tvr|^2\, dx\\
&\quad + c(\beta+1)^2\int_\Omega \eta^{\beta}|\X \eta|^2 \tvr^{\beta+2}\weight |\X u|^{\kappa+2}\, dx.
\end{aligned}
\end{equation}
To estimate $P^k_2$, we note that
\begin{align*}
[X_i,X_k]w&=  \eta^{\beta+2} |\X u|^\kappa [X_i, X_k]X_ku + \kappa \eta^{\beta+2} |\X u|^{\kappa-1}X_ku 
[X_i, X_k](|\X u|)\\
&\quad + (\beta+2)\eta^{\beta+1} |\X u|^\kappa X_ku [X_i, X_k]\eta,
\end{align*}
and hence, we rewrite 
\begin{equation}\label{pk}
\begin{aligned}
P^k_2 &= \int_\Omega  \eta^{\beta+2}\tvr^{\beta+2}\sum_i A_i(\X u) |\X u|^\kappa [X_i, X_k]X_ku\, dx\\
&\quad +\kappa\int_\Omega \eta^{\beta+2}\tvr^{\beta+2}\sum_i A_i(\X u)  |\X u|^{\kappa-1}X_ku 
[X_i, X_k](|\X u|)\, dx \\
&\quad + (\beta+2)\int_\Omega \eta^{\beta+1} \tvr^{\beta+2}\sum_i A_i(\X u) |\X u|^\kappa X_ku [X_i, X_k]\eta\, dx\\
&= P^k_{2,1}+ P^k_{2,2}+ P^k_{2,3}, 
\end{aligned}
\end{equation}
which is estimated similarly as in the last part of the proof of Lemma \ref{cacci:Xu}. The estimate of $P^k_{2,3}$ is straightforward. Using \eqref{eq:str} and \eqref{eq:comm}, we have
\begin{equation}\label{pk3}
|P^k_{2,3}| \leq c(\beta+1)\int_\Om \eta^{\beta+1} (| T\eta|+|\X\eta|)\tvr^{\beta+2} \F(|\X u|) |\X u|^{\kappa+2} \dx,
\end{equation}
which is clearly majorised by the right hand side of \eqref{est1claim}.

To estimate the rest of \eqref{pk}, we shall use $[X_i, X_k]X_ku= X_k[X_i, X_k]u + [[X_i,X_k],X_k]u$ and 
$$ [X_i, X_k](|\X u|)= |\X u|^{-1}\sum_{j=1}^m X_ju [X_i, X_k]X_ju =|\X u|^{-1}\sum_{j=1}^m X_ju \big(X_j[X_i, X_k]u+[[X_i,X_k],X_j]u\big),$$
to split the first two terms as $P^k_{2,1}=P^k_{2,1,1}+P^k_{2,1,2} $ and $P^k_{2,2}=P^k_{2,2,1}+P^k_{2,2,2}$ so that the second terms i.e. $P^k_{2,1,2}$ and $P^k_{2,2,2}$ contain the terms with the commutator of commutators.

Now we rewrite the first terms of $ P^k_{2,1}+ P^k_{2,2}$ using integral by parts as 
\begin{equation}\label{pk21121}
\begin{aligned}
P^k_{2,1,1}+ P^k_{2,2,1} &=  \sum_{i}  \int_\Omega  \eta^{\beta+2}\tvr^{\beta+2} A_i(\X u) |\X u|^\kappa X_k[X_i, X_k]u\, dx\\
&\qquad  +\kappa\sum_{i,j} \int_\Omega \eta^{\beta+2}\tvr^{\beta+2} A_i(\X u)  |\X u|^{\kappa-2}X_ku 
X_ju X_j[X_i, X_k]u\, dx\\
&= - \sum_{i} \int_\Om X_k\Big(\eta^{\beta+2}\tvr^{\beta+2} A_i(\X u) |\X u|^\kappa\Big) [X_i, X_k]u \dx\\
&\qquad  -\kappa\sum_{i,j}  \int_\Om X_j\Big(\eta^{\beta+2}\tvr^{\beta+2} A_i(\X u)  |\X u|^{\kappa-2}X_ku 
X_ju \Big)[X_i, X_k]u\dx. 
\end{aligned}
\end{equation}
Then, further computation of the terms in \eqref{pk21121} followed by the use of \eqref{eq:str} and \eqref{eq:comm} yields
\begin{align*}
|P^k_{2,1,1}| + | P^k_{2,2,1}| &\leq c(\kappa+1)^2 
\int_\Om\eta^{\beta+2}\tvr^{\beta+2}  \F(|\X u|) |\X u|^{\kappa} |\X \X u| \big(|\X u|+|Tu| \big)\dx\\
&\quad + c(\beta+1)(\kappa+1) \int_\Om \eta^{\beta+2}\tvr^{\beta+1}  |\X u|^{\kappa +1} |\X\tvr| \big(|\X u|+|Tu| \big)\dx\\
&\quad + c(\beta+1)(\kappa+1)  \int_\Om \eta^{\beta+1}\tvr^{\beta+2}  |\X u|^{\kappa +1} |\X\eta| \big(|\X u|+|Tu| \big)\dx,
\end{align*}
and then, by Young's inequality on each term, we get 
\begin{equation}\label{pk2fstest}
\begin{aligned}
|P^k_{2,1,1}| + | P^k_{2,2,1}|
&\leq c\eps \int_\Om\eta^{\beta+2}\tvr^{\beta+2}  \F(|\X u|) |\X u|^{\kappa} |\X \X u|^2\dx \\
&\quad + \frac{c}{\eps} (\kappa+1)^4\int_\Omega\eta^{\beta+2}\tvr^{\beta+2}  \F(|\X u|) |\X u|^{\kappa} \big(|\X u|^2+|Tu|^2 \big)\, dx\\
&\quad  + c(\beta+1)^2 \int_\Omega \eta^{\beta+2} \tvr^{\beta}
\weight |\X u|^{\kappa+2}|\X \tvr|^2\, dx\\
&\quad  + c(\beta+1)^2 \int_\Omega\eta^{\beta}|\X \eta|^2 \tvr^{\beta+2}\weight |\X u|^{\kappa+2}\, dx
\end{aligned}
\end{equation}
 Now we are only left with
\begin{equation}\label{pk2nd}
\begin{aligned}
P^k_{2,1,2}+ P^k_{2,2,2}&=  \sum_{i}  \int_\Omega  \eta^{\beta+2}\tvr^{\beta+2} A_i(\X u) |\X u|^\kappa [[X_i,X_k],X_k]u\, dx\\
&\quad  +\kappa\sum_{i,j} \int_\Omega \eta^{\beta+2}\tvr^{\beta+2} A_i(\X u)  |\X u|^{\kappa-2}X_ku 
X_ju [[X_i,X_k],X_j]u\, dx\\
\end{aligned}
\end{equation}
To estimate \eqref{pk2nd}, we use \eqref{eq:str} and \eqref{eq:XXX} to get
\begin{equation}\label{pk2ndest}
\begin{aligned}
|P^k_{2,1,2}|+| P^k_{2,2,2}| \leq c(\kappa+1) \int_\Om \eta^{\beta+2}\tvr^{\beta+2} \F(|\X u|) |\X u|^{\kappa+1}\big( |\X u| + |Tu|\big)\dx,
\end{aligned}
\end{equation}
which is estimated easily to get the right hand side of \eqref{est1claim}. Thus, combining the estimates above we can see that the claimed estimate \eqref{est1claim} holds also for $I^k_4$ and the proof is finished. 
\end{proof}

The following lemma is a truncated analogue of Corollary \ref{cor:est0 XTu}. 

\begin{Lem}\label{lem:tech}
Let $B_r \subset \Om$ and $\eta\in C^\infty_0(B_r)$ be fixed.  
Then, for any $\tau \in (\frac{1}{2},1)$ and $\beta \geq 0$, 
\begin{equation}\label{eq:tech}
 \begin{aligned}
  \int_\Omega \eta^{\tau(\beta+2)+4}&\tvr^{\tau(\beta+4)} 
\weight |\X u|^4|\X Tu|^2\, dx\\
 &\leq  c_0 \int_\Omega \eta^{\tau(\beta+2)+4}\tvr^{\tau(\beta+4)} 
\weight |\X u|^4 \big(|Tu|^2+|\X u|^2\big)\, dx\\
&\qquad + c\mu(r)^{4(1-\tau)}K_\eta^{2-\tau}\Big(\int_{B_r} \weight |\X u|^\frac{2}{1-\tau}\dx\Big)^{1-\tau}\mathcal{I}^\tau
 \end{aligned}
\end{equation}
holds for any small enough $c_0=c_0(N,p,L,\tau)>0$, where $c= c(N,p,L,\tau)>0$ is large enough, 
$K_\eta =\| \eta\|_{L^\infty}^2+\| \X\eta\|_{L^\infty}^2+\|\eta
T\eta\|_{L^\infty}$, and 
\begin{equation}\label{def:I}
\begin{aligned}
\mathcal{I}=&  \,c (\beta+1)^2\int_\Omega \eta^{\beta+2}\big(\eta^2 + |\X \eta|^2+\eta| T\eta|\big)
\tvr^{\beta+4}
\weight |\X u|^4\, dx\\
& +  c(\beta+1)^2\int_\Omega \eta^{\beta+4} \tvr^{\beta+2}\weight |\X u|^4
| \X \tvr|^2\, dx\\
& +
c\int_\Omega\eta^{\beta+4} \tvr^{\beta+4}\weight |\X u|^2 | Tu|^2\, dx.
\end{aligned}
\end{equation}
\end{Lem}

\begin{proof}
Let us denote the left hand side of \eqref{eq:tech} by $M$, i.e. 
\begin{equation}\label{def:M}
 M=\ \int_\Omega \eta^{\tau(\beta+2)+4}\,\tvr^{\tau(\beta+4)} 
\weight |\X u|^4|\X Tu|^2\, dx,
\end{equation}
for any fixed $1/2<\tau<1$ and $\beta\geq 0$. Now, we recall \eqref{eq:Teq}, i.e. 
\begin{equation}\label{eq:Teq1}
\begin{aligned}
\sum_{i,j}\int_\Om D_jA_i(\X u)X_jT_ku X_i\varphi\dx &= \sum_{i,j}\int_\Om D_jA_i(\X u)[X_j,T_k]u X_i\varphi\dx \\
&\quad - \sum_{i,j}\int_\Om  \varphi D_jA_i(\X u) [X_i, T_k]X_j u\dx.
\end{aligned}
\end{equation}
We use $\varphi=\varphi_k = \eta^{\tau(\beta+2)+4}\,\tvr^{\tau(\beta+4)}|\X u|^4\, T_ku$ on \eqref{eq:Teq1} and take summation over $k\in\{1,\ldots,n\}$, 
 to obtain
\begin{equation}\label{est:Ki}
 \begin{aligned}
 \sum_{i,j,k}&\int_\Omega \eta^{\tau(\beta+2)+4}\,\tvr^{\tau(\beta+4)} 
|\X u|^4 D_jA_i(\X u)X_jT_ku\,X_iT_ku\, dx\\
&=-(\tau(\beta+2)+4)\sum_{i,j,k} \int_\Omega  \eta^{\tau(\beta+2)+3}\,\tvr^{\tau(\beta+4)} 
|\X u|^4 T_ku D_jA_i(\X u)X_jT_ku\,X_i\eta\dx\\
&\quad -\tau(\beta+4)\sum_{i,j,k}\int_\Omega \eta^{\tau(\beta+2)+4}
\,\tvr^{\tau(\beta+4)-1} 
|\X u|^4 T_ku D_jA_i(\X u)X_jT_ku\,X_i \tvr\dx\\
&\quad -4\sum_{i,j,k,l}\int_\Omega \eta^{\tau(\beta+2)+4}\,\tvr^{\tau(\beta+4)} 
|\X u|^2 X_lu T_ku D_jA_i(\X u)X_jT_ku\,X_iX_lu\dx\\
&\quad + \sum_{i,j,k}\int_\Om D_jA_i(\X u)[X_j,T_k]u X_i\varphi_k\dx \\
&\quad  - \sum_{i,j,k}\int_\Om  \varphi_k D_jA_i(\X u) [X_i, T_k]X_j u\dx\\
&= K_1 +K_2+ K_3+ K_4+K_5,
\end{aligned}
\end{equation}
where the integrals in the right hand side of \eqref{est:Ki} are denoted by
$K_1,K_2,K_3,K_4,K_5$ in order. To prove the lemma, 
we estimate both sides of \eqref{est:Ki} as follows. 

For the left hand side, 
we have by the structure condition \eqref{eq:pstr} that
\begin{equation}\label{est:K}
 \text{left of \eqref{est:Ki}} \geq \int_\Omega \eta^{\tau(\beta+2)+4}\,\tvr^{\tau(\beta+4)} 
\weight |\X u|^4|\X Tu|^2\, dx = M, 
\end{equation}
and for the right hand side of \eqref{est:Ki}, we estimate each item $K_i$, one by one. 
Let us denote 
\begin{equation}\label{def:tilde K}
\tilde K = \int_\Om \eta^{(2\tau-1)(\beta+2)+6} 
\,\tvr^{(2\tau-1)(\beta+4)} \weight |\X u|^4 |Tu|^2 |\X Tu|^2\dx,
\end{equation} 
which shall appear frequently in the following. 

First, we estimate $K_1$ by the structure condition \eqref{eq:pstr} and 
H\"older's inequality, to get
\begin{equation}\label{est:K1}
\begin{aligned}
\vert K_1\vert & \leq c(\beta+1)\int_\Omega \eta^{\tau(\beta+2)+3}\,\tvr^{\tau(\beta+4)} 
\weight |\X u|^4|Tu||\X Tu||\X\eta|\, dx\\
&\leq c(\beta+1)\Big(\int_\Om \eta^{(2\tau-1)(\beta+2)+6} 
\,\tvr^{(2\tau-1)(\beta+4)} \weight |\X u|^4 |Tu|^2 |\X Tu|^2\dx \Big)^\frac{1}{2}\\
&\qquad\qquad \times  \Big(\int_\Om \eta^{\beta+2} \tvr^{\beta+4} 
\weight |\X u|^4 |\X\eta|^2\dx\Big)^\frac{1}{2}\\ 
&= c(\beta+1) \tilde K^\frac{1}{2} 
\Big(\int_\Om \eta^{\beta+2} \tvr^{\beta+4} 
\weight |\X u|^4 |\X\eta|^2\dx\Big)^\frac{1}{2}.
\end{aligned}
\end{equation}
Second, we estimate $K_2$ also by the structure condition \eqref{eq:pstr} and 
H\"older's inequality,  
\begin{equation}\label{est:K2}
\begin{aligned}
 \vert K_2\vert & \leq c(\beta+1)\int_\Omega \eta^{\tau(\beta+2)+4}\,\tvr^{\tau(\beta+4)-1} 
\weight |\X u|^4|Tu||\X Tu||\X \tvr|dx\\
&\leq c(\beta+1)\Big(\int_\Om \eta^{(2\tau-1)(\beta+2)+6} 
\,\tvr^{(2\tau-1)(\beta+4)} \weight |\X u|^4 |Tu|^2 |\X Tu|^2\dx \Big)^\frac{1}{2}\\
&\qquad\qquad \times \Big(\int_\Om \eta^{\beta+4} \tvr^{\beta+2} 
\weight |\X u|^4 |\X \tvr|^2\dx\Big)^\frac{1}{2}\\
 &= c(\beta+1) \tilde K^\frac{1}{2} 
\Big(\int_\Om \eta^{\beta+4} \tvr^{\beta+2} 
\weight |\X u|^4 |\X \tvr|^2\dx\Big)^\frac{1}{2}.
\end{aligned}
\end{equation}
Similarly, we estimate $K_3$ by 
structure condition \eqref{eq:pstr} and 
H\"older's inequality. We have 
\begin{equation}\label{est:K3}
\begin{aligned}
\vert K_3\vert &\leq c\int_\Omega \eta^{\tau(\beta+2)+4}\,\tvr^{\tau(\beta+4)} 
\weight |\X u|^3|Tu||\X Tu||\X\X u|\, dx\\
&\leq c\,\Big(\int_\Om \eta^{(2\tau-1)(\beta+2)+6} 
\,\tvr^{(2\tau-1)(\beta+4)} \weight |\X u|^4 |Tu|^2 |\X Tu|^2\dx \Big)^\frac{1}{2}\\
&\qquad \times \Big(\int_\Om \eta^{\beta+4} \tvr^{\beta+4} 
\weight |\X u|^2 |\X\X u|^2\dx\Big)^\frac{1}{2}\\
&= c \tilde K^\frac{1}{2} 
\Big(\int_\Om \eta^{\beta+4} \tvr^{\beta+4} 
\weight |\X u|^2 |\X\X u|^2\dx\Big)^\frac{1}{2} \leq c\tilde K^\frac{1}{2}\mathcal{I}^\frac{1}{2},
\end{aligned}
\end{equation}
where the last inequality of the above follows from Lemma \ref{lem:start} with
\begin{equation*}
\begin{aligned}
\mathcal{I}=&  \,c (\beta+1)^2\int_\Omega \eta^{\beta+2}\big(\eta^2 + |\X \eta|^2+\eta| T\eta|\big)
\tvr^{\beta+4}
\weight |\X u|^4\, dx\\
& +  c(\beta+1)^2\int_\Omega \eta^{\beta+4} \tvr^{\beta+2}\weight |\X u|^4
| \X \tvr|^2\, dx\\
& +
c\int_\Omega\eta^{\beta+4} \tvr^{\beta+4}\weight |\X u|^2 | Tu|^2\, dx, 
\end{aligned}
\end{equation*}
which is a rearrangement of the right hand side of \eqref{eq:cvineq1} (with $\kappa=2$ and $\beta\mapsto \beta+2$). 

In fact,  we can combine \eqref{est:K1}, \eqref{est:K2} and \eqref{est:K3} to obtain 
\begin{equation}\label{eq:k123}
\vert K_1\vert+\vert K_2\vert+\vert K_3\vert \leq c \tilde K^\frac{1}{2} \mathcal{I}^\frac{1}{2}. 
\end{equation}
Then $\tilde K$ is estimated by H\"older's inequality as 
\begin{equation}\label{est:tilde K}
 \begin{aligned}
  \tilde K&= \int_\Om \eta^{(2\tau-1)(\beta+2)+6} 
\,\tvr^{(2\tau-1)(\beta+4)} \weight |\X u|^4 |Tu|^2 |\X Tu|^2\dx\\
&\leq  \Big(\int_\Omega\eta^{\tau(\beta+2)+4}\,\tvr^{\tau(\beta+4)} 
\weight |\X u|^4|\X Tu|^2\, dx\Big)^{\frac{2\tau-1}{\tau}}\\
&\quad \times\Big(\int_\Omega\eta^{\frac{2\tau}{1-\tau}+4}
\weight |\X u|^4 |Tu|^{\frac{2\tau}{1-\tau}}
|\X Tu|^2\dx\Big)^\frac{1-\tau}{\tau}\\
&= M^{\frac{2\tau-1}{\tau}} H^\frac{1-\tau}{\tau}, 
 \end{aligned}
\end{equation}
where $M$ is as in \eqref{def:M} and $H$ is the second integral on the right hand side of (\ref{est:tilde K}), i.e. 
\begin{equation}\label{def:H}
 H=\int_\Omega\eta^{\frac{2\tau}{1-\tau}+4}
\weight |\X u|^4 |Tu|^{\frac{2\tau}{1-\tau}}
|\X Tu|^2\dx, 
\end{equation}
which is estimated using Corollary \ref{cor:est XTu} and \eqref{trivial} with $q = 2/(1-\tau)$, as 
\begin{equation}\label{est:H}
\begin{aligned}
H\leq  c\mu(r)^4\int_\Omega \eta^{q+2} \weight | Tu|^{q-2}|\X Tu|^2\, dx\le c K_\eta^\frac{q+2}{2} \mu(r)^4\int_{B_r}\weight |\X u|^q \, dx 
\end{aligned}
\end{equation}
where $K_\eta =\| \eta\|_{L^\infty}^2+\| \X\eta\|_{L^\infty}^2+\|\eta
T\eta\|_{L^\infty}$ and $c=c(N,p,L,\tau)>0$. 

To estimate the rest, we note that
\begin{equation}\label{est:K4}
 \begin{aligned}
K_4&= \sum_{i,j,k}\int_\Omega \eta^{\tau(\beta+2)+4}\,\tvr^{\tau(\beta+4)} 
|\X u|^4 D_jA_i(\X u)[X_j,T_k]u\,X_iT_ku\, dx\\
&\ +(\tau(\beta+2)+4)\sum_{i,j,k} \int_\Omega  \eta^{\tau(\beta+2)+3}\,\tvr^{\tau(\beta+4)} 
|\X u|^4 T_ku D_jA_i(\X u)[X_j,T_k]u\,X_i\eta\dx\\
&\ +\tau(\beta+4)\sum_{i,j,k}\int_\Omega \eta^{\tau(\beta+2)+4}
\,\tvr^{\tau(\beta+4)-1} 
|\X u|^4 T_ku D_jA_i(\X u)[X_j,T_k]u\,X_i \tvr\dx\\
&\ +4\sum_{i,j,k,l}\int_\Omega \eta^{\tau(\beta+2)+4}\,\tvr^{\tau(\beta+4)} 
|\X u|^2 X_lu T_ku D_jA_i(\X u)[X_j,T_k]u\,X_iX_lu\dx\\
&=K_{4,0}+K_{4,1}+K_{4,2}+K_{4,3}.
\end{aligned}
\end{equation}
Then $K_{4,0}$ is estimated by \eqref{eq:pstr}, \eqref{eq:comm} and 
H\"older's inequality as 
\begin{equation*}
\begin{aligned}
\vert K_{4,0}\vert & \leq c(\beta+1)\int_\Omega \eta^{\tau(\beta+2)+4}\,\tvr^{\tau(\beta+4)} 
\weight |\X u|^4\big(|Tu|+|\X u|\big)|\X Tu|\, dx\\
&\leq c(\beta+1)\Big(\int_\Om \eta^{(2\tau-1)(\beta+2)+6} 
\,\tvr^{(2\tau-1)(\beta+4)} \weight |\X u|^4 \big(|Tu|^2+|\X u|^2\big) |\X Tu|^2\dx \Big)^\frac{1}{2}\\
&\qquad\qquad \times  \Big(\int_\Om \eta^{\beta+4} \tvr^{\beta+4} 
\weight |\X u|^4 \dx\Big)^\frac{1}{2}\\ 
&= c(\beta+1) \bar K^\frac{1}{2} 
\Big(\int_\Om \eta^{\beta+4} \tvr^{\beta+4} 
\weight |\X u|^4 \dx\Big)^\frac{1}{2} \leq c\bar K^\frac{1}{2} \mathcal I^\frac{1}{2}, 
\end{aligned}
\end{equation*}
where $\bar K$ is the term similar to but larger than $\tilde K$ and is estimated similarly as 
\begin{equation}\label{est:bar K}
 \begin{aligned}
  \bar K&= \int_\Om \eta^{(2\tau-1)(\beta+2)+6} 
\,\tvr^{(2\tau-1)(\beta+4)} \weight |\X u|^4 \big(|Tu|^2+|\X u|^2\big) |\X Tu|^2\dx\\
&\leq  c\Big(\int_\Omega\eta^{\tau(\beta+2)+4}\tvr^{\tau(\beta+4)} 
\weight |\X u|^4|\X Tu|^2\, dx\Big)^{\frac{2\tau-1}{\tau}}\\
&\quad \times \Big(\int_\Omega\eta^{\frac{2\tau}{1-\tau}+4}
\weight |\X u|^4 \big(|Tu|^{\frac{2\tau}{1-\tau}}+|\X u|^{\frac{2\tau}{1-\tau}}\big)
|\X Tu|^2\dx\Big)^\frac{1-\tau}{\tau}\\
&= cM^{\frac{2\tau-1}{\tau}} (H+H')^\frac{1-\tau}{\tau}, 
 \end{aligned}
\end{equation}
for some $c=c(N,p,L,\tau)>0$, $H$ is in \eqref{def:H} and 
\begin{equation}\label{def:H'}
 H'=\int_\Omega\eta^{\frac{2\tau}{1-\tau}+4}
\weight |\X u|^{4+\frac{2\tau}{1-\tau}} |\X Tu|^2\dx,
\end{equation}
which is estimated using Corollary \ref{cor:est0 XTu} and \eqref{trivial} with $q = 2/(1-\tau)$, as 
\begin{equation}\label{est:H'}
\begin{aligned}
H'&\leq  c\mu(r)^4\int_\Omega \eta^{q+2} \weight |\X u|^{q-2}|\X Tu|^2\, dx
\le c K_\eta^\frac{q+2}{2} \mu(r)^4\int_{B_r}\weight |\X u|^q \, dx 
\end{aligned}
\end{equation}
where $K_\eta =\| \eta\|_{L^\infty}^2+\| \X\eta\|_{L^\infty}^2+\|\eta
T\eta\|_{L^\infty}$ and $c=c(N,p,L,\tau)>0$. 

Since we have $|[X_j, T_k]u|^2 \leq c\big(|Tu|^2+|\X u|^2\big)$ from \eqref{eq:comm}, 
it is not hard to see that the estimates of $K_{4,1},K_{4,2},K_{4,3}$ can be performed by H\"older's inequality in exactly the same way as that of $K_1,K_2,K_3$, i.e. \eqref{est:K1},\eqref{est:K2},\eqref{est:K3} respectively, where $|\X T u|^2$ being replaced by $\big(|Tu|^2+|\X u|^2\big)$. In other words, we have 
\begin{equation}\label{eq:k123}
|K_{4,1}| +|K_{4,2}| +|K_{4,3}| \leq c \mathcal{K}^\frac{1}{2} \mathcal{I}^\frac{1}{2}, 
\end{equation}
 where $\mathcal{I}$ is as in \eqref{def:I} and 
\begin{equation}\label{def:cal K}
\mathcal K = \int_\Om \eta^{(2\tau-1)(\beta+2)+6} 
\,\tvr^{(2\tau-1)(\beta+4)} \weight |\X u|^4 |Tu|^2 \big(|Tu|^2+|\X u|^2\big) \dx, 
\end{equation} 
which can be estimated similarly as above to obtain
\begin{equation}\label{est:cal K}
 \begin{aligned}
  \mathcal K
&\leq  c\Big(\int_\Omega\eta^{\tau(\beta+2)+4}\tvr^{\tau(\beta+4)} 
\weight |\X u|^4| Tu|^2\, dx\Big)^{\frac{2\tau-1}{\tau}}\\
&\quad \times \Big(\int_\Omega\eta^{\frac{2\tau}{1-\tau}+4}
\weight |\X u|^4 \big(|Tu|^{\frac{2\tau}{1-\tau}}+|\X u|^{\frac{2\tau}{1-\tau}}\big)
| Tu|^2\dx\Big)^\frac{1-\tau}{\tau}\\
&= c\bar M^{\frac{2\tau-1}{\tau}} \bar H^\frac{1-\tau}{\tau}, 
 \end{aligned}
\end{equation}
 where $\bar H$ is the term analogous to $H$ and $H'$ as in \eqref{def:H} and \eqref{def:H'} which can be estimated 
similarly by \eqref{eq:est Tu} and 
\begin{equation}\label{barM}
\bar M = \int_\Omega\eta^{\tau(\beta+2)+4}\,\tvr^{\tau(\beta+4)} 
\weight |\X u|^4| Tu|^2\, dx.  
\end{equation}
Finally, we recall that 
$$ K_5= - \sum_{i,j,k}\int_\Om  \eta^{\tau(\beta+2)+4}\,\tvr^{\tau(\beta+4)}|\X u|^4\, T_ku D_jA_i(\X u) [X_i, T_k]X_j u\dx,$$ 
which is estimated by \eqref{eq:pstr} and \eqref{eq:XTX} as 
\begin{equation}\label{est:K5}
 \begin{aligned}
 |K_5| &\leq c\int_\Om  \eta^{\tau(\beta+2)+4}\,\tvr^{\tau(\beta+4)}\weight |\X u|^4 |Tu|  
\big(|\XX u|+|\X Tu| + |\X u| + |Tu|\big)\dx\\
&\leq c (\mathcal{K}_0^\frac{1}{2}+\tilde K^\frac{1}{2}+\bar K^\frac{1}{2}+\mathcal{K}^\frac{1}{2}) \mathcal{I}^\frac{1}{2}, 
 \end{aligned}
\end{equation}
where the latter inequality is obtained from H\"older's inequality and the previous estimates, where $\tilde K, \bar K$ and $\mathcal K$ are as in \eqref{est:bar K} and \eqref{def:cal K}, $\mathcal{I}$ is as in \eqref{def:I} and 
the first term of \eqref{est:K5} estimated similarly as  
\begin{equation}\label{def:tilde K0}
\mathcal K_0 = \int_\Om \eta^{(2\tau-1)(\beta+2)+6} 
\,\tvr^{(2\tau-1)(\beta+4)} \weight |\X u|^8 \dx \leq  c M_0^{\frac{2\tau-1}{\tau}} H_0^\frac{1-\tau}{\tau},
\end{equation} 
where $M_0$ and $H_0$ are the analogous terms containing only $\X u$ in place of $Tu$ or $\X Tu$. 
Combining all the estimates of the above, we finally have 
\begin{equation}
\begin{aligned}
M &\leq c \big(\tilde K^\frac{1}{2}+\bar K^\frac{1}{2}+\mathcal{K}^\frac{1}{2} + \mathcal{K}_0^\frac{1}{2}\big) 
\mathcal{I}^\frac{1}{2}\\
 &\leq c\Big(M^{\frac{2\tau-1}{2\tau}}+\bar M^{\frac{2\tau-1}{2\tau}}+M_0^{\frac{2\tau-1}{2\tau}}\Big)
\big(H+H'+\bar H+H_0\big)^\frac{1-\tau}{2\tau}\mathcal{I}^\frac{1}{2}
\end{aligned}
\end{equation}
which, by Young's inequality, leads to 
$$ M \leq c_0(\bar M+M_0) + c\big(H+H'+\bar H+H_0\big)^{1-\tau}\mathcal I^\tau, $$
which, with the estimates \eqref{est:H} and \eqref{est:H'}, implies \eqref{eq:tech} 
and the proof is finished.
\end{proof}
\begin{Rem}\label{rem:groupcase2}
For the special case when $X_1,\ldots,X_m$ are left invariant with respect to a step $2$ Carnot Group, since we have $ [X_i,T_k]=0$, hence it corresponds to the case of $c_0=0$ for \eqref{eq:tech}. However, it does not make any difference because of the following corollary. 
\end{Rem}

\begin{Cor}\label{cor:tech}
Let $B_r \subset \Om$ and $\eta\in C^\infty_0(B_r)$ be fixed.  
Then, for any $\tau \in (\frac{1}{2},1)$ and $\beta \geq 0$, 
\begin{equation*}
 \begin{aligned}
  \int_\Omega \eta^{\tau(\beta+2)+4}\tvr^{\tau(\beta+4)} 
\weight |\X u|^4|\X Tu|^2\, dx
\leq c\mu(r)^{4(1-\tau)}K_\eta^{2-\tau}\Big(\int_{B_r} \weight |\X u|^\frac{2}{1-\tau}\dx\Big)^{1-\tau}\mathcal{I}^\tau
 \end{aligned}
\end{equation*}
where $\mathcal{I}$ is as in \eqref{def:I}, $K_\eta =\| \eta\|_{L^\infty}^2+\| \X\eta\|_{L^\infty}^2+\|\eta
T\eta\|_{L^\infty}$ and
$c= c(N,p,L,\tau)>0$.
\end{Cor}
\begin{proof}
The first term in the right hand side of \eqref{eq:tech} can be estimated by H\"older's inequality as 
\begin{equation}\label{eq:tech'}
 \begin{aligned}
  \int_\Omega &\eta^{\tau(\beta+2)+4}\tvr^{\tau(\beta+4)} 
\weight |\X u|^4 \big(|Tu|^2+|\X u|^2\big)\, dx \\
&\leq \Big(\int_\Omega\eta^{\beta+4} \tvr^{\beta+4}\weight |\X u|^2 \big(|Tu|^2+|\X u|^2\big)\, dx\Big)^\tau \\
&\ \times \Big(\int_\Om \eta^{2+\frac{2}{1-\tau}} \weight |\X u|^{2+\frac{2}{1-\tau}}\big(|Tu|^2+|\X u|^2\big)\dx \Big)^{1-\tau} \leq c\, \mathcal I^\tau (\bar H+H_0)^{1-\tau},
 \end{aligned}
\end{equation}
where $\mathcal{I}$ is as in \eqref{def:I} and $\bar H, H_0$ are as in the proof of Lemma \ref{lem:tech}, which can be estimated easily by \eqref{eq:est Tu} as earlier. This concludes the proof. 
\end{proof}

\subsection{The main Lemma}
We are ready to  prove the main lemma for the proof of the $C^{1,\alpha}$ regularity. This is a Caccioppoli type estimate of the truncation that is devoid of the weight $\weight = (\delta+|\X u|^2)^\frac{p-2}{2}$ which is similar to estimates for equations with uniform ellipticity. 

Towards this, we fix $B_r\subset \Om$ 
and use a standard test function $\eta\in C^\infty_0(B_r)$ that satisfies 
\begin{equation}\label{eq:propeta}
\begin{aligned}
&0\leq \eta \leq 1\quad\text{in}\ B_r,\ \eta \equiv 1 \quad\text{in}\ B_{r/2},\\
&|\X \eta| \leq 4/r \quad\text{and}\quad |\XX \eta | \leq 16N/r^2\quad\text{in} \ B_r,
\end{aligned}
\end{equation}
and hence from \eqref{eq:ineqT}, $K_\eta =\| \eta\|_{L^\infty}^2+\| \X\eta\|_{L^\infty}^2+\|\eta
T\eta\|_{L^\infty}\leq c/r^2$ for large enough $c=c(N)>0$. 

The following is the main lemma of this section, which has been proved previously in \cite{Muk-Zhong} for the special case when $X_1,\ldots,X_m$ are left invariant with respect to the Heisenberg Group. In the Euclidean case, similar estimates have been proved earlier, see \cite{Tolk, Lieb--bound}, etc. 
\begin{Lem}\label{lem:main}
Let $\tvr$ be the truncation as in \eqref{def:v} and $\eta\in C^\infty_0(B_r)$ satisfy \eqref{eq:propeta}. 
Then we have the inequality 
 \begin{equation}\label{eq:mainest}
  \int_{B_r} \eta^{\beta+4}\tvr^{\beta+2}|\X \tvr|^2\, dx \leq 
\frac{c}{r^2}(\beta+1)^2\mu(r)^4 | B_r|^{1-1/\gamma}\Big(\int_{B_r}\eta^{\gamma\beta}\tvr^{\gamma\beta}\, dx\Big)^{1/\gamma},
\end{equation}
for any $\beta\ge 0$ and $\gamma>1$, where $c=c(N,p, L,\gamma)>0$ 
is a constant.
\end{Lem}

\begin{proof}
We shall assume $1<\gamma<3/2$ without loss of generality, since \eqref{eq:mainest} can be extended to the full range by applying H\"older's inequality to the right hand side. 

Recall that $\tvr=\min\big(\mu(r)/8\,,\,\max\,(\mu(r)/4-X_lu,0)\big)$ as in \eqref{def:v}, for some fixed $l\in \{1,\ldots,m\}$. 
For this $l$, we test \eqref{eq:yeqwXl} with the function $\varphi=\eta^{\beta+4}\tvr^{\beta+3}$ for $\beta\geq 0$ and $\eta\in C^\infty_0(B_r)$ as given, which is allowed since it follows from \eqref{eq:doman reg} and \eqref{vis0} that
$\X \tvr\in L^2_\loc(\Om;\R^{m})$ and $ T\tvr\in L^2_\loc(\Om; \R^n)$. Thus, we obtain 
\begin{equation}\label{lemest2}
\begin{aligned}
-(\beta+3)\sum_{i,j}&\int_\Omega \eta^{\beta+4}\tvr^{\beta+2} 
D_jA_i(\X u)X_jX_lu X_i\tvr\, dx\\
& = (\beta+4)\sum_{i,j}\int_\Omega \eta^{\beta+3}\tvr^{\beta+3} 
D_jA_i(\X u)X_jX_lu X_i\eta\, dx\\
&\ \ - (\beta+4)\sum_{i,j}\int_\Omega \eta^{\beta+3}\tvr^{\beta+3} 
D_jA_i(\X u)[X_j,X_l]u X_i\eta\, dx\\
&\ \ -(\beta+3)\sum_{i,j} \int_\Omega \eta^{\beta+4} \tvr^{\beta+2} 
D_jA_i(\X u)[X_j,X_l]u X_i\tvr \dx\\
&\ \ -\sum_{i}  \int_\Om A_i(\X u) [X_i, X_l](\eta^{\beta+4}\tvr^{\beta+3})\dx. 
\end{aligned}
\end{equation}
We can combine the first two terms of the right hand side and rewrite \eqref{lemest2} as 
\begin{equation}\label{lemest3}
\begin{aligned}
-(\beta+3)\sum_{i,j}&\int_\Omega \eta^{\beta+4}\tvr^{\beta+2} 
D_jA_i(\X u)X_jX_lu X_i\tvr\, dx\\
& = (\beta+4)\sum_{i,j}\int_\Omega \eta^{\beta+3}\tvr^{\beta+3} 
D_jA_i(\X u)X_lX_ju X_i\eta\, dx\\
&\ \ -(\beta+3)\sum_{i,j} \int_\Omega \eta^{\beta+4} \tvr^{\beta+2} 
D_jA_i(\X u)[X_j,X_l]u X_i\tvr \dx\\
&\ \ -\sum_{i}  \int_\Om A_i(\X u) [X_i, X_l](\eta^{\beta+4}\tvr^{\beta+3})\dx\\
&= I_1+I_2+I_3,
\end{aligned}
\end{equation}
where we denote the terms in the right hand side of 
\eqref{lemest3} by $I_1,I_2,I_3$, respectively and 
we estimate both sides 
of \eqref{lemest3} in the following. 

For the left hand side, from \eqref{vis0}, the structure condition \eqref{eq:pstr} and \eqref{eq:comble1}, we have 
\begin{equation}\label{est:lemleft}
\begin{aligned}
\text{left of } \eqref{lemest3} & \ge (\beta+3)\int_E \eta^{\beta+4}
\tvr^{\beta+2}\weight |\X \tvr|^2\, dx\\
& \ge c_0(\beta+1)\F(\mu(r))\int_{B_r}\eta^{\beta+4}\tvr^{\beta+2}|\X \tvr|^2\, dx,
\end{aligned}
\end{equation}
for some constant $c_0=c_0(N,p,L)>0$, where $E$ as in \eqref{eq:setE}. 
For the right hand side of (\ref{lemest3}), we claim that
each item $I_1,I_2,I_3$ satisfies  
\begin{equation}\label{lem:claim}
\begin{aligned}
| I_k| \le & \frac{c_0}{6}(\beta+1)\F(\mu(r))\int_{B_r}\eta^{\beta+4}\tvr^{\beta+2}|\X \tvr|^2\, dx\\
&\ +\frac{c}{r^2}(\beta+1)^3
  \F(\mu(r))\mu(r)^4 | B_r|^{1-1/\gamma}\Big(\int_{B_r}\eta^{\gamma\beta}\tvr^{\gamma\beta}\, dx\Big)^{1/\gamma},
\end{aligned}
\end{equation}
where $k=1,2,3$, $1<\gamma<3/2$ and  $c=c(N,p,L,\gamma)>0$. Then \eqref{eq:mainest} follows from the estimate \eqref{est:lemleft} for the
left hand side of \eqref{lemest3}) and the above claim
\eqref{lem:claim}) for each item in the right, thereby completing the proof. 
Thus, we are only left with 
proving the claim \eqref{lem:claim}. 
To this end, we estimate each $I_k$ one by one, in the rest of the proof. 

First, from chain rule, notice that 
\begin{align*}
I_1&= (\beta+4)\sum_{i}\int_\Omega \eta^{\beta+3}\tvr^{\beta+3} 
X_l\big(A_i(\X u)\big) X_i\eta\, dx\\
&=- (\beta+4)\sum_{i}\int_\Omega A_i(\X u) X_l\big( \eta^{\beta+3}\tvr^{\beta+3} X_i\eta\big)\dx,
\end{align*}
where the latter equality is due to integral by parts. Then by structure condition \eqref{eq:pstr}, 
\begin{equation}\label{lemest4}
\begin{aligned}
| I_1|\le &\, c(\beta+1)^2\int_\Omega \eta^{\beta+2} \tvr^{\beta+3}\weight |\X u|\big(|\X\eta|^2+\eta|\X\X\eta|\big)\, dx\\
&+ c(\beta+1)^2\int_\Omega \eta^{\beta+3}\tvr^{\beta+2}\weight |\X u||\X \tvr\|\X\eta|\, dx\\
\le & \, \frac{c}{r^2}(\beta+1)^2\F(\mu(r))\mu(r)^4\int_{B_r} \eta^\beta
\tvr^\beta\, dx\\
&\, +\frac{c}{r}(\beta+1)^2\F(\mu(r))
\mu(r)^2\int_{B_r}\eta^{\beta+2}\tvr^{\beta+1}|\X \tvr|\, dx,
\end{aligned}
\end{equation}
where $c=c(N,p,L)>0$. For the latter inequality of \eqref{lemest4}, we have used \eqref{trivial}. 
Now we apply Young's inequality to the last term of  
\eqref{lemest4} to end up with 
\begin{equation}\label{lemest:I1}
\begin{aligned}
| I_1|\le &\, \frac{c_0}{6} (\beta+1)\F(\mu(r))
\int_{B_r}\eta^{\beta+4}\tvr^{\beta+2}|\X \tvr|^2\, dx\\
&\, +\frac{c}{r^2}(\beta+1)^3\F(\mu(r))\mu(r)^4\int_{B_r}\eta^\beta \tvr^\beta\, dx,
\end{aligned}
\end{equation}
where $c=c(N,p,L)>0$ and $c_0$ is the same constant as in \eqref{est:lemleft}.
The claimed estimate \eqref{lem:claim} for $I_1$, follows from the 
above estimate \eqref{lemest:I1} and H\"older's inequality.

Second, we estimate $I_2$ by structure condition \eqref{eq:pstr} and \eqref{eq:comm}, to have 
$$ | I_2| \le c(\beta+1) \int_\Omega \eta^{\beta+4} \tvr^{\beta+2} \weight 
| \X \tvr| \big(|Tu|+|\X u|)\, dx,$$
and it is important to note that the above integral is supported on $E$ from \eqref{vis0}. Therefore, it follows by H\"older's inequality that
\begin{equation}\label{lemest5}
\begin{aligned}
| I_2| &\le c(\beta+1) \Big(\int_E \eta^{\beta+4}\tvr^{\beta+2}\weight|\X \tvr|^2\, dx\Big)^{\frac{1}{2}}
\Big(\int_E \eta^{\gamma(\beta+2)}\tvr^{\gamma(\beta+2)}\weight \, dx\Big)^{\frac{1}{2\gamma}}\\
& \qquad \qquad \times\Big(\int_\Omega \eta^q \,\weight\big(|Tu|^q+|\X u|^q)\, dx\Big)^{\frac{1}{q}},
\end{aligned}
\end{equation}
where $q=2\gamma/(\gamma-1)$. The fact that the first two integrals are on $E$ is crucial as \eqref{eq:comble1} can be exploited to carry out 
the following estimates, 
\begin{align}
\label{lemest6}
\int_E \eta^{\beta+4}\tvr^{\beta+2}\weight|\X \tvr|^2\, dx
 \le c \F(\mu(r)) \int_{B_r} \eta^{\beta+4}
\tvr^{\beta+2}|\X \tvr|^2\, dx, 
\end{align}
and 
\begin{align}
\label{lemest7}
\int_E \eta^{\gamma(\beta+2)}\tvr^{\gamma(\beta+2)}\weight \, dx\le c \F(\mu(r))\mu(r)^{2\gamma}\int_{B_r} \eta^{\gamma\beta} \tvr^{\gamma\beta}\, dx,
\end{align}
which cannot be done otherwise (unless $t\mapsto \F(t)$ is increasing which corresponds to $p\geq 2$ for the $p$-Laplacian). The last integral of \eqref{lemest5} is estimated from \eqref{eq:est Tu} and \eqref{trivial} to get
\begin{equation}\label{lemest8}
\begin{aligned}
\int_\Omega \eta^q \,\weight \big(|Tu|^q+|\X u|^q)\, dx& \le \frac{c}{r^q}\int_{B_r}\weight |\X u|^q\, dx\le \frac{c}{r^q}
| B_r|\F(\mu(r))\mu(r)^q,
\end{aligned}
\end{equation}
for some $c=c(N,p,L,\gamma)>0$. Combining the above three estimates \eqref{lemest6}, \eqref{lemest7} and \eqref{lemest8} for the three integrals in \eqref{lemest5} respectively, we end up with 
\begin{equation*}
| I_2| \le \frac{c}{r}(\beta+1)\F(\mu(r))
\mu(r)^2 | B_r|^{\frac{\gamma-1}{2\gamma}}\Big(\int_{B_r}\eta^{\beta+4}\tvr^{\beta+2}|\X \tvr|^2\, dx\Big)^{\frac{1}{2}}
\Big(\int_{B_r}\eta^{\gamma\beta}\tvr^{\gamma\beta}\, dx\Big)^{\frac{1}{2\gamma}},
\end{equation*}
from which, together with Young's inequality, the claim 
\eqref{lem:claim} for $I_2$ follows.

Finally, for $I_3$, we have 
\begin{equation}\label{lemest10}
\begin{aligned}
I_3 =&-\sum_{i}  \int_\Om A_i(\X u) [X_i, X_l](\eta^{\beta+4}\tvr^{\beta+3})\dx\\
=&-(\beta+4)\sum_i\int_\Omega \eta^{\beta+3}\tvr^{\beta+3} A_i(\X u) [X_i, X_l]\eta\, dx\\
&\ -(\beta+3)\sum_i\int_\Omega \eta^{\beta+4}\tvr^{\beta+2} A_i(\X u) [X_i, X_l]\tvr\, dx
= I_3^1+I_3^2,
\end{aligned}
\end{equation}
where we denote the last two integrals in the above equality by $I_3^1$ and $I_3^2$, respectively. The estimate for $I_3^1$ easily follows from the structure condition \eqref{eq:pstr} and \eqref{trivial} as 
\begin{equation}\label{est:I31}
\begin{aligned}
| I_3^1|&\le  \, c(\beta+1) \int_\Omega \eta^{\beta+3}\tvr^{\beta+3}
\weight |\X u|\big(|T\eta|+|\X\eta|\big)\, dx\\
&\le  \,\frac{c}{r^2}(\beta+1)\F(\mu(r))\mu(r)^4\int_{B_r}\eta^\beta
\tvr^\beta\, dx,
\end{aligned}
\end{equation}
and by H\"older's inequality, $I_3^1$ satisfies estimate \eqref{lem:claim}. To estimate $I_3^2$, we use the structure condition \eqref{eq:pstr},\eqref{eq:comm} and 
\eqref{vis0} to get
\begin{equation}\label{lemest11}
\begin{aligned}
| I_3^2|&\le c(\beta+1)\int_E \eta^{\beta+4}\tvr^{\beta+2}\weight |\X u|\big(|\X\tvr|+|T\tvr|\big)\, dx,\\
&\leq c(\beta+1)\int_E \eta^{\beta+4}\tvr^{\beta+2}\weight |\X u|\big(|\XX u|+|\X T u| +|\X u|+|Tu|\big)\, dx. 
\end{aligned}
\end{equation}
As before, it is crucial that the above integral is supported on $E$ of \eqref{eq:setE} since we require more weight of $|\X u|$ to estimate \eqref{lemest11}, which can be endowed by virtue of \eqref{eq:comble1}. First, we note that 
by H\"older's inequality, we have 
\[\begin{aligned} | I_3^2|\le c&(\beta+1)\Big(\int_E\eta^{\gamma(\beta+2)}\tvr^{\gamma\beta+4(\gamma-1)}\weight\, dx\Big)^{\frac 1 2}\\
& \times \Big(\int_E \eta^{(2-\gamma)(\beta+2)+4}\tvr^{(2-\gamma)(\beta+4)} \weight |\X u|^2
\big(|\XX u|^2+|\X T u|^2 +|\X u|^2+|Tu|^2\big)\, dx\Big)^{\frac 1 2}.
\end{aligned}\]
Then, using \eqref{eq:comble1}, we can obtain
\begin{equation}\label{est:I32}
| I_3^2| \le c(\beta+1)\F(\mu(r))^\frac{1}{2}\mu(r)^{2(\gamma-1)-1} \Big(\int_{B_r} \eta^{\gamma\beta}
\tvr^{\gamma\beta}\, dx\Big)^{\frac 1 2}(M'+M+ \bar M+M_0)^\frac{1}{2}
\end{equation}
where, we denote 
\begin{equation}\label{def:K new}
 M= \int_\Omega \eta^{(2-\gamma)(\beta+2)+4}\,\tvr^{(2-\gamma)(\beta+4)} 
 |\X u|^4 \weight |\X Tu|^2\, dx,
\end{equation}
and $M', \bar M, M_0$ are the other analogous terms containing $|\XX u|^2, |Tu|^2$ and $|\X u|^2$ respectively, in place of $|\X T u|^2$ in \eqref{def:K new}. Note that $M, \bar M$ and $M_0$ are the same 
terms as defined in the proof of Lemma \ref{lem:tech}, with $\tau=2-\gamma$ when $1<\gamma<3/2$. Then, $M$ is estimated by Lemma \ref{lem:tech} and Corollary \ref{cor:tech}, $M'+\bar M+M_0$ is estimated by  H\"older's inequality exactly as in \eqref{eq:tech'} in the proof of Corollary \ref{cor:tech} (to estimate the term corresponding to $M'$, Lemma \ref{lem:start} and \eqref{eq:est Xu} are further used) and hence, we have 
\begin{equation}\label{estm}
M'+M+ \bar M+M_0 \leq  \frac{c}{r^{2\gamma}}\mu(r)^{4(\gamma-1)}\Big(\int_{B_r} \weight |\X u|^\frac{2}{\gamma-1}\dx\Big)^{\gamma-1}\mathcal{I}^{2-\gamma}
\end{equation}
for some $c=c(N,p,L,\gamma)>0$, where $\mathcal{I}$ is as in \eqref{def:I}, i.e. 
\begin{equation*}
\begin{aligned}
\mathcal{I}=&  \,c (\beta+1)^2\int_\Omega \eta^{\beta+2}\big(\eta^2 + |\X \eta|^2+\eta| T\eta|\big)
\tvr^{\beta+4}
\weight |\X u|^4\, dx\\
& +  c(\beta+1)^2\int_\Omega \eta^{\beta+4} \tvr^{\beta+2}\weight |\X u|^4
| \X \tvr|^2\, dx\\
& +
c\int_\Omega\eta^{\beta+4} \tvr^{\beta+4}\weight |\X u|^2 | Tu|^2\, dx.
\end{aligned}
\end{equation*}
It is evident that each term of $\mathcal I$ are estimable from \eqref{trivial}. Indeed, the first term of $\mathcal I$ is estimated by H\"older's inequality and \eqref{trivial}
to get 
\begin{equation}\label{I1}
\begin{aligned}
\int_\Omega \eta^{\beta+2}
\tvr^{\beta+4}
&\weight |\X u|^4\big(\eta^2 + |\X \eta|^2+\eta| T\eta|\big)\, dx\\
& \leq \frac{c}{r^2}\F(\mu(r))\mu(r)^8
|B_r|^{1-\frac{1}{\gamma}}
\Big(\int_{B_r} \eta^{\gamma\beta}\tvr^{\gamma\beta}\dx\Big)^\frac{1}{\gamma}.
\end{aligned}
\end{equation}
For the second term of $\mathcal I$, we have 
\begin{equation}\label{I2}
\begin{aligned}
\int_\Omega \eta^{\beta+4} \tvr^{\beta+2}\weight |\X u|^4
| \X \tvr|^2\, dx\leq  c\F(\mu(r))\mu(r)^4
\int_{B_r} \eta^{\beta+4} \tvr^{\beta+2}| \X \tvr|^2\, dx.
\end{aligned}
\end{equation}
For the third term of $\mathcal I$, we use H\"older's inequality, \eqref{eq:est Tu} and \eqref{trivial} to obtain
\begin{equation}\label{I3}
\begin{aligned}
\int_\Omega 
\eta^{\beta+4} \tvr^{\beta+4}
&\weight |\X u|^2|Tu|^2\dx\\
&\leq \Big(\int_\Om \eta^\frac{2\gamma}{\gamma-1}\weight |\X u|^2
|Tu|^\frac{2\gamma}{\gamma-1}\dx\Big)^{1-\frac{1}{\gamma}}\\
&\quad \times\Big( \int_\Om \eta^{\gamma(\beta+2)}\tvr^{\gamma(\beta+4)}
\weight |\X u|^2 \dx\Big)^\frac{1}{\gamma}\\
&\leq \frac{c}{r^2}\F(\mu(r))\mu(r)^8
|B_r|^{1-\frac{1}{\gamma}}
\Big(\int_{B_r} \eta^{\gamma\beta}\tvr^{\gamma\beta}\dx\Big)^\frac{1}{\gamma}
\end{aligned}
\end{equation}
for some $c=c(N,p,L,\gamma)>0$. Combining \eqref{I1}, \eqref{I2} and \eqref{I3}, we get
 \begin{equation}\label{est:I}
\begin{aligned}
\mathcal I \leq c(\beta+1)^2 \F(\mu(r))\mu(r)^4 \mathcal J
\end{aligned}
\end{equation}
where we denote 
\begin{equation}\label{def:J new}
\mathcal J = \int_{B_r} \eta^{\beta+4} \tvr^{\beta+2} |\X \tvr|^2\dx 
+ \frac{\mu(r)^4}{r^2} |B_r|^{1-\frac{1}{\gamma}}
\Big(\int_{B_r} \eta^{\gamma\beta}\tvr^{\gamma\beta}\dx\Big)^\frac{1}{\gamma}.
\end{equation}
Using \eqref{est:I} and \eqref{trivial} on \eqref{estm}, we obtain
\begin{equation*}
M'+M+ \bar M+M_0 \leq  \frac{c}{r^{2\gamma}} (\beta+1)^{2(2-\gamma)}\F(\mu(r))
 \mu(r)^6 |B_r|^{\gamma-1} \mathcal J^{2-\gamma}, 
\end{equation*}
which is further used in \eqref{est:I32} to obtain
$$ |I^2_3|\leq \frac{c}{r^\gamma}(\beta+1)^{3-\gamma} 
\F(\mu(r))\mu(r)^{2\gamma}\,
\mathcal J^\frac{2-\gamma}{2}|B_r|^\frac{\gamma-1}{2}
\Big(\int_{B_r} 
\eta^{\gamma\beta}\tvr^{\gamma\beta}\dx\Big)^\frac{1}{2}.$$
Then, by Young's inequality, we end up with
\begin{equation*}
\begin{aligned}
|I^2_3|\leq \frac{c_0}{12}(\beta+1)\F(\mu(r)) \mathcal J + \frac{c}{r^2}(\beta+1)^{\frac{4}{\gamma}-1}
 \F(\mu(r))\mu(r)^4
|B_r|^{1-\frac{1}{\gamma}}
\Big(\int_{B_r} \eta^{\gamma\beta}\tvr^{\gamma\beta}\dx\Big)^\frac{1}{\gamma},
\end{aligned}
\end{equation*}
where $c_0>0$ is the same constant as in \eqref{lem:claim} and $\mathcal J$ as in \eqref{def:J new}. Thus, 
$I_3^2$ satisfies 
an estimate similar to \eqref{lem:claim} and hence the claim \eqref{lem:claim}
for $I_3$ follows, since both $I_3^1$ and $I_3^2$ satisfy similar estimates.  
This concludes the proof of the claim \eqref{lem:claim}, and hence the proof of the lemma.
\end{proof}

The following corollary follows from Lemma \ref{lem:main} by using Sobolev's inequality \eqref{eq:sob emb} on \eqref{eq:mainest} and carrying out Moser's iteration. We refer to \cite[Corollary 3.5]{Muk-Zhong} for the proof. In the Euclidean setting, similar statements have been proved earlier in \cite{Dib, Tolk}, etc.
\begin{Cor}\label{prop:case1}
There exists a constant $\theta = \theta(N,p,L)>0$ such that the following holds for any ball $B_r\subset\Omega$. If we have 
\begin{equation}\label{condition1}
 \vert\{x\in B_r : X_lu<\mu(r)/4\} \vert\leq \theta |B_r|
\end{equation}
for an index $l\in\{1,\ldots,m\}$, then 
 \[\inf_{B_{r/2}}X_lu\ge 3\mu(r)/16.\]
Analogously, if we have 
\begin{equation}\label{condition2}
\vert\{x\in B_r : X_lu>-\mu(r)/4\}\vert\leq \theta |B_r|,
\end{equation}
for an index $l\in\{1,\ldots,m\}$, then  
\[\sup_{B_{r/2}}X_lu \le\, -3\mu(r)/16.\]
\end{Cor}

\subsection{De Giorgi's method} It is evident that, equipped with Lemma \ref{lem:main} and Corollary \ref{prop:case1}, the problem is reduced to the case for that of a uniformly elliptic equation where the H\"older continuity of the gradient has been shown in the fundamental work of De Giorgi \cite{DeG}. The techniques have been adopted for degenerated equations for the Euclidean case in \cite{Lady-Ural, Evans, Dib, Tolk}, etc. and for the case of the Heisenberg Group in \cite{Min-Z-Zhong, Zhong, Muk-Zhong}, etc. 

To this end, some notations are in order. For any $B_{\rho_0}\subset \R^N$, the De Giorgi's
class $DG^+(B_{\rho_0})$ consists of functions $w\in
\subsob^{1,2}(B_{\rho_0})\cap L^\infty(B_{\rho_0})$, which satisfy the inequality 
\begin{equation}\label{eq:DG}
\int_{B_{\rho'}}|\X (w-k)^+|^2\, dx\le
\frac{\gamma}{(\rho-\rho')^2}
\int_{B_\rho}|(w-k)^+|^2\, dx+\chi^2|
A_{k,\rho}^+(w)|^{1-\frac{2}{Q}+\epsilon}
\end{equation}
for some $\gamma, \chi,\epsilon>0$, 
where $ A_{k,\rho}^+(w)=\{ x\in B_\rho: (w-k)^+=\max (w-k,0)>0\}$ for any arbitrary $k\in\R$, the balls $B_{\rho'},B_\rho$ and 
$B_{\rho_0}$ are concentric with $0<\rho'<\rho\le \rho_0$. Also, $A_{k,\rho}^-(w)$ and the class $DG^-(B_{\rho_0})$ are similarly defined and $DG(B_{\rho_0})=
DG^+(B_{\rho_0})\cap DG^-(B_{\rho_0})$. All properties of classical 
De Giorgi class functions, also hold for these classes. 

In order to establish integral inequality of the type \eqref{eq:DG} for the gradient, we require the following lemma. The proof is similar to Lemma 4.1 of \cite{Muk-Zhong} but more involved. 
\begin{Lem}\label{lem:cacci:k}
Let $B_{r_0}\subset\Omega$ and $0<r<r_0/2$ be fixed. Suppose that there is $\tau>0$ such that 
 \begin{equation}\label{comble'}
 \vert \X u\vert\ge \tau \mu(r) \quad \text{in }\, A_{k,r}^+(X_l u),
 \end{equation}
 for an index $l\in\{1,\ldots,m\}$ and a constant $k\in \R$. Then for any $q\ge 4$ and any 
$0<r^{\prime\prime}<r^\prime\le r$, we have the following inequality, 
\begin{equation*}
\begin{aligned}
 \int_{B_{r^{\prime\prime}}} \weight | \X(X_lu-k)^{+}|^2\, dx
 \le  
 \frac{c}{(r^\prime-r^{\prime\prime})^2}\int_{B_{r^{\prime}}} \weight
 |(X_lu-k)^{+}|^2\, dx\,+\, cK| A^+_{k,r^\prime}(X_lu)|^{1-\frac{2} {q}},
\end{aligned}
\end{equation*}
where $K = r_0^{-2}|B_{r_0}|^{2/q}\mu(r_0)^2\F(\mu(r_0))$ 
and $c=c(N,p,L,q,\tau)>0$. 
\end{Lem}
\begin{proof}
Note that we can assume $|k|\leq \mu(r_0)$ without loss of generality, since otherwise all the terms vanish. 
Recalling \eqref{eq:yeqw2} with $Y=X_l$, we have 
\begin{equation}\label{eq:yeqw2xl}
\begin{aligned}
\sum_{i,j}\int_\Om D_jA_i(\X u)X_jX_lu X_i\varphi\dx &= \sum_{i,j}\int_\Om D_jA_i(\X u)[X_j,X_l]u X_i\varphi\dx \\
&\quad + \sum_{i}  \int_\Om A_i(\X u) [X_i, X_l]\varphi\dx.
\end{aligned}
\end{equation}
Let $\eta\in C^\infty_0(B_{r'})$ is a standard cutoff function such that $\eta = 1$ in $B_{r''}$ and $|\eta|+|\X\eta|\leq 2/(r'-r'')$, we choose 
 $\varphi = \eta^2 (X_lu -k)^+$ as a test function in 
 equation \eqref{eq:yeqw2xl} to get 
\begin{equation}\label{eqj}
\begin{aligned}
 \sum_{i,j}\int_{B_r}&\eta^2D_jA_i(\X u)X_jX_lu X_i((X_lu -k)^+)\, dx\\
&= -2\sum_{i,j}\int_{B_r}\eta (X_lu -k)^+ D_jA_i(\X u)X_jX_lu X_i \eta\, dx\\
&\quad + \sum_{i,j}\int_\Om D_jA_i(\X u)[X_j,X_l]u X_i(\eta^2 (X_lu -k)^+)\, dx\\
&\quad -\sum_{i,j}\int_{B_r} \eta^2 (X_lu -k)^+  D_jA_i(\X u) [X_i,X_l]X_ju\, dx\\
&= J_1+J_2+J_3,
\end{aligned}
\end{equation}
where integral by parts is performed on the last term of \eqref{eq:yeqw2xl}. 
From \eqref{eq:pstr}, note that 
\begin{equation}\label{eqj0}
\text{LHS of \eqref{eqj}} \geq \int_{B_r}\eta^2\, \weight | \X(X_lu-k)^{+}|^2\, dx, 
\end{equation}
and we now estimate each $J_i$ of the right hand side of \eqref{eqj}, one by one. 

From \eqref{eq:pstr} and Young's inequality, we have 
\begin{equation}\label{eqj1}
\begin{aligned}
|J_1|&\leq c\int_{B_r}|\eta| |(X_lu -k)^+| \weight |\XX u| |\X \eta|\, dx\\
&\leq \eps \int_{B_r}\eta^2\, \weight | \X(X_lu-k)^{+}|^2\, dx
+\frac{c}{\eps}\int_{B_r}|\X\eta|^2 \weight |(X_lu-k)^{+}|^2\, dx.
\end{aligned}
\end{equation}

Similarly, from \eqref{eq:pstr}, \eqref{eq:comm} and Young's inequality, we have 
\begin{equation}\label{eqj2}
\begin{aligned}
|J_2|&\leq c\int_\Om \weight \big(|\Xu|+|Tu|\big) \Big(\eta^2 | \X(X_lu-k)^{+}|+|\eta| |\X\eta| |(X_lu -k)^+)|\Big)\dx\\
&\leq \eps \int_{B_r}\eta^2 \weight | \X(X_lu-k)^{+}|^2\, dx
+\eps \int_{B_r}|\X\eta|^2 \weight |(X_lu-k)^{+}|^2\, dx\\
&\quad +\frac{c}{\eps} \int_{A^+_{k,r}}\eta^2 \weight \big(|\Xu|^2+|Tu|^2\big)\dx. 
\end{aligned}
\end{equation}
Then the last term of the above is estimated by H\"older's inequality, \eqref{eq:est Tu} and \eqref{comble'} as 
\begin{equation}\label{eqj2'}
 \begin{aligned}
 \int_{A^+_{k,r}}\eta^2 &\weight \big(|\Xu|^2+|Tu|^2\big)\dx\\
&\le \Big(\int_{B_{r_0/2}} \weight 
\big(|\Xu|^q+|Tu|^q\big)\,
dx\Big)^{\frac 2 q}\Big(\int_{A_{k,r}^+}\weight\,
dx\Big)^{1-\frac{2}{q}}\\
&\leq c\,r_0^{-2}\mu(r_0)^2\F(\mu(r_0))
|B_{r_0}|^\frac{2}{q}| A^+_{k,r}(X_lu)|^{1-\frac{2} {q}}
 \end{aligned}
\end{equation}
for some $c=c(N,p,L,q,\tau)>0$. 

Finally, for $J_3$, by \eqref{eq:pstr},\eqref{eq:comm} and H\"older's inequality, 
\begin{equation*}
 \begin{aligned}
 |J_3| &\leq \int_{B_r} \eta^2 \weight |(X_lu -k)^+| \big(|\XX u|+|\X T u| +|\X u|+|Tu|\big)\, dx\\
&\leq \Big(\int_{A_{k,r}^+} \eta^2 |(X_lu -k)^+|^2 \weight  \big(|\XX u|^2+|\X T u|^2 +|\X u|^2+|Tu|^2\big)\dx\Big)^\frac{1}{2}\Big(\int_{A_{k,r}^+} \weight dx\Big)^\frac{1}{2}
 \end{aligned}
\end{equation*}
The estimate of integrand with the terms $|\X u|^2+|Tu|^2$ of the above is same as \eqref{eqj2'}. However, estimating the first two terms of the above integrand is more delicate; the use of \eqref{eq:beta0XX},\eqref{eq:beta0XT} or any other estimate of Section \ref{sec:apest} right away would not be appropriate because the support $A_{k,r}^+$ would be lost that way, thereby producing terms larger than the required right hand side. 

To circumvent this, we require the following variants of the Caccioppoli type estimates that is similar to \cite{Zhong, Muk-Zhong}. 
For any $\kappa\geq 0$, we test \eqref{eq:Teq} with $\varphi_k=\eta^2 |(X_lu -k)^+|^2 |Tu|^\kappa T_ku $. Then, taking summation over $k\in\{1,\ldots,n\}$, we obtain 
\begin{equation}\label{ueq}
  \begin{aligned}
 \sum_{i,j,k}\int_\Omega &\eta^2  |(X_lu -k)^+|^2 |Tu|^\kappa D_jA_i(\X u)X_jT_k uX_iT_ku\, dx \\
+\kappa  &\sum_{i,j,k}\int_\Omega  \eta^2 | Tu|^{\kappa-1}T_k u  |(X_lu -k)^+|^2 D_jA_i(\X u)X_jT_k uX_i(|Tu|)\dx \\
   &\ = -2\sum_{i,j,k}\int_\Omega \eta\,  |(X_lu -k)^+|^2 |Tu|^\kappa T_k u D_jA_i(\X u)X_jT_k uX_i\eta\, dx\\
 &\qquad -2 \sum_{i,j,k}\int_\Omega  \eta^2 (X_lu -k)^+|Tu|^\kappa T_k u D_jA_i(\X u)X_jT_k uX_i((X_lu -k)^+)\dx \\
&\qquad + \sum_{i,j,k}  \int_\Om D_jA_i(\X u)[X_j,T_k]u X_i\varphi_k\dx \\
&\qquad + \sum_{i,j,k}    \int_\Om  \varphi_k D_jA_i(\X u) [X_i, T_k]X_j u\dx\\
&= U_1+U_2+U_3+U_4
  \end{aligned}
 \end{equation}
From the structure condition \eqref{eq:pstr}, we have 
\begin{equation}\label{ueq0}
\text{LHS of \eqref{ueq}} \geq  \int_{B_r}\eta^2\weight 
|(X_lu -k)^+|^2 |Tu|^\kappa |\X Tu|^2\dx. 
\end{equation}
The right hand side of \eqref{ueq} is estimated as follows. Let us denote 
\begin{equation}\label{Mcal}
 \begin{aligned}
  \mathcal{M} = \int_{B_r}\eta^2 \weight | \X(X_lu-k)^{+}|^2\, dx
  + \int_{B_r}(\eta^2+|\X\eta|^2)\weight |(X_lu-k)^{+}|^2\, dx. 
 \end{aligned} 
\end{equation}
Then, from \eqref{eq:pstr} and H\"older's inequality, we have 
\begin{equation}\label{ueq12}
\begin{aligned}
  |U_1|+|U_2| &\leq c \int_{B_r} |\eta| |(X_lu-k)^+|^2\weight
  |Tu|^{\kappa+1}|\X Tu||\X\eta|\dx\\
 &\quad + c \int_{B_r} \eta^2|(X_lu-k)^+|\weight
  |Tu|^{\kappa+1}|\X Tu||\X (X_lu-k)^+|\dx\\
&\leq c \,\mathcal{M}^\frac{1}{2} \Big(\int_{B_r}\eta^2
|(X_lu- k)^+|^2 \weight|Tu|^{2\kappa+2}|\X Tu|^2 dx\Big)^\frac{1}{2}.
 \end{aligned}
\end{equation}
To continue the estimates, note that
\begin{equation}\label{ueq3}
  \begin{aligned}
 U_3 &= \sum_{i,j,k}\int_\Omega \eta^2  |(X_lu -k)^+|^2 |Tu|^\kappa D_jA_i(\X u)[X_j,T_k]uX_iT_ku\, dx \\
&\qquad +\kappa  \sum_{i,j,k}\int_\Omega  \eta^2 | Tu|^{\kappa-1}T_k u  |(X_lu -k)^+|^2 D_jA_i(\X u)[X_j,T_k]uX_i(|Tu|)\dx \\
   &\qquad  +2\sum_{i,j,k}\int_\Omega \eta\,  |(X_lu -k)^+|^2 |Tu|^\kappa T_k u D_jA_i(\X u)[X_j,T_k]uX_i\eta\, dx\\
 &\qquad +2 \sum_{i,j,k}\int_\Omega  \eta^2 (X_lu -k)^+|Tu|^\kappa T_k u D_jA_i(\X u)[X_j,T_k]uX_i((X_lu -k)^+)\dx,
  \end{aligned}
 \end{equation}
which, from \eqref{eq:pstr}, \eqref{eq:comm} and H\"older's inequality leads to 
\begin{equation}\label{ueq3est}
  \begin{aligned}
 |U_3| &\leq c(\kappa+1)\int_\Omega \eta^2 \weight |(X_lu -k)^+|^2 |Tu|^\kappa \big(|\Xu|+|Tu| \big)|\X Tu|\, dx \\
   &\qquad  +c\int_\Omega |\eta| \weight |(X_lu -k)^+|^2 |Tu|^{\kappa+1} \big(|\Xu|+|Tu| \big)|\X\eta|\, dx\\
 &\qquad +c\int_\Omega  \eta^2 |(X_lu -k)^+| \weight |Tu|^{\kappa+1} \big(|\Xu|+|Tu| \big) |\X (X_lu -k)^+)|\dx\\
&\leq c(\kappa+1) 
\mathcal{M}^\frac{1}{2} \Big(\int_{B_r}\eta^2
|(X_lu- k)^+|^2 \weight |Tu|^{2\kappa}\big(|\Xu|^{2}+|Tu|^{2} \big) |\X Tu|^2 dx\Big)^\frac{1}{2}\\
&\qquad + c\,
\mathcal{M}^\frac{1}{2}  \Big(\int_{B_r}\eta^2
|(X_lu- k)^+|^2 \weight |Tu|^{2\kappa+2}\big(|\Xu|^{2}+|Tu|^{2} \big)dx\Big)^\frac{1}{2}
  \end{aligned}
 \end{equation}
Finally, from \eqref{eq:pstr}, \eqref{eq:XTX} and H\"older's inequality, we have 
\begin{equation*}
  \begin{aligned}
 |U_4| &\leq c\int_\Omega \eta^2 \weight |(X_lu -k)^+|^2 |Tu|^{\kappa+1} 
\big(|\XX u| +|\X Tu| +|\Xu|+|Tu| \big)\, dx \\
&\leq c\,\mathcal{M}^\frac{1}{2} 
 \Big(\int_{B_r}\eta^2
|(X_lu- k)^+|^2 \weight |Tu|^{2\kappa+2}\big(|\XX u|^2+|\X Tu|^2 + |\Xu|^{2}+|Tu|^{2} \big)dx\Big)^\frac{1}{2},
  \end{aligned}
 \end{equation*}
which combined with the previous estimates \eqref{ueq}, \eqref{ueq0}, \eqref{ueq12} and \eqref{ueq3est}, leads to 
\begin{equation}\label{parta1}
  \begin{aligned}
  \int_{B_r}\eta^2\weight 
|(X_lu -k)^+|^2 |Tu|^\kappa |\X Tu|^2\dx\leq  c(\kappa+1)\mathcal{M}^\frac{1}{2} a(2\kappa+2)^\frac{1}{2}
  \end{aligned}
 \end{equation}
where $c=c(N,p,L)>0$, $\mathcal M$ is as in \eqref{Mcal} and we denote
\begin{equation}\label{defa}
a(\kappa)=
\int_{B_r}\eta^2
|(X_lu- k)^+|^2 \weight 
\big(|Tu|^{\kappa}+|\X u|^{\kappa}\big) \big(|\XX u|^2+|\X Tu|^2 + |\Xu|^{2}+|Tu|^{2} \big)dx
\end{equation}
for any $\kappa\geq 0$. Notice that the left hand side of \eqref{parta1} contains a part of $a(\kappa)$. To get the other parts one requires similar estimates. From the symmetric structure of commutation relations \eqref{eq:comm}, it is not difficult to see that a similar estimate can be obtained as above by testing \eqref{eq:Teq} with 
$\varphi_k=\eta^2 |(X_lu -k)^+|^2 |\X u|^\kappa T_ku $, i.e. we can obtain
\begin{equation}\label{parta2}
  \begin{aligned}
  \int_{B_r}\eta^2\weight 
|(X_lu -k)^+|^2 |\X u|^\kappa |\X Tu|^2\dx\leq  c(\kappa+1)\mathcal{M}^\frac{1}{2} a(2\kappa+2)^\frac{1}{2}.
  \end{aligned}
 \end{equation}
Similarly, it is also not hard to see that by taking
$\varphi_l=\eta^2 |(X_lu -k)^+|^2 \big(|\X u|^\kappa +|Tu|^\kappa\big) X_lu $ 
as the test function in \eqref{eq:yeqwXl} and following the steps similarly as above, we can obtain 
\begin{equation}\label{parta34}
  \begin{aligned}
  \int_{B_r}\eta^2\weight 
|(X_lu -k)^+|^2 \big(|\X u|^\kappa +|Tu|^\kappa\big) |\XX u|^2\dx\leq  c(\kappa+1)\mathcal{M}^\frac{1}{2} a(2\kappa+2)^\frac{1}{2}.
  \end{aligned}
 \end{equation}
Finally, the following is easy to see from Young's and H\"older's inequality that 
\begin{equation}\label{partatrivial}
  \begin{aligned}
  \int_{B_r}\eta^2\weight 
|(X_lu -k)^+|^2 \big(|\X u|^\kappa +|Tu|^\kappa\big) \big(|\Xu|^{2}+|Tu|^{2} \big)\dx\leq  c\mathcal{M}^\frac{1}{2} a(2\kappa+2)^\frac{1}{2}.
  \end{aligned}
 \end{equation}
Thus, by adding \eqref{parta1},\eqref{parta2},\eqref{parta34} and \eqref{partatrivial}, we finally get 
\begin{equation}\label{ait}
  \begin{aligned}
  a(\kappa) \leq  c(\kappa+1)\mathcal{M}^\frac{1}{2} a(2\kappa+2)^\frac{1}{2},
  \end{aligned}
 \end{equation}
for any $\kappa\geq 0$, where $\mathcal M$ is as in \eqref{Mcal}. We iterate \eqref{ait} with the sequence
$ \kappa_j = 2^j-2$ for $j\in \N$ and letting $c= c(N,p,L,j)$ and $a_j =a(\kappa_j)$, so that we get 
\begin{equation}\label{it}
a_1 \,\leq\, (c\,\mathcal{M})^\frac{1}{2} a_2^\frac{1}{2}\,\leq\, \ldots
\,\leq (c\,\mathcal{M})^{(1-{1/2^j})}\,a_{j+1}^{1/2^j}
\end{equation} 
for every $j\in \N$. 
Now, for 
a large enough $j$ to be chosen later, 
we estimate using all the apriori estimates obtained finally in Section \ref{sec:apest}, i.e. \eqref{eq:est Xu}, \eqref{eq:est Tu}, \eqref{eq:est XTu}, \eqref{eq:est0 XTu}, to get
\begin{equation}\label{est a}
 a_{j+1}  \le \frac{c \mu(r_0)^2}{r_0^{\kappa_{j+1}+4}}\int_{B_{r_0}}
  \weight |\X u|^{\kappa_{j+1}+2} \, dx
\leq c\,r_0^{-(\kappa_{j+1}+4)}
 \F(\mu(r_0))
 \mu(r_0)^{\kappa_{j+1}+4}\, |B_{r_0}|. 
\end{equation}
Now we are ready to estimate $J_3$. Notice that, from \eqref{comble'} we have 
$$ |J_3 | \leq c\,a_1^{1/2} \Big(\int_{A_{k,r}^+} \weight dx\Big)^\frac{1}{2}
\leq c\,a_1^{1/2}\F(\mu(r_0))^{1/2}
  |A_{k,r}^+(X_lu)|^{1/2}$$ 
for some $c=c(N,p,L,\tau)>0$. Then, applying the iteration \eqref{it} and the estimate \eqref{est a},
\begin{equation*}
  \begin{aligned}
   |J_3| &\leq c\,\mathcal{M}^{\frac{1}{2}(1-{1/2^j})}
   \,a_{j+1}^{1/2^{j+1}}\F(\mu(r_0))^{1/2}
  |A_{k,r}^+(X_lu)|^\frac{1}{2}\\
  &\leq \frac{c}{r_0^{(1+{1/2^j})}}\mathcal{M}^{\frac{1}{2}(1-{1/2^j})}
  \F(\mu(r_0))^{\frac{1}{2}(1+{1/2^j})}
  \mu(r_0)^{(1+{1/2^j})}|B_{r_0}|^{1/2^{j+1}}|A_{k,r}^+(X_lu)|^\frac{1}{2}.
  \end{aligned}
 \end{equation*}
Then, by Young's inequality, we finally obtain
\begin{equation}\label{eqj3}
 |J_3| \leq \mathcal M/2 + c\,r_0^{-2}
\F(\mu(r_0))\mu(r_0)^2
|B_{r_0}|^{1/(2^j+1)}| A^+_{k,r}(X_lu)|^{2^j/(2^j+1)}
\end{equation}
for some $c = c(N,p,L,\tau,j)>0$. This, with  
$j = j(q)\in\N$ such that $2^j/(2^j+1)\,\geq\, 1-2/q$, gives us the required estimate. Combining \eqref{eqj1} and \eqref{eqj2},\eqref{eqj2'} with a small enough $\eps>0$ and \eqref{eqj3} with the choice of $j=j(q)$, we finally end up with 
\begin{align*}
\int_{B_r}\eta^2 \weight | \X(X_lu-k)^{+}|^2\, dx &\leq c\int_{B_r}(\eta^2+|\X\eta|^2)\weight |(X_lu-k)^{+}|^2\, dx \\
&\ +  c\,r_0^{-2}\mu(r_0)^2\F(\mu(r_0))
|B_{r_0}|^\frac{2}{q}| A^+_{k,r}(X_lu)|^{1-\frac{2} {q}}
\end{align*}
for some $c=c(N,p,L,q,\tau)>0$ and the proof is finished. 
\end{proof}
\begin{Rem}\label{rem:-version}
Some observations on Lemma \ref{lem:cacci:k} are in order:
\begin{enumerate}
\item We can obtain a similar inequality, corresponding to that of Lemma \ref{lem:cacci:k} 
 with $(X_lu-k)^+$ replaced by $(X_lu-k)^-$ and
 $A^+_{k,r}(X_lu)$ replaced by $A^-_{k,r}(X_lu)$.
\item It can be observed from the proof that Lemma \ref{lem:cacci:k} holds without the condition \eqref{comble'}, for the case of $p\geq 2$ and for this range, it is possible to prove the $C^{1,\alpha}$ regularity directly following the direction of the classical result of DiBenedetto \cite{Dib}, without using anything related to the truncation from the previous subsections. However, for the full range $1<p<\infty$,  
the application of Lemma \ref{lem:cacci:k} shall be made in the following whenever Corollary \ref{prop:case1} ensures the condition \eqref{comble'}. 
\end{enumerate}
\end{Rem}
Le us fix $B_{r_0}\subset \Om$. For any $0<r<r_0$, similarly as \eqref{def:mu}, let us denote 
\[ \omega_l(r)=\osc_{B_r} X_lu, \quad \omega(r)=\max_{1\le l\le m}\omega_l(r),\]
and it is clear that we have $\omega(r)\le 2\mu(r)$. 
The following oscillation Lemma is a consequence of Lemma \ref{lem:cacci:k} and Corollary \ref{prop:case1}. Its proof is the same as \cite[Theorem 4.3]{Muk-Zhong} and quite similar to those in \cite{Dib, Tolk}, etc. Nevertheless, we provide a brief outline for sake of completeness. 
\begin{Lem}\label{lem:hold}
There exists a constant $s=s(N,p,L)\ge 0$ such that
for every $0<r\leq r_0/16$, we have the following, 
\begin{equation}\label{mur4r}
\omega(r) \le (1-2^{-s})\omega(8r) + 2^s\big(\delta +\mu(r_0)^2\big)^\frac{1}{ 2}\left(\frac{r}{r_0}\right)^\alpha,
\end{equation}
where $\alpha = 1/2$ when $1<p<2$ and $\alpha = 1/p$ when $p\geq 2$. 
\end{Lem}
\begin{proof}
 Letting $\alpha$ be as given, note that we may assume 
\begin{equation}\label{assum:mu}
\omega(r) \ge \big(\delta +\mu(r_0)^2\big)^\frac{1}{ 2}\left(\frac{r}{r_0}\right)^{\alpha},
\end{equation}
since, otherwise, \eqref{mur4r} is true with $s=0$. 
In the following, we assume that (\ref{assum:mu}) is true we divide the 
proof into two alternative cases, similarly as in \cite{Dib,Tolk, Muk-Zhong}, etc.

{\it Case 1}. For at least one index $l\in\{1,\ldots,m\}$, we have either
\begin{equation}\label{small1}
  |\{x\in B_{4r} : X_lu<\mu(4r)/4\}|
 \leq \theta |B_{4r}|\quad \text{or}\quad |\{x\in B_{4r}: X_lu>-\mu(4r)/4\}|
  \leq \theta |B_{4r}|, 
 \end{equation}
where $\theta = \theta(N,p,L)>0$ is the constant in Corollary \ref{prop:case1}. Then, from Corollary \ref{prop:case1}, 
we have
\begin{equation}\label{comparable2}
  \vert\X u\vert\ge  3\mu(2r)/16\quad \text{in }\ B_{2r},
 \end{equation}
which allows us to use Lemma \ref{lem:cacci:k} with $q = 2Q$ to obtain 
\begin{equation}\label{DG estimate}
\begin{aligned}
\int_{B_{r''}} | \X (X_iu-k)^+|^2\, dx
\,\le\, & \frac{c}{(r'-r'')^2}\int_{B_{r'}}|(X_iu-k)^+|^2\,dx \\
&+  cK \F(\mu(2r))^{-1}| A^+_{k,r'}(X_iu)|^{1-\frac{1} {Q}}
\end{aligned}
\end{equation}
for any $0<r''<r'\leq 2r,\ i\in\{1,\ldots,2n\}$ and all $k\in \R$,
where $K = r_0^{-2}
|B_{r_0}|^{1/Q}\mu(r_0)^2\F(\mu(r_0))$. This implies that $X_iu$ belongs to the De Giorgi class $DG^+(B_{2r})$ for each $i\in \{1,\ldots,m\}$ and hence, it satisfies an oscillation lemma, see \cite{DeG, Lady-Ural, Giu}, etc. from which, as in \cite{Muk-Zhong}, it is not hard to obtain \eqref{mur4r}. 

{\it Case 2}. If Case 1 does not happen, then for every $i\in\{1,\ldots,m\}$, we have   
\begin{equation}\label{case2 1}
  |\{x\in B_{4r} : X_iu<\mu(4r)/4\}|
  > \theta |B_{4r}| \quad\text{and}\quad |\{x\in B_{4r}: X_iu>-\mu(4r)/4\}|
 >\theta |B_{4r}|.
 \end{equation} 
Then, note that, on the set 
$\{x\in B_{8r}: X_iu >\mu(8r)/4 \}$, we trivially have
 \begin{equation}\label{comparable4}
 \vert \X u\vert\ge \mu(8r)/4\quad \text{in } A^+_{k,8r}(X_iu)
 \end{equation}
 for all $k\geq \mu(8r)/4$. 
Thus, we can apply Lemma \ref{lem:cacci:k} with $q = 2Q$ to conclude that 
\begin{equation}\label{DG estimate 2}
\begin{aligned}
\int_{B_{r''}} | \X (X_iu-k)^+|^2\, dx
\,\le\, &\frac{c}{(r'-r'')^2}\int_{B_{r'}}|(X_iu-k)^+|^2\,dx \\
& + c K\,\F(\mu(8r))^{-1}| A^+_{k,r'}(X_iu)|^{1-\frac{1} {Q}}
\end{aligned}
\end{equation}
where $K = r_0^{-2}
|B_{r_0}|^{1/Q}\mu(r_0)^2\F(\mu(r_0))$, 
whenever $k\geq k_0= \mu(8r)/4 $ and $0<r^{\prime\prime}<r^\prime\le 8r$. 
Then, we have standard estimates for the supremum and infimum of $X_iu$, see \cite{Giu,Lady-Ural}, etc. 
such that for some $s_1 = s_1(n,p,L)>0$ he following holds, 
\begin{equation}\label{sup est}
 \sup_{B_{2r}} X_iu \le\sup_{B_{8r}} X_i u -2^{-s_1}\big(\sup_{B_{8r}}
X_iu-\mu(8r)/4\big)+ cK^\frac{1}{2}\F(\mu(8r))^{-1/2}r^\frac{1}{2}.
\end{equation}
and from the second part of Case 2 and Remark \ref{rem:-version}, we similarly have
\begin{equation}\label{inf est}
 \inf_{B_{2r}} X_iu \ge \inf_{B_{8r}} X_iu +2^{-s_1}\big(-\inf_{B_{8r}}
X_iu-\mu(8r)/4\big) - c K^\frac{1}{2}\F(\mu(8r))^{-1/2}r^\frac{1}{2}.
\end{equation}
Combining \eqref{sup est} and \eqref{inf est} as in \cite{Muk-Zhong}, one can also get \eqref{mur4r} and the proof is finished. 
\end{proof}

From \eqref{mur4r} of Lemma \ref{lem:hold}, it is easy to prove Theorem \ref{thm:mainthm2} for $\delta>0$ by a standard iteration argument which is classical. We refer to \cite{Giu, Lady-Ural, Gil-Tru}, etc. for details. Since all the constants in the above estimates are independent of $\delta$, hence from the arguments of Remark \ref{rem:del0}, we can obtain \eqref{eq:locboundholder} also for $\delta=0$. 
Thus, the proof of Theorem \ref{thm:mainthm2} is complete.  

\subsection{Concluding Remarks}
Here are some further remarks about the method and techniques used above throughout the paper and some discussion on further generalizations. 

\begin{enumerate}
\item All the apriori estimates for the equation \eqref{eq:maineq} in Section \ref{sec:apest}, rely only on the structure conditions \eqref{eq:str} and are independent of the ``$p$'' of the $p$-Laplacian. The other arguments in the later sections can also be appropriately generalized for the regularity theory of more general classes of equations; e.g. the equation 
$\dvh \big( f(|\X u|) \X u \big)=0$ with $f$ as any doubling differentiable function on $(0,\infty)$ so that $t\mapsto tf(t)$ is monotonic, admits weak solutions in sub-elliptic Orlicz-Sobolev spaces (defined similarly as Definition \eqref{def:subesob}) which can be similarly shown to be locally $C^{1,\alpha}$ whenever $1+|\X u|f'(|\X u|)/f(|\X u|)$ is positive and finite. We refer to \cite{Muk0, Muk1} for the case of the Heisenberg group.  \\

\item In the special case when the vector fields are left invariant with respect to a Carnot group of step $2$ 
(in particular, the Heisenberg group as in \cite{Zhong, Muk-Zhong}), the commutation relations are more special than \eqref{eq:comm} as mentioned earlier in Remark \ref{rem:carnot2}, up to the choice of exponential coordinates. This is why, in that case the estimates are simpler and some of the items in the integral estimates gets suppressed, see Remark \ref{rem:groupcase2}. 
But, unlike the previous case for the Heisenberg group by Mukherjee-Zhong \cite{Muk-Zhong}, here  
the lengths of the main crucial lemmas like Lemma \ref{lem:tech} and Lemma \ref{lem:cacci:k} are drastically extended and Corollary \ref{cor:est0 XTu}, Corollary \ref{cor:tech}, etc. have to be added. Also, in \cite{Zhong}, the approximation has been accomplished by a Hilbert-Haar theory for the case of the Heisenberg group, which depends on the group law and hence does not have an easy analogue for the case of general H\"ormander vector fields (some deeper theorems and newer adaptations are needed for it). Therefore, in this paper, we have used the Riemannian approximation which is more natural and easier to use.\\

\item The generalization of the above techniques from step $2$ to any higher step does not seem to be immediately available. The inequalities \eqref{eq:XTX} and \eqref{eq:XXX} in their present form, do not lead to anything useful for the case of step $3$ or any higher step. The higher order commutators have more intricate strucures. For instance, letting $X^{(s)}_1,\ldots, X^{(s)}_k$ as vector fields corresponding 
to $s$-th order commutators, the nature of the intermediate integral estimates for $\X^{(s)} u$ is unclear. In case of a Carnot group of step $r$, i.e. $1\leq s\leq r$, 
notice that $|\X \X^{(s)} u| \neq |\X^{(s)}\X u|$ except for the cases of $s=r$ (as $X_iX^{(r)}_j=X^{(r)}_j X_i$ for the last step of the Lie algebra) and $s=1$ (from the definition of the norm even though $X_iX_j\neq X_jX_i$). This feature leads to serious difficulties in the integral estimates for the intermediary steps that do not occur in a set up where these two are the only possible alternatives, i.e. the step $2$ case with $s=1,2$. Similar features typical to vector fields of step $2$ Carnot groups are also imbibed in more general step $2$ vector fields. Lastly, we refer to \cite{Dom-Man--hh} for some results on certain special cases of vector fields in higher steps. 

\end{enumerate}

\section*{Acknowledgement} 
The authors have been supported by European Union’s Horizon 2020 research and innovation programme under the 
(GHAIA) Marie Sk\l{}odowska-Curie grant agreement No. 777822. The authors are thankful to the anonymous referee and to Juan Manfredi for a careful reading of the paper and providing valuable suggestions.


\bibliographystyle{plain}
\bibliography{MyBib}

\end{document}